\documentclass[]{article}

\usepackage[left=2cm,right=2cm, top=2.4cm,bottom=2.4cm,bindingoffset=0cm]{geometry}
\usepackage{tikz, titlesec, tikz-cd}
\usetikzlibrary{calc,intersections,through,backgrounds}
\usepackage{amsmath, amssymb}
\usepackage{amsthm}

    \usepackage{dynkin-diagrams}

\usepackage{cases}

\usepackage{amsthm}
\let\oldproofname=\proofname
\renewcommand{\proofname}{\rm\bf{\oldproofname}}
\usepackage{hyperref}

\usepackage[numbers, sort&compress]{natbib}

\usepackage{caption}
\usepackage{subcaption}

\usepackage{yhmath}
\usepackage{mathdots, mathtools}
\usepackage{MnSymbol}

\usepackage[width=0.75\textwidth]{caption}

\usetikzlibrary{positioning, bending, arrows.meta}
\usetikzlibrary{decorations.text}
\usetikzlibrary{decorations.pathmorphing}
\usetikzlibrary{decorations.markings}

\tikzset{>=latex}
  \tikzset{->-/.style={decoration={
  markings,
  mark=at position .5 with {\arrow{>}}},postaction={decorate}}}
    \tikzset{-<-/.style={decoration={
  markings,
  mark=at position .5 with {\arrow{<}}},postaction={decorate}}}

    \tikzset{->/.style={decoration={
  markings,
  mark=at position .8 with {\arrow{>}}},postaction={decorate}}}

\newtheorem{lemma}{Lemma}[section]
\newtheorem{proposition}[lemma]{Proposition}
\newtheorem{theorem}[lemma]{Theorem}
\newtheorem{corollary}[lemma]{Corollary}

\theoremstyle{definition}
\newtheorem{remark}[lemma]{Remark}
\newtheorem{definition}[lemma]{Definition}

\newtheorem{example}[lemma]{Example}

\newcommand{\GL}[1]{{GL_#1}}
\newcommand{\gl}{{\mathfrak{gl}}}
\newcommand{\PGL}{{PGL}}
\newcommand{\Gr}{{Gr}}
\newcommand{\Orth}{{O}}
\newcommand{\Sp}{{Sp}}
\newcommand{\SO}{{SO}}
\newcommand{\PSp}{{PSp}}
\newcommand{\Z}{\mathbb{Z}}
\newcommand{\C}{\mathbb{C}}
\newcommand{\R}{\mathbb{R}}
\newcommand{\Q}{\mathcal{Q}}

\newcommand{\Hom}{{H}}
\newcommand{\Id}{  I}
\newcommand{\wind}{{{wind}}}
\newcommand{\wt}{{{wt}}}
\newcommand{\eps}{{\varepsilon}}
\newcommand{\Rect}{\mathcal{G}}
\newcommand{\mes}{{Meas}}
\newcommand{\g}{{\mathfrak g}}
\newcommand{\F}{{\mathcal F}}
\newcommand{\op}{{ \bar \Gamma}}
\newcommand{\tr}{\mathrm{Tr}}
\newcommand{\grad}{\mathrm{grad}\,}
\usepackage{bbm}
\newcommand{\lw}{w_{0}}

\newcommand{\I}{\mathcal I}
\newcommand{\word}{\mathbf  i}

\titleformat{\section}
  {\normalfont\large\bfseries}{\thesection}{1em}{}
  
  \makeatletter
\newsavebox\myboxA
\newsavebox\myboxB
\newlength\mylenA

\newcommand*\xoverline[2][0.75]{%
    \sbox{\myboxA}{$\m@th#2$}%
    \setbox\myboxB\null
    \ht\myboxB=\ht\myboxA%
    \dp\myboxB=\dp\myboxA%
    \wd\myboxB=#1\wd\myboxA
    \sbox\myboxB{$\m@th\overline{\copy\myboxB}$}
    \setlength\mylenA{\the\wd\myboxA}
    \addtolength\mylenA{-\the\wd\myboxB}%
    \ifdim\wd\myboxB<\wd\myboxA%
       \rlap{\hskip 0.5\mylenA\usebox\myboxB}{\usebox\myboxA}%
    \else
        \hskip -0.5\mylenA\rlap{\usebox\myboxA}{\hskip 0.5\mylenA\usebox\myboxB}%
    \fi}
\makeatother

\begin{document}

\tikzset{->-/.style={decoration={
  markings,
  mark=at position .5 with {\arrow{>}}},postaction={decorate}}}
  
%
  \tikzset{-<-/.style={decoration={
  markings,
  mark=at position .3 with {\arrow{<}}},postaction={decorate}}}
  
\usetikzlibrary{angles, quotes}



\title{
Planar networks and simple Lie groups beyond type A}

\author{Anton Izosimov\thanks{
Department of Mathematics,
University of Arizona and School of Mathematics \& Statistics, University of Glasgow;
e-mail: {\tt Anton.Izosimov@glasgow.ac.uk}
} }

\date{}

\maketitle

\abstract{

The general linear group $\GL{n}$, along with its adjoint simple group $\PGL_n$, 
can be described by means of weighted planar networks. 
 In this paper, we give a network description 
 for simple Lie groups of types $B$ and~$C$. The corresponding networks are axially symmetric modulo a sequence of cluster mutations along the axis of symmetry. We extend to this setting the result of Gekhtman, Shapiro, and Vainshtein on the Poisson property of Postnikov's boundary measurement map. We also show that $B$ and~$C$ type networks with positive weights parametrize the totally nonnegative part of the respective group. Finally, we construct network parametrizations of double Bruhat cells in symplectic and odd-dimensional orthogonal groups, and identify the corresponding face weights with Fock-Goncharov cluster coordinates.


%

}

\tableofcontents

\section{Introduction}

\subsection{Overview}

Description of matrices by planar graphs, or \emph{networks}, has its roots in the theory of total positivity and  is widely used in contemporary literature in the context of cluster algebras and Poisson geometry~\cite{fomin2000total, gekhtman2010cluster}. 
The modern definition of a matrix corresponding to a weighted network -- the \emph{boundary measurement matrix} -- is due to Postnikov~\cite{Pos}. The entries of that matrix are certain sign-twisted path generating functions associated with the network. Alternatively, the boundary measurement matrix may be defined in the framework of the dimer model on bipartite graphs, cf. \cite{postnikov2009matching, lam2013notes}. 

The boundary measurement map sending a network to the associated matrix has many remarkable features. First, that map is a homomorphism in the sense that the boundary measurement matrix of concatenation of two networks is the product of their boundary measurement matrices. Second, that map is Poisson:  there is a simple log-canonical Poisson bracket on networks whose pushforward by the  boundary measurement map coincides with the standard multiplicative Poisson structure on $GL_n$~\cite{GSV2}. Third, the boundary measurement map is \emph{positive} in the sense that it takes networks with positive weights to totally nonnegative matrices. The latter two features - the Poisson property and positivity - can be both understood in the framework of cluster algebras. The network structure is, in a certain sense, dual to the cluster structure on $GL_n$. All of this also applies, with some modifications, to the corresponding simple group $\PGL_n$.

Simple Lie groups of types other than $A$ also carry a cluster structure \cite{berenstein2005cluster, fock2006cluster}. However, it has not previously been known whether such groups can be described in the framework of planar networks. The goal of the present paper is to develop network models for simple Lie groups of types $B$ and $C$. To that end we introduce a class of networks that we call \emph{move-symmetric}. Such networks have a genuine axial symmetry in type $C$ and ``twisted'' symmetry in type~$B$. The twist is given by a sequence of network \emph{moves}, or \emph{cluster mutations}. 
We show that move-symmetric networks encode symplectic and odd-dimensional orthogonal groups in the same way as usual networks describe~$\GL{n}$. Specifically, the boundary measurement matrix of a move-symmetric network belongs to  $\Orth_{2n+1}$ or $\Sp_{2n}$, and is, moreover, nonnegative in the sense of Lusztig \cite{Lus} when the weights are positive real numbers. 
Furthermore, there is a  log-canonical Poisson bracket on move-symmetric networks  whose pushforward by the boundary measurement map is the standard Poisson structure on the corresponding group, see Theorem~\ref{thm1}. 



As an illustration of our construction, we describe a class of 
move-symmetric networks that parametrize \emph{double Bruhat cells} in  $\Orth_{2n+1}$ and $\Sp_{2n}$, see Proposition \ref{dbc}. 
Such networks are associated with \emph{double reduced words} in the corresponding Weyl group and can be thought of as graphical counterparts of parametrization of double Bruhat cells by means of \emph{factorization coordinates}  \cite{fomin1999double}. We also show that (appropriately symmetrized) \emph{face weights} of such networks coincide with Fock-Goncharov cluster coordinates~\cite{fock2006cluster} on double Bruhat cells in the corresponding centerless groups $\SO_{2n+1}$, $\PSp_{2n}$, see Proposition \ref{prop:fg}.

Obtaining network descriptions for simple Lie groups of type $D$ as well as exceptional types remains an open problem. An extension of our results to Grassmannians and  loop groups, as well as applications to integrable systems, cf. \cite{GSV2, GSV3, GSTV}, will be discussed in a separate publication.

Some ingredients of our construction in type $C$ case have already appeared in the literature. Namely, one can deduce that the boundary measurement matrix of an axially symmetric network is symplectic from the results of~\cite{karpman2018total}, while a network parametrization of double Bruhat cells in the symplectic group is given in~\cite{lin}. The latter work also treats types $B$ and $D$, but the corresponding networks lack planarity and hence good Poisson and positivity properties. Other work relating simple Lie groups and their homogeneous spaces to network-like objects includes \cite{chepuri2021electrical, bychkov2023electrical, lam2015electrical} where the authors study symplectic groups and Lagrangian Grassmannians in the framework of electrical networks, and \cite{galashin2020ising} on total positivity in orthogonal Grassmannians in the setting of the Ising model. Finally, we mention the work \cite{le2019cluster} on cluster structures on higher Teichm\"uller spaces for classical groups in types other than $A$. In types $B$ and $C$, the corresponding quivers have properties similar to that of our move-symmetric networks.   




\medskip
{\bf Acknowledgments.} The author is grateful to Jonathan Boretsky, Christopher Eur, Michael Gekhtman, George Lusztig, Alexander Shapiro, and the anonymous referees for fruitful discussions and useful remarks. A part
of this work was done during the author’s visit to Max Planck Institute for Mathematics
in Bonn. The author would like to thank the Institute’s faculty and staff for their support
and stimulating atmosphere. This work was partially supported by NSF grant DMS-2008021 and the Simons Foundation
through its Travel Support for Mathematicians program.



\medskip
\subsection{Networks and plabic graphs in type A}
\paragraph{Networks and matrices.} In this and next section we follow \cite{Pos}, sometimes adapting its terminology to better suit our purposes. 
In what follows, an $n\times n$ \emph{network} is an edge-weighted finite connected directed graph embedded in a rectangle, with $n$ sources on the left  side of the rectangle, $n$ sinks on the right side, and all other vertices in the interior of the rectangle. We label sources and sinks by integers $1, \dots, n$, from bottom up.
The edge weights are non-zero complex numbers  (more generally one can allow weights in any field). 

The \emph{boundary measurement matrix} of an $n \times n$ network is the $n \times n$ matrix given by
\begin{equation}\label{eq:bmm}
A_{ij}:= \sum_{\gamma: i \to j} (-1)^{\wind(\gamma)}\wt(\gamma).
\end{equation}
Here the sum is taken over all directed paths $\gamma$ from source $i$ to sink $j$. The \emph{weight $\wt(\gamma)$} of a path $\gamma$ is the product of weights of all edges belonging to $\gamma$. The number $\wind(\gamma) \in \Z$ is called the \emph{winding index} and can be thought of as the number of $360^\circ$ turns made by $\gamma$, see \cite[Section 4]{Pos}. In the presence of directed cycles the sum defining $A_{ij}$ is an infinite geometric-like series which is understood as evaluation of the rational function it formally adds up to, see Figure~\ref{fig:network}.
 \begin{figure}[t]
 \centering
\begin{tikzpicture}[scale = 0.8]
\node [draw,circle,color=black, fill=black,inner sep=0pt,minimum size=1pt] (A) at (0,1) {};
\node [draw,circle,color=black, fill=black,inner sep=0pt,minimum size=1pt](B) at (1,1) {};
\node [draw,circle,color=black, fill=black,inner sep=0pt,minimum size=1pt] (C) at (0,2) {};
\node [draw,circle,color=black, fill=black,inner sep=0pt,minimum size=1pt] (D) at (1,2) {};
\draw [->] (-1, 1) -- (A)   node[midway, below] {\footnotesize $ f$};
\draw [->-] (B) -- (A) node[midway, below] {\footnotesize $ g$};
\draw [->-] (A) -- (C) node[midway, left] {\footnotesize $ d$};
\draw [->-] (C) -- (D) node[midway, above] {\footnotesize $ b$};
\draw [->] (-1, 2) -- (C) node[midway, above] {\footnotesize $ a$};
\draw [->-] (B) -- (2,1) node[midway, below] {\footnotesize $ h$};
\draw [->-] (D) -- (2,2) node[midway, above] {\footnotesize $ c$};
\draw [->-] (D) -- (B) node[midway, left] {\footnotesize $ e$};

\draw [dashed] (-1,2.75) -- (2,2.75);
\draw [dashed] (2,0.25) -- (2,2.75) node[pos = 0.3, right] {\footnotesize$1$} node[pos = 0.7, right] {\footnotesize$2$};;
\draw [ dashed] (-1,0.25) -- (2,0.25) ;
\draw [ dashed]  (-1,0.25) -- (-1,2.75)   node[pos = 0.3, left] {\footnotesize$1$} node[pos = 0.7, left] {\footnotesize$2$};;

\node () at (9, 1.5) { $A_{11} = fdbeh - f(dbeg)dbeh + f(dbeg)^2dbeh \,- \dots = \displaystyle\frac{fdbeh}{1 + dbeg}.$};
\end{tikzpicture}
\caption{A  network and the (1,1) entry of its boundary measurement matrix.}\label{fig:network}
\end{figure}
 For real positive weights the entries of the boundary measurement matrix are always finite, and the matrix itself is \emph{totally nonnegative}, i.e., has all minors $\geq 0$.

There is a natural \emph{concatenation} operation on networks defined by gluing sinks of one network to the corresponding sources of the other, see \cite[Section 3.1]{GSV2} and Figure \ref{fig:conc}.
 \begin{figure}
 \centering
\begin{tikzpicture}[scale = 0.9]
\node () at (0,0) {
\begin{tikzpicture}[scale = 0.8]
\node [draw,circle,color=black, fill=black,inner sep=0pt,minimum size=1pt] (A) at (0,1) {};
\node [draw,circle,color=black, fill=black,inner sep=0pt,minimum size=1pt] (C) at (0,2) {};
\node [inner sep=0pt]  (B) at (1,1) {};
\node [inner sep=0pt]  (D) at (1,2) {};
\draw [->] (-1, 1) -- (A)   node[midway, below] {\footnotesize $ d$};
\draw [-<-] (B) -- (A) node[midway, below] {\footnotesize $ e$};
\draw [->-] (A) -- (C) node[midway, left] {\footnotesize $ c$};
\draw [->-] (C) -- (D) node[midway, above] {\footnotesize $ b$};
\draw [->] (-1, 2) -- (C) node[midway, above] {\footnotesize $ a$};

\draw [dashed] (-1,2.75) -- (1,2.75);
\draw [dashed] (1,0.25) -- (1,2.75) node[pos = 0.3, right] {\footnotesize$1$} node[pos = 0.7, right] {\footnotesize$2$};;
\draw [ dashed] (-1,0.25) -- (1,0.25) ;
\draw [ dashed]  (-1,0.25) -- (-1,2.75)   node[pos = 0.3, left] {\footnotesize$1$} node[pos = 0.7, left] {\footnotesize$2$};;

\end{tikzpicture}};
\node () at (2,0) {$\times$};
\node () at (4,0) {
\begin{tikzpicture}[scale = 0.8]
\node [inner sep=0pt]  (B) at (1,1) {};
\node [draw,circle,color=black, fill=black,inner sep=0pt,minimum size=1pt] (A) at (0,1) {};
\node [draw,circle,color=black, fill=black,inner sep=0pt,minimum size=1pt] (C) at (0,2) {};
\node  [inner sep=0pt]  (D) at (1,2) {};
\draw [->] (-1, 1) -- (A)   node[midway, below] {\footnotesize $ i$};
\draw [-<-] (B) -- (A) node[midway, below] {\footnotesize $ j$};
\draw [-<-] (A) -- (C) node[midway, right] {\footnotesize $ h$};
\draw [->-] (C) -- (D) node[midway, above] {\footnotesize $ g$};
\draw [->] (-1, 2) -- (C) node[midway, above] {\footnotesize $ f$};

\draw [dashed] (-1,2.75) -- (1,2.75);
\draw [dashed] (1,0.25) -- (1,2.75) node[pos = 0.3, right] {\footnotesize$1$} node[pos = 0.7, right] {\footnotesize$2$};;
\draw [ dashed] (-1,0.25) -- (1,0.25) ;
\draw [ dashed]  (-1,0.25) -- (-1,2.75)   node[pos = 0.3, left] {\footnotesize$1$} node[pos = 0.7, left] {\footnotesize$2$};;

\end{tikzpicture}};
\node () at (6,0) {$=$};
\node () at (8.5,0) {
\begin{tikzpicture}[scale = 0.8]
\node [draw,circle,color=black, fill=black,inner sep=0pt,minimum size=1pt] (A) at (0,1) {};
\node [draw,circle,color=black, fill=black,inner sep=0pt,minimum size=1pt] (B) at (1,1) {};
\node [draw,circle,color=black, fill=black,inner sep=0pt,minimum size=1pt] (C) at (0,2) {};
\node [draw,circle,color=black, fill=black,inner sep=0pt,minimum size=1pt] (D) at (1,2) {};
\draw [->] (-1, 1) -- (A)   node[midway, below] {\footnotesize $ d$};
\draw [-<-] (B) -- (A) node[midway, below] {\footnotesize $ ei$};
\draw [->-] (A) -- (C) node[midway, left] {\footnotesize $ c$};
\draw [->-] (C) -- (D) node[midway, above] {\footnotesize $ bf$};
\draw [->] (-1, 2) -- (C) node[midway, above] {\footnotesize $ a$};
\draw [->-] (B) -- (2,1) node[midway, below] {\footnotesize $ j$};
\draw [->-] (D) -- (2,2) node[midway, above] {\footnotesize $ g$};
\draw [->-] (D) -- (B) node[midway, right] {\footnotesize $ h$};

\draw [dashed] (-1,2.75) -- (2,2.75);
\draw [dashed] (2,0.25) -- (2,2.75) node[pos = 0.3, right] {\footnotesize$1$} node[pos = 0.7, right] {\footnotesize$2$};;
\draw [ dashed] (-1,0.25) -- (2,0.25) ;
\draw [ dashed]  (-1,0.25) -- (-1,2.75)   node[pos = 0.3, left] {\footnotesize$1$} node[pos = 0.7, left] {\footnotesize$2$};;

\end{tikzpicture}

};
\end{tikzpicture}
\caption{Concatenation of networks.}\label{fig:conc}
\end{figure}
 The boundary measurement matrix of concatenation of two networks is the product of their boundary measurement matrices.

A  \emph{gauge transformation} of a network is multiplication of weights of all edges adjacent to a given internal vertex by either $\lambda$ or $\lambda^{-1}$ (where $\lambda \in \C^*$), depending on whether the edge is pointing towards or away from the vertex. This operation does not change the boundary measurement matrix, allowing one to express the latter in terms of coordinates on the quotient of the space of edge weights by gauge transformations. As such coordinates, one can take the \emph{face weights}. A \emph{face} of a network is a connected component of its complement in the rectangle. The boundary of every face $F$ consists of a sequence of edges $e_1, \dots, e_k$ and, possibly, a piece of the boundary of the rectangle. Define the weight of $F$ as the product $\wt(e_1)^{\eps_1} \cdots \wt(e_k)^{\eps_k}$ where $\eps_i $ is $1$ if the orientation of $e_i$ agrees with the counter-clockwise orientation of $F$, and  $-1$ otherwise; see Figure \ref{fig:face}.
 Face weights, subject to the condition that their product over all faces is equal to $1$, determine the edge weights uniquely up to gauge transformations.  The entries of the boundary measurement matrix are rational functions of face weights. Furthermore, for real positive face weights those entries are always finite, and the matrix itself is totally nonnegative.

  \begin{figure}
 \centering
\begin{tikzpicture}[scale = 1.1]
\node [draw,circle,color=black, fill=black,inner sep=0pt,minimum size=1pt] (A) at (0,1) {};
\node [draw,circle,color=black, fill=black,inner sep=0pt,minimum size=1pt] (B) at (1,1) {};
\node [draw,circle,color=black, fill=black,inner sep=0pt,minimum size=1pt] (C) at (0,2) {};
\node [draw,circle,color=black, fill=black,inner sep=0pt,minimum size=1pt] (D) at (1,2) {};
\node  () at (0.5,2.5){ \footnotesize $y_1$};
\node () at (-0.6,1.5){ \footnotesize $y_2$};
\node () at (0.5,1.5) {\footnotesize $y_3$};
\node () at (1.6,1.5) {\footnotesize $y_4$};
\node () at (0.5,0.5){ \footnotesize $y_5$};
\draw [->] (-1, 1) -- (A)   node[midway, below, opacity = 0.3] {\footnotesize $ f$};
\draw [->-] (B) -- (A) node[midway, below, opacity = 0.3] {\footnotesize $ g$};
\draw [->-] (A) -- (C) node[midway, left, opacity = 0.3] {\footnotesize $ d$};
\draw [->-] (C) -- (D) node[midway, above, opacity = 0.3] {\footnotesize $ b$};
\draw [->] (-1, 2) -- (C) node[midway, above, opacity = 0.3] {\footnotesize $ a$};
\draw [->-] (B) -- (2,1) node[midway, below, opacity = 0.3] {\footnotesize $ h$};
\draw [->-] (D) -- (2,2) node[midway, above, opacity = 0.3] {\footnotesize $ c$};
\draw [->-] (D) -- (B) node[midway, right, opacity = 0.3] {\footnotesize $ e$};

\draw [dashed] (-1,2.75) -- (2,2.75);
\draw [dashed] (2,0.25) -- (2,2.75); 
\draw [ dashed] (-1,0.25) -- (2,0.25) ;
\draw [ dashed]  (-1,0.25) -- (-1,2.75) ;

\node () at (7, 1.5) { $\begin{gathered}y_1 = abc, \quad y_2 = fda^{-1}, \quad y_3 = (dbeg)^{-1},\\ y_4 = c^{-1}eh,\quad y_5 = f^{-1}gh^{-1}. \\   y_1y_2y_3y_4y_5 = 1. 
\end{gathered}$};
\end{tikzpicture}
\caption{Definition of face weights.}\label{fig:face}
\end{figure}

\paragraph{Perfect networks and plabic graphs.} 
A network is called \emph{perfect} if all its sources and sinks are univalent, while any interior vertex has either exactly one incoming edge, in which case the vertex is called \emph{white}, or exactly one outgoing edge, in which case the vertex is \emph{black}. For instance, the network in Figure \ref{fig:network} is perfect, with two leftmost interior vertices being black and two rightmost interior vertices being white.  


 An \emph{$n \times n$ plabic graph} is a finite connected undirected graph embedded in a rectangle, with $n$ univalent vertices on each of the vertical sides of the rectangle. 
All other vertices are in the interior of the rectangle and colored black or white. 
A \emph{perfect orientation} of a plabic graph is an orientation which makes it a perfect network. 
Plabic graphs admitting a perfect orientation are called \emph{perfectly orientable}. 
A plabic graph is \emph{(face-)weighted} if each of its faces is marked with a non-zero complex number, with the product of all numbers being $1$. We denote the space of all possible face weightings on a plabic graph $\Gamma$ by $\F(\Gamma) $. 
A face-weighted plabic graph is thus a pair $(\Gamma, \mathcal Y)$, where $\mathcal Y \in \F(\Gamma)$. Given a perfectly orientable face-weighted plabic graph, one associates to it a boundary measurement matrix by choosing a perfect orientation and calculating the matrix for the resulting face-weighted network. By theorem \cite[Theorem 10.1]{Pos}, that matrix is independent on the choice of a perfect orientation. A perfectly orientable plabic graph $\Gamma$ is \emph{non-degenerate} if the determinant of its boundary measurement matrix is not identically equal to $0$ as a function of face weights. 

Let $\Gamma$ be an $n \times n$ non-degenerate plabic graph, and let
$
\mes \colon \F(\Gamma) \dashrightarrow GL_n,
$
be the \emph{boundary measurement map}: by definition, for any $\mathcal Y \in \F(\Gamma)$, the matrix  $\mes( \mathcal Y)$ is the boundary measurement matrix of the weighted graph  $(\Gamma, \mathcal Y)$ (here and below dashed arrows are used for rational maps). 
Set $\Rect_n := \bigsqcup_\Gamma \F(\Gamma)$, where the union is taken over all isotopy classes of $n \times n$ non-degenerate plabic graphs. This space can be thought of as the space of {face-weighted} $n \times n$ non-degenerate plabic graphs, up to isotopy. It is a monoid under concatenation  (the definition of concatenation of face-weighted graphs is illustrated in Figure~\ref{fig:conc2}). The boundary measurement map extends to a map $
\mes \colon \Rect_n \dashrightarrow \GL{n},
$
which is a rational homomorphism of monoids.

 \begin{figure}[t]
 \centering
\begin{tikzpicture}[scale = 0.9]
\node () at (0,0) {
\begin{tikzpicture}[scale = 0.8]
\node [draw,circle,color=black, fill=white,inner sep=0pt,minimum size=3pt] (A) at (0,1) {};
\node [inner sep=0pt]  (B) at (1,1) {};
\node [draw,circle,color=black, fill=black,inner sep=0pt,minimum size=3pt] (C) at (0,2) {};
\node [inner sep=0pt]  (D) at (1,2) {};
\draw (-1, 1) -- (A)   node[midway, below] {};
\draw (B) -- (A) node[midway, below] {};
\draw (A) -- (C) node[midway, left] {};
\draw (C) -- (D) node[midway, above] {};
\draw (-1, 2) -- (C) node[midway, above] {};
\node () at (0,2.4){ \footnotesize $y_1$};
\node () at (-0.5,1.5){ \footnotesize $y_2$};
\node () at (0.5,1.5) {\footnotesize $y_3$};
\node () at (0,0.6){ \footnotesize $y_4$};

\draw [dashed] (-1,2.75) -- (1,2.75);
\draw [dashed] (1,0.25) -- (1,2.75) node[pos = 0.3, right] {\footnotesize$1$} node[pos = 0.7, right] {\footnotesize$2$};;
\draw [ dashed] (-1,0.25) -- (1,0.25) ;
\draw [ dashed]  (-1,0.25) -- (-1,2.75)   node[pos = 0.3, left] {\footnotesize$1$} node[pos = 0.7, left] {\footnotesize$2$};;
\end{tikzpicture}};
\node () at (2,0) {$\times$};
\node () at (4,0) {
\begin{tikzpicture}[scale = 0.8]
\node [draw,circle,color=black, fill=black,inner sep=0pt,minimum size=3pt] (A) at (0,1) {};
\node [inner sep=0pt]  (B) at (1,1) {};
\node [draw,circle,color=black, fill=white,inner sep=0pt,minimum size=3pt] (C) at (0,2) {};
\node  [inner sep=0pt]  (D) at (1,2) {};
\draw  (-1, 1) -- (A)   node[midway, below] {};
\draw  (B) -- (A) node[midway, below] {};
\draw  (A) -- (C) node[midway, left] {};
\draw  (C) -- (D) node[midway, above] {};
\draw  (-1, 2) -- (C) node[midway, above] {};
\node () at (0,2.4){ \footnotesize $y'_1$};
\node () at (-0.5,1.5){ \footnotesize $y'_2$};
\node () at (0.5,1.5) {\footnotesize $y'_3$};
\node () at (0,0.6){ \footnotesize $y'_4$};

\draw [dashed] (-1,2.75) -- (1,2.75);
\draw [dashed] (1,0.25) -- (1,2.75) node[pos = 0.3, right] {\footnotesize$1$} node[pos = 0.7, right] {\footnotesize$2$};;
\draw [ dashed] (-1,0.25) -- (1,0.25) ;
\draw [ dashed]  (-1,0.25) -- (-1,2.75)   node[pos = 0.3, left] {\footnotesize$1$} node[pos = 0.7, left] {\footnotesize$2$};;

\end{tikzpicture}};
\node () at (6,0) {$=$};
\node () at (9,0) {
\begin{tikzpicture}[scale = 0.8]
\node [draw,circle,color=black, fill=white,inner sep=0pt,minimum size=3pt] (A) at (0,1) {};
\node [draw,circle,color=black, fill=black,inner sep=0pt,minimum size=3pt] (B) at (2,1) {};
\node [draw,circle,color=black, fill=black,inner sep=0pt,minimum size=3pt] (C) at (0,2) {};
\node [draw,circle,color=black, fill=white,inner sep=0pt,minimum size=3pt] (D) at (2,2) {};
\draw  (-1, 1) -- (A)   node[midway, below] {};
\draw  (B) -- (A) node[midway, below] {};
\draw  (A) -- (C) node[midway, left] {};
\draw (C) -- (D) node[midway, above] {};
\draw  (-1, 2) -- (C) node[midway, above] {};
\draw  (B) -- (3,1) node[midway, below] {};
\draw  (D) -- (3,2) node[midway, above] {};
\draw  (D) -- (B) node[midway, left] {};
\node () at (1,2.4){ \footnotesize $y_1y_1'$};
\node () at (-0.5,1.5){ \footnotesize $y_2$};
\node () at (1,1.5) {\footnotesize $y_3y_2'$};
\node () at (2.5,1.5) {\footnotesize $y_3'$};
\node () at (1,0.6){ \footnotesize $y_4y_4'$};

\draw [dashed] (-1,2.75) -- (3,2.75);
\draw [dashed] (3,0.25) -- (3,2.75) node[pos = 0.3, right] {\footnotesize$1$} node[pos = 0.7, right] {\footnotesize$2$};;
\draw [ dashed] (-1,0.25) -- (3,0.25) ;
\draw [ dashed]  (-1,0.25) -- (-1,2.75)   node[pos = 0.3, left] {\footnotesize$1$} node[pos = 0.7, left] {\footnotesize$2$};;
\end{tikzpicture}

};
\end{tikzpicture}
\caption{Concatenation of face-weighted plabic graphs.}\label{fig:conc2}
\end{figure}

\paragraph{Poisson property of the boundary measurement map.} There is a natural Poisson structure on the face weight space of a plabic graph which is well-behaved under concatenation and the boundary measurement map \cite{GSV2}. This structure is conveniently described using the notion of the \emph{dual quiver}. For the purposes of the present paper, a quiver is a directed graph, with some edges designated as \emph{half-edges}. The quiver $\Q$ dual to a plabic graph $\Gamma$ is defined as follows. Vertices of $\Q$ correspond to faces of $\Gamma$.
Edges of $\Q$ correspond to edges of $\Gamma$ which connect either two internal vertices of different colors,
or an internal vertex with a boundary vertex. 
An edge $e^*$ in $\Q$ corresponding to $e$ in $\Gamma$ is directed in such a way that the white endpoint of $e$ (if it exists) lies to the left of $e^*$ and
the black endpoint of $e$ (if it exists) lies to the right of $e$.
Further, $e^*$ is a whole edge if both endpoints of $e$ are internal vertices, and a half-edge if one of the endpoints of $e$ is a boundary vertex, see Figure~\ref{fig:dual}.

  \begin{figure}
 \centering
\begin{tikzpicture}[scale = 1]
\node [draw,circle,color=black, fill=black,inner sep=0pt,minimum size=3pt, opacity = 0.3] (A) at (0,1) {};
\node [draw,circle,color=black, fill=white,inner sep=0pt,minimum size=3pt, opacity = 0.3] (B) at (1,1) {};
\node [draw,circle,color=black, fill=white,inner sep=0pt,minimum size=3pt, opacity = 0.3] (C) at (0,2) {};
\node [draw,circle,color=black, fill=black,inner sep=0pt,minimum size=3pt, opacity = 0.3] (D) at (1,2) {};
\node [draw,circle,color=black, fill=black,inner sep=0pt,minimum size=3pt, label =  {[label distance=-2]  above: \footnotesize $y_1$}] (Y1) at (0.5,2.5){};
\node [draw,circle,color=black, fill=black,inner sep=0pt,minimum size=3pt, label =  {[label distance=-2]  left: \footnotesize $y_2$}] (Y2) at (-0.5,1.5){ };
\node [draw,circle,color=black, fill=black,inner sep=0pt,minimum size=3pt, label =  {[label distance=-2]  45: \footnotesize $y_3$}] (Y3) at(0.5,1.5) {};
\node [draw,circle,color=black, fill=black,inner sep=0pt,minimum size=3pt, label = {[label distance=-2]  right: \footnotesize $y_4$}] (Y4) at (1.5,1.5) {};
\node [draw,circle,color=black, fill=black,inner sep=0pt,minimum size=3pt, label =  {[label distance=-4]  below: \footnotesize $y_5$}] (Y5) at(0.5,0.5){ };
\draw [thick, ->] (Y2) -- (Y3);
\draw [thick,->] (Y4) -- (Y3);
\draw [thick,->] (Y3) -- (Y1);
\draw [thick,->] (Y3) -- (Y5);
\draw [thick,dotted,->] (Y1) to [bend right] (Y2);
\draw [thick,dotted, ->] (Y5) to [bend left]  (Y2);
\draw [thick,dotted, ->] (Y1) to [bend left]  (Y4);
\draw [thick,dotted, ->] (Y5)   to [bend right](Y4);
\draw[ opacity = 0.3]  (-1, 1) -- (A)   node[midway, below] {};
\draw[ opacity = 0.3]  (B) -- (A) node[midway, below] {};
\draw [ opacity = 0.3] (A) -- (C) node[midway, left] {};
\draw [ opacity = 0.3](C) -- (D) node[midway, above] {};
\draw[ opacity = 0.3]  (-1, 2) -- (C) node[midway, above] {};
\draw [ opacity = 0.3] (B) -- (2,1) node[midway, below] {};
\draw [ opacity = 0.3] (D) -- (2,2) node[midway, above] {};
\draw [ opacity = 0.3] (D) -- (B) node[midway, left] {};

\draw [dashed, , opacity = 0.3] (-1,3) -- (2,3);
\draw [dashed, , opacity = 0.3] (2,3) -- (2,0) ;
\draw [ dashed, , opacity = 0.3] (-1,0) -- (2,0) ;
\draw [ dashed, , opacity = 0.3] (-1,0) -- (-1,3) ;

\end{tikzpicture}
\caption{A plabic graph and its dual quiver. Dotted arrows are half-edges.}\label{fig:dual}
    
\end{figure}

Consider the face weight space $\F(\Gamma)$ of a plabic graph $\Gamma$. As coordinates on that space one can take the face weights $y_1, \dots, y_f$, satisfying the relation $y_1 \cdots y_f = 1$. In those coordinates, the Poisson structure is the \emph{log-canonical bracket} defined by the dual quiver $\Q$:
$
\{y_i, y_j\} = q_{ij}y_iy_j.
$
Here $q_{ij}$ is the skew-symmetric adjacency matrix of $\Q$, i.e. $q_{ij} = \#\{ i \to j \} - \#\{ j \to i \}$ where $\#\{ i \to j\}$ stands for the number of edges of $\Q$ oriented from $i$ to $j$, counting every half-edge as $1/2$. The so-defined bracket descends to the subvariety $y_1  \cdots  y_f = 1$ and hence defines a Poisson structure on $\F(\Gamma)$. Endowed with such a bracket, the space $\Rect_n = \bigsqcup_\Gamma \F(\Gamma)$ of face-weighted plabic graphs is a \emph{Poisson monoid}, meaning that the concatenation map $\Rect_n \times \Rect_n \to \Rect_n$ is Poisson. 

\begin{example}\label{ex:qpb}
Consider the plabic graph $\Gamma$ in Figure \ref{fig:dual}. The coordinate ring $\C[\F(\Gamma)]$ of its face weight space $\F(\Gamma)$ is the quotient of the Laurent polynomial ring $\C[y_1^{\pm 1}, \dots, y_5^{\pm 1}]$ in the face weights $y_i$ by the ideal  $ \langle y_1y_2y_3y_4 y_5 -  1\rangle$. A Poisson bracket on $\C[y_1^{\pm 1}, \dots, y_5^{\pm 1}]$ can be read off from the dual quiver. On generators, it is given by
\begin{gather*}
\{y_1, y_2\} = \frac{1}{2}y_1y_2, \quad \{y_1, y_3\} = -y_1y_3, \quad \{y_1, y_4\} = \frac{1}{2}y_1y_4, \quad \{y_1, y_5\} = 0, \\
\{y_2, y_3\} = y_2y_3, \quad \{y_2,y_4\} = 0, \quad  \{y_2, y_5\} = -\frac{1}{2}y_2y_5,\\
\{y_3,y_4\} = -y_3y_4, \quad \{y_3,y_5\} = y_3y_5, \\
\{y_4,y_5\} =  -\frac{1}{2}y_4y_5.
\end{gather*}
(The brackets of the form $\{y_i, y_j\}$ for $i \geq j$, as well as brackets involving  $y_i^{-1}$, can be computed from the ones given using the standard properties of the Poisson bracket.) The function $y_1y_2y_3y_4 y_5 $ is in the center of this Poisson algebra. For example,
\begin{gather*}
\{y_1, y_1y_2y_3y_4 y_5 \}  =  \{y_1, y_2\} y_3y_4y_5 + y_1 \{y_1, y_3\} y_4y_5 + y_1y_2 \{y_1, y_4\} y_5  +  y_1y_2y_3 \{y_1, y_5\} \\ =  (\tfrac{1}{2} - 1 + \tfrac{1}{2}) \,y_1y_2y_3y_4 y_5 = 0.
\end{gather*}
As a result, $ \langle y_1y_2y_3y_4 y_5  - 1\rangle$ is a Poisson ideal. So, the quotient $\C[\F(\Gamma)] = \C[y_1^{\pm 1}, \dots, y_5^{\pm 1}] /  \langle y_1y_2y_3y_4 y_5  - 1\rangle$ inherits a Poisson algebra structure, turning $\F(\Gamma)$ into a Poisson variety. 
 \end{example}


Now, recall that any reductive Lie group $G$ carries a \emph{standard} multiplicative Poisson structure (where multiplicativity means that the multiplication map $G \times G \to G$ is Poisson, so that $G$ endowed with such a structure is a \emph{Poisson-Lie group}). It is determined by a choice of two complementary Borel subalgebras in the Lie algebra of $G$ and is, in a certain sense, the simplest of all multiplicative Poisson structures on $G$ \cite{belavin1982solutions}.

It is proved in \cite{GSV2} that the boundary measurement map $\mes \colon \Rect_n \dashrightarrow \GL{n}$ takes the log-canonical Poisson structure on $ \Rect_n$ to the standard Poisson structure on $\GL{n}$ (where as upper and lower Borel subalgebras one takes upper- and lower-triangular matrices respectively). 


\paragraph{From $GL$ to $PGL$.}

 \begin{figure}[t]
 \centering
\begin{tikzpicture}[scale = 0.9]
\node () at (0.5,0.25){ \footnotesize $\lambda^{-1}$};
\node () at (0.5,2.25){ \footnotesize $\lambda$};
\node () at (0.5,0.75){ \footnotesize $1$};
\node () at (0.5,1.25){ \footnotesize $1$};
\node () at (0.5,1.75){ \footnotesize $1$};

\draw [dashed] (0,2.5) -- (1,2.5);
\draw [dashed] (1,0) -- (1,2.5) ;
\draw [ dashed] (0,0) -- (1,0) ;
\draw [ dashed]  (0,0) -- (0,2.5);
\draw (0,0.5) -- (1,0.5);
\draw (0,1) -- (1,1);
\draw (0,2) -- (1,2);
\draw (0,1.5) -- (1,1.5);
\end{tikzpicture}
\caption{The center of the monoid of plabic graphs.}\label{fig:center}
\end{figure}

Graphs shown in Figure \ref{fig:center} form the center of the monoid $\Rect_n$ of plabic graphs. Taking the quotient of $\Rect_n$ by its center one gets the boundary measurement map valued in $\PGL_{n}$. Forming such a quotient is equivalent to considering face weightings with weights assigned to all faces except the upper and lowermost ones (without imposing any conditions on the product of weights). We call  such weightings \emph{projective}. 
The results of \cite{GSV2} imply that the pushforward of the  the log-canonical Poisson structure on $ \Rect_n / Z(\Rect_n)$ (defined for each graph $\Gamma$ by deleting the uppermost and lowermost vertices of the dual quiver of $\Gamma$) by the boundary measurement map $\mes \colon \Rect_n / Z(\Rect_n) \dashrightarrow \PGL_{n}$ is the standard Poisson structure on $\PGL_{n}$.

\subsection{Plabic graphs in types $B$ and $C$}

  \begin{figure}
 \centering
\begin{tikzpicture}[scale = 0.8]
\node [draw,circle,color=black, fill=black,inner sep=0pt,minimum size=3pt, opacity = 1] (A) at (0,1) {};
\node [draw,circle,color=black, fill=white,inner sep=0pt,minimum size=3pt, opacity = 1] (B) at (1,1) {};
\node [draw,circle,color=black, fill=white,inner sep=0pt,minimum size=3pt, opacity = 1] (C) at (0,2) {};
\node [draw,circle,color=black, fill=black,inner sep=0pt,minimum size=3pt, opacity = 1] (D) at (1,2) {};
\node  (Y1) at (0.5,2.4){\footnotesize$ y_1$};
\node(Y2) at (-0.5,1.5){ \footnotesize$ y_2$ };
\node  (Y3) at(0.5,1.5) {\footnotesize$ y_3$};
\node  (Y4) at (1.5,1.5) {\footnotesize$ y_4$};
\node (Y5) at(0.5,0.6){ \footnotesize$ y_1$};
\draw[ opacity = 1]  (-1, 1) -- (A)   node[midway, below] {};
\draw[ opacity = 1]  (B) -- (A);
\draw [ opacity = 1] (A) -- (C) node[midway, left] {};
\draw [ opacity = 1](C) -- (D);
\draw[ opacity = 1]  (C) -- +(-1,0) node[midway, above] {};
\draw [ opacity = 1] (B) -- +(1,0) node[midway, below] {};
\draw [ opacity = 1] (D) -- (2,2) node[midway, above] {};
\draw [ opacity = 1] (D) -- (B) node[midway, left] {};


\draw [dashed] (-1,2.75) -- (2,2.75);
\draw [dashed] (2,0.25) -- (2,2.75) ; 
\draw [ dashed] (-1,0.25) -- (2,0.25) ;
\draw [ dashed]  (-1,0.25) -- (-1,2.75);
%
%

\end{tikzpicture}
\caption{A face-weighted symmetric plabic graph.}\label{fig:symm}
    
\end{figure}

\paragraph{Move-symmetric plabic graphs.} One way to describe classical simple groups in types other than $A$ is as fixed point sets of a suitable involution on $\GL{n}$. Accordingly, one should expect that graphs modeling such groups are those with a certain symmetry. This is indeed so in type $C$. Define the \emph{midline} of a rectangle as the line joining the midpoints of its left and right sides. 
Say that a plabic graph is \emph{symmetric} if reflecting it in the midline one gets the same graph but with opposite color vertices. Furthermore, if such a graph is face-weighted, then we require that the reflection in the midline preserves the weights, see Figure~\ref{fig:symm}.
 Such graphs were introduced in~\cite{karpman2018combinatorics} and applied in \cite{karpman2018total} to the study of Lagrangian Grassmannians. It turns out that for a suitable symplectic form the boundary measurement matrix of such a graph is symplectic.

To describe the odd-dimensional orthogonal group, we consider plabic graphs that are, in a sense, as symmetric as they can be, namely symmetric up to \emph{square moves} (note that symmetric graphs per se necessarily have an even number of sources/sinks, unless a source on the midline is directly connected to a sink; for that reason, one cannot use symmetric graphs to describe odd-size matrices). A {square move} is a local modification of the graph shown in Figure \ref{fig:sqm}, see \cite[Section 12]{Pos}. 
  \begin{figure}[h]
 \centering
\begin{tikzpicture}[scale = 1]
\node  () at (0,0) {
\begin{tikzpicture}[scale = 0.9]
\node [draw,circle,color=black, fill=black,inner sep=0pt,minimum size=3pt, opacity = 1] (A) at (0,1) {};
\node [draw,circle,color=black, fill=white,inner sep=0pt,minimum size=3pt, opacity = 1] (B) at (1,1) {};
\node [draw,circle,color=black, fill=white,inner sep=0pt,minimum size=3pt, opacity = 1] (C) at (0,2) {};
\node [draw,circle,color=black, fill=black,inner sep=0pt,minimum size=3pt, opacity = 1] (D) at (1,2) {};
\node  (Y1) at (0.5,2.5){\footnotesize$y_1$};
\node(Y2) at (-0.5,1.5){ \footnotesize$y_4$ };
\node  (Y3) at(0.5,1.5) {\footnotesize$y_0$};
\node  (Y4) at (1.5,1.5) {\footnotesize$y_2$};
\node (Y5) at(0.5,0.5){ \footnotesize$y_3$};
\draw[ opacity = 1]  (-0.7, 0.3) -- (A)   node[midway, below] {};
\draw[ opacity = 1]  (B) -- (A) node[midway, below] {};
\draw [ opacity = 1] (A) -- (C) node[midway, left] {};
\draw [ opacity = 1](C) -- (D) node[midway, above] {};
\draw[ opacity = 1]  (-0.7, 2.7) -- (C) node[midway, above] {};
\draw [ opacity = 1] (B) -- (1.7,0.3) node[midway, below] {};
\draw [ opacity = 1] (D) -- (1.7,2.7) node[midway, above] {};
\draw [ opacity = 1] (D) -- (B) node[midway, left] {};


\end{tikzpicture}
};
\node  () at (3,0) {$\longrightarrow$};
\node  () at (7,0) {
\begin{tikzpicture}[scale = 0.9]
\node [draw,circle,color=black, fill=white,inner sep=0pt,minimum size=3pt, opacity = 1] (A) at (0,1) {};
\node [draw,circle,color=black, fill=black,inner sep=0pt,minimum size=3pt, opacity = 1] (B) at (1,1) {};
\node [draw,circle,color=black, fill=black,inner sep=0pt,minimum size=3pt, opacity = 1] (C) at (0,2) {};
\node [draw,circle,color=black, fill=white,inner sep=0pt,minimum size=3pt, opacity = 1] (D) at (1,2) {};
\node  (Y1) at (0.5,2.6){\footnotesize  $\displaystyle\frac{y_1}{1+y_0^{-1}}$};
\node(Y2) at (-0.9,1.5){ \footnotesize$y_4(1+y_0)$ };
\node  (Y3) at(0.5,1.5) {\footnotesize $y_0^{-1}$};
\node  (Y4) at (1.9,1.5) {\footnotesize $y_2(1+y_0)$};
\node (Y5) at(0.5,0.4){\footnotesize $\displaystyle\frac{y_3}{1+y_0^{-1}}$};
\draw[ opacity = 1]  (-0.7, 0.3) -- (A)   node[midway, below] {};
\draw[ opacity = 1]  (B) -- (A) node[midway, below] {};
\draw [ opacity = 1] (A) -- (C) node[midway, left] {};
\draw [ opacity = 1](C) -- (D) node[midway, above] {};
\draw[ opacity = 1]  (-0.7, 2.7) -- (C) node[midway, above] {};
\draw [ opacity = 1] (B) -- (1.7,0.3) node[midway, below] {};
\draw [ opacity = 1] (D) -- (1.7,2.7) node[midway, above] {};
\draw [ opacity = 1] (D) -- (B) node[midway, left] {};


\end{tikzpicture}
};
\end{tikzpicture}
\vspace{-10pt}
\caption{A square move.}\label{fig:sqm}
    
\end{figure}
 The indicated rule for face weights transformation is uniquely determined by the requirement of preservation of the boundary measurement matrix.  In terms of the dual quiver, a square move is a $Y$-type  \emph{cluster mutation} (in \cite{FG} such mutations are called $\chi$-type). 

\begin{definition}
A  plabic graph (weighted or not) is \emph{move-symmetric} if after performing square moves at all its square faces which have two opposite vertices on the midline, and then reflecting the resulting graph in the said line, one gets the initial graph (with initial weights if the graph was weighted) but with opposite color vertices,  see Figure \ref{fig:ms}. 
\end{definition}

Note that symmetric graphs do not have vertices on the midline and are, therefore, move-symmetric.

  \begin{figure}
 \centering
\begin{tikzpicture}[scale = 0.7]
\node () at (0,0){\begin{tikzpicture}[scale = 0.5]
\node [draw,circle,color=black, fill=black,inner sep=0pt,minimum size=3pt, opacity = 1] (A) at (0.3,0) {};
\node [draw,circle,color=black, fill=white,inner sep=0pt,minimum size=3pt, opacity = 1] (B) at (1,0.7) {};
\node [draw,circle,color=black, fill=black,inner sep=0pt,minimum size=3pt, opacity = 1] (C) at (1.7,0) {};
\node [draw,circle,color=black, fill=white,inner sep=0pt,minimum size=3pt, opacity = 1] (D) at (1,-0.7) {};
\node [draw,circle,color=black, fill=white,inner sep=0pt,minimum size=3pt, opacity = 1] (E) at (1,2) {};
\node [draw,circle,color=black, fill=black,inner sep=0pt,minimum size=3pt, opacity = 1] (F) at (1,-2) {};
\node () at (1,0) {\footnotesize$1$};
\node () at (0,1) {\footnotesize$y_2$};
\node () at (0,-1) {\footnotesize$2y_2$};
\node () at (2,1) {\footnotesize$y_3$};
\node () at (2,-1) {\footnotesize$\frac{1}{2}y_3$};
\node () at (1,2.5) {\footnotesize$y_1$};
\node () at (1,-2.5) {\footnotesize$y_1$};
\draw [] (A) -- (B) -- (C) -- (D) -- (A);
\draw (A) -- +(-1.3,0);
\draw (B) -- (E);
\draw (E) -- +(2,0);
\draw (E) -- +(-2,0);
\draw (C) -- +(1.3,0);
\draw (D) -- (F);
\draw (F) -- +(2,0);
\draw (F) -- +(-2,0);
\draw [dashed] (-1,-3) -- (-1,3) -- (3,3) -- (3,-3) -- cycle;
\end{tikzpicture}};
\node () at (4,0) {$\xrightarrow{\mbox{square move}}$};

\node () at (8,0){\begin{tikzpicture}[scale = 0.5]

\draw [dashed] (-1,-3) -- (-1,3) -- (3,3) -- (3,-3) -- cycle;
\node [draw,circle,color=black, fill=white,inner sep=0pt,minimum size=3pt, opacity = 1] (A) at (0.3,0) {};
\node [draw,circle,color=black, fill=black,inner sep=0pt,minimum size=3pt, opacity = 1] (B) at (1,0.7) {};
\node [draw,circle,color=black, fill=white,inner sep=0pt,minimum size=3pt, opacity = 1] (C) at (1.7,0) {};
\node [draw,circle,color=black, fill=black,inner sep=0pt,minimum size=3pt, opacity = 1] (D) at (1,-0.7) {};
\node [draw,circle,color=black, fill=white,inner sep=0pt,minimum size=3pt, opacity = 1] (E) at (1,2) {};
\node [draw,circle,color=black, fill=black,inner sep=0pt,minimum size=3pt, opacity = 1] (F) at (1,-2) {};
\node () at (1,0) {\footnotesize$1$};
\node () at (0,1) {\footnotesize$2y_2$};
\node () at (0,-1) {\footnotesize$y_2$};
\node () at (2,1) {\footnotesize$\frac{1}{2}y_3$};
\node () at (2,-1) {\footnotesize$y_3$};
\node () at (1,2.5) {\footnotesize$y_1$};
\node () at (1,-2.5) {\footnotesize$y_1$};
\draw [] (A) -- (B) -- (C) -- (D) -- (A);
\draw (A) -- +(-1.3,0);
\draw (B) -- (E);
\draw (E) -- +(2,0);
\draw (E) -- +(-2,0);
\draw (C) -- +(1.3,0);
\draw (D) -- (F);
\draw (F) -- +(2,0);
\draw (F) -- +(-2,0);

\end{tikzpicture}};
\node () at (12,0) {$\xrightarrow{\mbox{reflection}}$};

\node () at (16,0){\begin{tikzpicture}[xscale = 0.5, yscale = -0.5]

\draw [dashed] (-1,-3) -- (-1,3) -- (3,3) -- (3,-3) -- cycle;
\node [draw,circle,color=black, fill=white,inner sep=0pt,minimum size=3pt, opacity = 1] (A) at (0.3,0) {};
\node [draw,circle,color=black, fill=black,inner sep=0pt,minimum size=3pt, opacity = 1] (B) at (1,0.7) {};
\node [draw,circle,color=black, fill=white,inner sep=0pt,minimum size=3pt, opacity = 1] (C) at (1.7,0) {};
\node [draw,circle,color=black, fill=black,inner sep=0pt,minimum size=3pt, opacity = 1] (D) at (1,-0.7) {};
\node [draw,circle,color=black, fill=white,inner sep=0pt,minimum size=3pt, opacity = 1] (E) at (1,2) {};
\node [draw,circle,color=black, fill=black,inner sep=0pt,minimum size=3pt, opacity = 1] (F) at (1,-2) {};
\node () at (1,0) {\footnotesize$1$};
\node () at (0,1) {\footnotesize$2y_2$};
\node () at (0,-1) {\footnotesize$y_2$};
\node () at (2,1) {\footnotesize$\frac{1}{2}y_3$};
\node () at (2,-1) {\footnotesize$y_3$};
\node () at (1,2.5) {\footnotesize$y_1$};
\node () at (1,-2.5) {\footnotesize$y_1$};
\draw [] (A) -- (B) -- (C) -- (D) -- (A);
\draw (A) -- +(-1.3,0);
\draw (B) -- (E);
\draw (E) -- +(2,0);
\draw (E) -- +(-2,0);
\draw (C) -- +(1.3,0);
\draw (D) -- (F);
\draw (F) -- +(2,0);
\draw (F) -- +(-2,0);

\end{tikzpicture}};
\end{tikzpicture}
\caption{ A face-weighted move-symmetric plabic graph. 
}\label{fig:ms}
    
\end{figure}


\paragraph{The main result.} Let 
$\Omega_n := \sum_{i=1}^n (-1)^{i+1}E_{i, n+1 - i}$ (where $E_{i,j}$ is a matrix with a $1$ at position $(i,j)$ and zeros elsewhere) be the $n \times n$ anti-diagonal  matrix with alternating $\pm 1$ on the anti-diagonal. Note that $\Omega_n$ is skew-symmetric for even $n$ and symmetric for odd~$n$. Let $G(\Omega_n) := \{ A \in \GL{n} \mid A\Omega_nA^t = \Omega_n\}$ be the set of matrices $A$ such that the corresponding linear operator preserves the form $\Omega_n$. It is the symplectic group $  \Sp(\Omega_n)$ for even $n$ or the orthogonal group $ O(\Omega_n)$ for odd $n$. 
The special choice of the form $\Omega_n$ provides a canonical Poisson structure on  $G(\Omega_n) $. The corresponding upper and lower Borel subalgebras consist of, respectively, upper- and lower-triangular matrices. 

Define also the nonnegative part $G^{\geq 0} (\Omega_n)$ as the set of real matrices that lie in $G(\Omega_n) $ and are totally nonnegative. We will show in Section \ref{sec:tp} that this agrees with the definition of total nonnegativity in a reductive group given in \cite{Lus}.

Given an unweighted move-symmetric plabic graph $\Gamma$, denote by $ \F^{ms}(\Gamma)$ the space of all face weightings $\mathcal Y \in \F^{ms}(\Gamma)$ such that~$(\Gamma, \mathcal Y)$ is move-symmetric as a weighted graph. 



%




\begin{theorem}\label{thm1} 
Let $\Gamma$ be an  $n \times n$  perfectly orientable move-symmetric plabic graph. Then: 
\begin{enumerate} 
\item For any move-symmetric weighting $\mathcal Y$, the boundary measurement matrix $A := \mes(\mathcal Y)$ is in $G(\Omega_n)$, i.e., is symplectic if $n$ is even and orthogonal  if $n$ is odd. 
Moreover, for positive real weights we have $A \in G(\Omega_n)^{\geq 0} $.

\item There is a   log-canonical multiplicative Poisson bracket on $ \F^{ms}(\Gamma)$ whose pushforward by the boundary measurement map  $ \F^{ms}(\Gamma)\dashrightarrow G(\Omega_n)$ coincides with the standard Poisson structure on $G(\Omega_n)$. 
\end{enumerate}
\end{theorem}

Put differently, let $\Rect_n^{ms} := \bigsqcup_\Gamma \F^{ms}(\Gamma)$, where the union is taken over all isotopy classes of $n \times n$ perfectly orientable move-symmetric plabic graphs. Then Theorem \ref{thm1} says that the boundary measurement map is a morphism of Poisson-Lie monoids $\Rect_n^{ms} \dashrightarrow G(\Omega_n)$ which takes the positive part of $\Rect_n^{ms}$ to the nonnegative part of $ G(\Omega_n)$. 

Theorem \ref{thm1} also admits a centerless version. The center of the monoid  $\Rect_n^{ms}$ of move-symmetric plabic graphs consists of graphs  shown in Figure \ref{fig:center} with $\lambda = \pm 1$. Taking the quotient by the center is equivalent to forgetting the weights of the uppermost and lowermost faces. The boundary measurement map on $\Rect_n^{ms} / Z(\Rect_n^{ms})$ takes values in the centerless group $\PSp(\Omega_n)$ for even $n$ and $\SO(\Omega_n)$ for odd $n$. 

Also note that, in the real case, $\Omega_n$ is either a real symplectic form or a real inner product of split signature. So, move-symmetric plabic graphs with real weights correspond to split real forms $\Sp_{2k}(\R)$, $O_{k, k+1}(\R)$ of complex Lie groups $\Sp_{2k}$, $O_{2k+1}$.

The proof of Theorem \ref{thm1} consists of three parts. In Section \ref{sec2}, we show that the boundary measurement matrix of a move-symmetric plabic graph is in $G(\Omega_n)$. In Section \ref{sec3}, we prove that there exists a multiplicative Poisson structure on the space $\F^{ms}(\Gamma)$ of move-symmetric weightings of a move-symmetric graph such that the boundary measurement map $ \F^{ms}(\Gamma)\dashrightarrow G(\Omega_n)$ is Poisson. Finally, in Section \ref{sec:pbc}, we show that this bracket is log-canonical and describe the corresponding quiver.

\medskip

\section{The boundary measurement matrix of a move-symmetric plabic graph}\label{sec2}

\paragraph{The main lemma.}
Let $\Gamma$ be a move-symmetric plabic graph. Define a birational involution $\sigma \colon \F(\Gamma) \dashedrightarrow \F(\Gamma)$ on its face weight space as the composition of square moves at all square faces which have two opposite vertices on the midline, and reflection in the midline. The fixed point set of $\sigma$ is the space $ \F^{ms}(\Gamma)$  of move-symmetric face weightings. 

Recall that $\Omega_n = \sum_{i=1}^n (-1)^iE_{i, n+1 - i}$ is the $n \times n$ anti-diagonal  matrix with alternating $\pm 1$ on the anti-diagonal, and let $\tau \colon \GL{n} \to \GL{n}$ be an involution given by $\tau(A) := \Omega_n A^{-t} \Omega_n^{-1}$, where $A^{-t} := (A^{-1})^t$. The fixed point set of $\tau$ is the subgroup $G(\Omega_n) = \{A \in \GL{n} \mid  A\Omega_n A^t = \Omega_n\}$.
The goal of this section is to establish the following.
\begin{lemma}\label{prop:sigma}
Let $\Gamma$ be an $n \times n$ perfectly orientable move-symmetric plabic graph. Then the following diagram commutes:
\begin{equation}\label{mainDiag}
\begin{tikzcd}
\F(\Gamma) \arrow[d,dashed,"\mes", swap] \arrow[dashed, r,"\sigma"]
&\F( \Gamma) \arrow[d, dashed, "\mes"]  \\ 
\GL{n}  \arrow[r,"\tau"] &
\GL{n}.
\end{tikzcd}
\end{equation}

\end{lemma}

This lemma immediately implies the first part of Theorem \ref{thm1}: for a move-symmetric plabic graph $\Gamma$ with a move-symmetric weighting $\mathcal Y \in \F^{ms}(\Gamma)$ the boundary measurement matrix $A$ belongs to $G(\Omega_n)$ (this also implies  that for real positive weights we have   $A \in G(\Omega_n)^{\geq 0}$). Indeed, since any $\mathcal Y \in \F^{ms}(\Gamma)$ is fixed by the involution $\sigma \colon \F(\Gamma) \dashrightarrow \F(\Gamma)$, the corresponding boundary measurement matrix $A:=\mes(\mathcal Y)$ is fixed by $\tau$, meaning that $A \in G(\Omega_n) $.

 Furthermore, we will also use Lemma \ref{prop:sigma} to prove that the boundary measurement map $\mes \colon  \F^{ms}(\Gamma) \dashrightarrow G(\Omega_n)$ is Poisson, and thus establish the second part of  Theorem \ref{thm1}.

To prove Lemma \ref{prop:sigma}, we  study how each of the steps defining the involution $\sigma$ affects the boundary measurement map. Since square moves do not alter the boundary measurement matrix, it suffices to understand how it changes under reflection in the midline and recoloring of vertices. 
\paragraph{Reflection of plabic graphs.} The following describes the relation between the boundary measurement maps of plabic graphs which are symmetric to each other with respect to the midline.
\begin{proposition}\label{prop:reflect}
Let $\Gamma$ be an $n \times n$ perfectly orientable plabic graph, and let $\rho(\Gamma)$ be the graph obtained from $\Gamma$ by reflection in the midline. Consider the map $\rho_* \colon \F(\Gamma) \to \F(\rho(\Gamma))$ induced by the reflection, and let $\rho_*(\mathcal Y )^{-1}$ be the weighting obtained from $\rho_*(\mathcal Y)$ by inverting all the weights.  
Let also $\lw := \sum_{i=1}^n E_{i, n+1 - i}  \in S_n$ be the order-reversing permutation (throughout the paper, we identify permutations and corresponding permutation matrices). Then the following diagram commutes 
\begin{equation*}
\begin{tikzcd}[column sep = 1.5cm]
\F(\Gamma) \arrow[d,dashed,"\mes", swap] \arrow[r,"\mathcal Y \,\mapsto\, \rho_*(\mathcal Y)^{-1}   "]
&\F(\rho(\Gamma)) \arrow[d, dashed, "\mes"]  \\ 
M_n  \arrow[r,"A \,\mapsto\, \lw A\lw "] &
M_n.
\end{tikzcd}
\end{equation*}
\end{proposition}
\begin{proof}
Let $\mathcal Y \in \F(\Gamma) $ be a collection of face weights on $\Gamma$, and let $\mes(\mathcal Y) = A$ be the corresponding boundary measurement matrix. We need to show that the graph $(\rho(\Gamma),\rho_*(\mathcal Y)^{-1})$, obtained from $(\Gamma, \mathcal Y)$ by reflection and inverting the face weights, has boundary measurement matrix $ \lw  A \lw $.
To that end, take a perfect orientation of $\Gamma$ and choose arbitrary edge weights in the gauge-equivalence class determined by the face weighting $\mathcal Y$. This gives a perfect network $\mathcal N$. Consider the network $\rho(\mathcal N)$ obtained from $\mathcal N$ by reflection in the midline.  Observe that reflection gives a weight-preserving bijection between paths from source $i$ to sink $j$ in $\mathcal N$ and paths from source $n+1-i$ to sink $n+1-j$ in $\rho(\mathcal N)$ (recall that we always label sources and sinks from bottom to top). So, the boundary measurement matrix of $\rho(\mathcal N)$ is $ \lw  A \lw $. At the same time, note that the weight of every face of $\rho(\mathcal N)$ is reciprocal to the weight of its mirror image in $\mathcal N$  (since relative orientations of edges around a face are reversed
upon reflection). So, the network $\rho(\mathcal N)$, viewed as a face-weighted plabic graph, is exactly $(\rho(\Gamma),\rho_*(\mathcal Y)^{-1})$, proving that the latter has boundary measurement matrix $ \lw  A \lw $.
\end{proof}

\paragraph{Recoloring of plabic graphs.} Next, we need to describe what happens to the boundary measurement matrix when we recolor vertices. For a plabic graph $\Gamma$, denote the same graph but with opposite color vertices by $\op$. 
Recall that a perfectly orientable plabic graph $\Gamma$ is \emph{non-degenerate} if the determinant of its boundary measurement matrix, viewed as a function on the face weight space $\F(\Gamma)$, is not identically $0$. 
\begin{proposition}\label{prop:nondeg} A perfectly orientable plabic graph $\Gamma$ is non-degenerate if and only if  $\op$ is also perfectly orientable (so, a plabic graph $\Gamma$ is non-degenerate if and only if both $\Gamma$ and $\op$ are perfectly orientable). 
\end{proposition}
\begin{proof}
Assume that  $\Gamma$ is a non-degenerate $n \times n$ plabic graph. Take some perfect orientation of $\Gamma$. Since  $\Gamma$ is non-degenerate, by \cite[Theorem 1.1]{Talaska}, it admits a collection of $n$ self-avoiding vertex-disjoint directed paths, each connecting a source to a sink.  Reverse the directions of all edges in $\Gamma$ that do not belong to the said collection. This gives a perfect orientation of $\op$, as needed.

Conversely, assume that $\op$ is perfectly orientable. Take  its perfect orientation and reverse the directions of all edges along with the colors of all vertices. This gives an orientation of $\Gamma$ which is perfect except for all sources being on the right and sinks on the left. Fix some genuine perfect orientation of $\Gamma$ and compare it to the so-obtained one. Since both orientations are perfect (in the sense that each white vertex has a single incoming edge and each black vertex has a single outgoing edge), the union of all edges on which they disagree is a vertex-disjoint collection of directed paths, each being either a self-avoiding cycle or a self-avoiding path joining a source to a sink (cf. \cite[Proof of Theorem 10.1.]{Pos}). Furthermore, since all sources of either orientation are sinks of the other, the said collection must have exactly $n$ paths of the latter type. So, by \cite[Theorem 1.1]{Talaska}, $\Gamma$ is non-degenerate, as desired. \end{proof}

\begin{proposition}\label{prop:recolor}
Let $\Gamma$ be an $n \times n$ non-degenerate plabic graph. Consider the map $ \F(\Gamma) \to \F(\bar \Gamma)$, $\mathcal Y \mapsto \mathcal Y^{-1}$ given by inverting all the weights. 
 Let also 
 $D_n := \sum_{i=1}^n (-1)^{i+1}E_{i,i} \in \GL{n}$ be the diagonal matrix with alternating $\pm 1$ on the diagonal.  Then the following diagram commutes 
\begin{equation*}
\begin{tikzcd}[column sep = 2cm]
\F(\Gamma) \arrow[d,dashed,"\mes", swap] \arrow[r,"\mathcal Y \,\mapsto \, \mathcal Y^{-1}"]
&\F(\op) \arrow[d, dashed, "\mes"]  \\ 
\GL{n}  \arrow[r,"A\, \mapsto \, D_n A^{-t} D_n"] &
\GL{n}.
\end{tikzcd}
\end{equation*}
\end{proposition}
\begin{proof}
Let $(\Gamma, \mathcal Y)$ be a face-weighted $n \times n$ non-degenerate plabic graph with boundary measurement matrix $\mes(\mathcal Y) = A$.
Taking a perfect orientation of $\Gamma$ and choosing arbitrary edge weights in the gauge-equivalence class determined by the face weighting $\mathcal Y$, we get a perfect network $\mathcal N$. Reverse all edges of $\mathcal N$, keeping the edge weights intact. The resulting network $\bar {\mathcal N}$ is perfect except for all sources being on the right and sinks on the left.  The corresponding plabic graph is $\op$ with face weights given by $\mathcal Y^{-1}$. The boundary measurement matrix of $\bar{ \mathcal N}$, which can be viewed as the  \emph{right-to-left} boundary measurement matrix of $(\bar \Gamma,\mathcal Y^{-1})$,  is $A^t$ (indeed, path reversal gives a weight-preserving one-to-one correspondence between directed paths in $\mathcal N$ from source $i$ to sink $j$ and  directed paths in $\bar{ \mathcal N}$ from source $j$ to sink $i$).  We need to show that the usual, left-to-right, boundary measurement matrix of $(\bar \Gamma,\mathcal Y^{-1})$ is $D_n A^{-t} D_n$. To that end, it suffices to prove that left-to-right and right-to-left boundary measurement matrices of a non-degenerate plabic graph are related by the involution $A \leftrightarrow D_n A^{-1} D_n$. 

Let $A$ and $B$ be, respectively, left-to-right and right-to-left boundary measurement matrices of a face-weighted $n \times n$ non-degenerate plabic graph (the above argument shows that for a non-degenerate graph both boundary measurement maps are well-defined). We aim to show that $B = D_n A^{-1} D_n$. To that end, recall the notion of the boundary measurement map into the Grassmannian \cite[Definition 4.6]{Pos}. Consider a face-weighted plabic graph $\Gamma$ in a disk, with boundary vertices labeled by $1, \dots, m$, in counter-clockwise order. Take any perfect orientation of $\Gamma$ (i.e. an orientation such that each white vertex has a single incoming edge and each black vertex has a single outgoing edge; there is no condition on location of sources and sinks). Each such orientation (assuming at least one exists) has the same number $k$ of sources. The corresponding boundary measurement matrix $C$ is a $k \times m$ matrix, with rows labeled by sources $i_1 < \dots < i_k$ and columns labeled by all boundary vertices $1, \dots, m$. The $(i,j)$ entry $C_{ij}$ of that matrix is defined as the right-hand side of \eqref{eq:bmm} times $(-1)^{s(i,j)}$, where $s(i,j)$ is the number of sources whose labels are strictly between $i$ and $j$ in terms of the linear ordering on the set $1, \dots, m$. In particular, if the label $j$ corresponds to a source, then $$C_{ij} = \begin{cases} 1\quad  \mbox{if } i = j, \\ 0\quad  \mbox{ if } i \neq j\end{cases}$$ (indeed, in that case there are no directed paths from $i$ to $j$ unless $i = j$, in which case there is an empty path whose weight is, by definition, equal to $1$). The so-defined boundary measurement matrix depends on the choice of a perfect orientation, but the associated element of $\Gr_{k,m}$ does not \cite[Theorem 10.1]{Pos}.

\begin{example}
Figure \ref{fig:grex} shows two perfect orientations of the same face-weighted plabic graph, along with the corresponding boundary measurement matrices. The matrix on the left is the reduced echelon form of the matrix on the right, so these matrices indeed define the same element in $\Gr_{2,3}$.

  \begin{figure}
 \centering
\begin{tikzpicture}[scale = 1]
        \node at (4,0) {

\begin{tikzpicture}[scale = 1]
 \node  at (30:0.5) [text opacity = 0.3]  {\footnotesize$y$};
     \node at (150:0.5) [text opacity = 0.3]  {\footnotesize$x$};
        \node at (-90:0.5)  [text opacity = 0.3] {\footnotesize$x^{-1}y^{-1}$};
\draw [dashed, ] (0,0) circle [radius=1cm];
\node [draw,circle,color=black, fill=black,inner sep=0pt,minimum size=3pt, opacity = 1] (B) at (0,0) {};
 \draw [->-, thick](B) -- (-30:1)node[midway, above] {\footnotesize $y$};;
  \draw [->- , thick] (210:1)--(B)node[midway, above] {\footnotesize $x$};;
    \draw  [->-, thick ](90:1)--(B) node[midway, right] {\footnotesize $\!1$};
 \node [label = {-30:\footnotesize $3$}] at (-20:0.8)  {};
 \node [label = {above:\footnotesize$1$}] at (90:0.8)  {};
  \node [label = {210:\footnotesize$2$}] at (200:0.8)  {};
        \end{tikzpicture}
        };
             \node at (7,0) {
$\left(\begin{array}{ccc}1 & 0 & -y \\0 & 1 & xy\end{array}\right)$
        };

            \node at (11,0) {
\begin{tikzpicture}[scale = 1]
\draw [dashed, ] (0,0) circle [radius=1cm];
\node [draw,circle,color=black, fill=black,inner sep=0pt,minimum size=3pt, opacity = 1] (B) at (0,0) {};
 \draw [->-, thick](-30:1)--(B)node[midway, above] {\footnotesize $\,\,\,y^{-1}$};;
  \draw [->- , thick] (210:1) -- (B)node[midway, above] {\footnotesize $x$};;
    \draw  [->- , thick](B) -- (90:1) node[midway, right] {\footnotesize $\!1$};
 \node [label = {-30:\footnotesize $3$}] at (-20:0.8)  {};
 \node [label = {above:\footnotesize$1$}] at (90:0.8)  {};
  \node [label = {210:\footnotesize$2$}] at (200:0.8)  {}; \node  at (30:0.5) [text opacity = 0.3]  {\footnotesize$y$};
     \node at (150:0.5) [text opacity = 0.3]  {\footnotesize$x$};
        \node at (-90:0.5)  [text opacity = 0.3] {\footnotesize$x^{-1}y^{-1}$};
        \end{tikzpicture}
        };
                     \node at (14,0) {
$\left(\begin{array}{ccc}x & 1 & 0  \\ -y^{-1} &0 & 1 \end{array}\right)$
        };

\end{tikzpicture}
\caption{Two perfect orientations of the same plabic graph in a disk and the corresponding boundary measurement matrices, representing the same element in the Grassmannian. The face weights of the underlying plabic graph are shown in grey.}\label{fig:grex}
    
\end{figure}
\end{example}

 \begin{figure}
 \centering
\begin{tikzpicture}[scale = 0.7]
\node () at (0,0) {
\begin{tikzpicture}[scale = 0.7]
\node [draw,circle,color=black, fill=white,inner sep=0pt,minimum size=3pt] (A) at (0,1) {};
\node [inner sep=0pt]  (B) at (1,1) {};
\node [draw,circle,color=black, fill=black,inner sep=0pt,minimum size=3pt] (C) at (0,2) {};
\node [inner sep=0pt]  (D) at (1,2) {};
\draw (-1, 1) -- (A)   node[midway, below] {};
\draw (B) -- (A) node[midway, below] {};
\draw (A) -- (C) node[midway, left] {};
\draw (C) -- (D) node[midway, above] {};
\draw (-1, 2) -- (C) node[midway, above] {};

\draw [dashed] (-1,2.75) -- (1,2.75);
\draw [dashed] (1,0.25) -- (1,2.75) node[pos = 0.3, right] {\footnotesize$1$} node[pos = 0.7, right] {\footnotesize$2$};;
\draw [ dashed] (-1,0.25) -- (1,0.25) ;
\draw [ dashed]  (-1,0.25) -- (-1,2.75)   node[pos = 0.3, left] {\footnotesize$1$} node[pos = 0.7, left] {\footnotesize$2$};;

\end{tikzpicture}};
\node () at (3,0) {$\longrightarrow$};
\node () at (6,0) {
\begin{tikzpicture}[scale = 0.7]
\node [draw,circle,color=black, fill=white,inner sep=0pt,minimum size=3pt] (A) at (0,1) {};
\node [inner sep=0pt]  (B) at (1,1) {};
\node [draw,circle,color=black, fill=black,inner sep=0pt,minimum size=3pt] (C) at (0,2) {};
\node [inner sep=0pt]  (D) at (1,2) {};
\draw (-1, 1) -- (A)   node[midway, below] {};
\draw (B) -- (A) node[midway, below] {};
\draw (A) -- (C) node[midway, left] {};
\draw (C) -- (D) node[midway, above] {};
\draw (-1, 2) -- (C) node[midway, above] {};

\draw [dashed] (-1,2.75) -- (1,2.75);
\draw [dashed] (1,0.25) -- (1,2.75) node[pos = 0.3, right] {\footnotesize$1$} node[pos = 0.7, right] {\footnotesize$2$};;
\draw [ dashed] (-1,0.25) -- (1,0.25) ;
\draw [ dashed]  (-1,0.25) -- (-1,2.75)   node[pos = 0.3, left] {\footnotesize$4$} node[pos = 0.7, left] {\footnotesize$3$};;

\end{tikzpicture}};
\end{tikzpicture}
\caption{Relabeling the sources, we go from the matrix-valued boundary measurement map to the Grassmannian-valued boundary measurement map.}\label{fig:relabeling}
\end{figure}

Now, return to the setting of a face-weighted $n \times n$ non-degenerate plabic graph in a rectangle, with left-to-right boundary measurement matrix $A$ and  right-to-left boundary measurement matrix $B$. To compute the associated Grassmannian boundary measurement map, we keep the labeling of sinks unchanged and relabel the sources  $(1, \dots, n) \mapsto (2n, \dots, n+1)$, see Figure \ref{fig:relabeling}. Then the left-to-right perfect orientation gives the $n \times 2n$ matrix $(D_n\lw A, \Id_n)$, where $\Id_n$ is the $n \times n$ identity matrix. Here multiplication by $\lw $ comes from reversal of the order of sources, while multiplication by $D_n$ takes care of signs $(-1)^{s(i,j)}$. Likewise, the right-to-left perfect orientation gives $(\Id_n, (-1)^{n-1}D_nB\lw )$. Since the Grassmannian elements defined by the two boundary measurement matrices must be the same, we have
$
(-1)^{n-1}D_nB\lw  D_n\lw A = \Id_n,
$
which, due to the identity  $\lw  D_n\lw  = (-1)^{n-1}D_n$, implies $B = D_nA^{-1}D_n$, as claimed.
\end{proof}



%

\paragraph{Proof of the main lemma.}

We now prove Lemma \ref{prop:sigma}. Let $\Gamma$ be a perfectly orientable move-symmetric plabic graph $\Gamma$. We first check that the vertical arrows in diagram \eqref{mainDiag} are well-defined, which is equivalent to saying that $\Gamma$ is non-degenerate. Let $\Gamma'$ be the graph obtained from $\Gamma$ as the result of square moves at all square faces which have two opposite vertices on the midline, and $\rho(\Gamma')$ be the reflection of $\Gamma'$ in the said line. By \cite[Lemma 12.2]{Pos}, $\Gamma'$ is perfectly orientable, and hence so is $\rho(\Gamma')$. At the same time, by definition of a move-symmetric graph we have $\rho(\Gamma')= \bar \Gamma$, so $\bar \Gamma$ is also perfectly orientable, which implies that $\Gamma$ is non-degenerate.  

Now, we show that diagram \eqref{mainDiag} commutes. To that end, we represent the involution $\sigma$ on $ \F(\Gamma)$ as the following composition of maps:
$$
\F(\Gamma)  \dashrightarrow \F(\Gamma') \xrightarrow{\mathcal Y \mapsto \rho_*(\mathcal Y)^{-1}} \F(\rho(\Gamma')) \xrightarrow{\mathcal Y \to \mathcal Y^{-1}}  \F(\xoverline[0.8]{ \rho(\Gamma')} = \Gamma).
$$
The leftmost map is given by square moves and hence does not alter the boundary measurement matrix. So, using Propositions \ref{prop:reflect} and \ref{prop:recolor} (the latter applies since non-degeneracy of $\Gamma$ implies non-degeneracy of $\rho(\Gamma') = \bar \Gamma$), we get
$$
\mes(\sigma(\mathcal Y)) =   D_n(\lw  \mes(\mathcal Y) \lw )^{-t}D_n = D_n \lw  \mes(\mathcal Y)^{-t} \lw  D_n = \Omega_n \mes(\mathcal Y)^{-t} \Omega_n^{-1}.
$$
This completes the proof of Lemma \ref{prop:sigma} and hence the first part of Theorem \ref{thm1}. \qed
%

\medskip

\section{Poisson property of the boundary measurement map}\label{sec3}

Let $\Gamma$ be a  perfectly orientable move-symmetric plabic graph, and $ \F^{ms}(\Gamma)$ be the space of move-symmetric weightings on $\Gamma$. 
In this section, we show that \emph{there exists} a multiplicative Poisson structure on $ \F^{ms}(\Gamma)$ whose pushforward by the boundary measurement map  $ \F^{ms}(\Gamma)\dashrightarrow G(\Omega_n) $  is the standard Poisson structure on $G(\Omega_n)$. 
In the next section, we will show that this Poisson structure on $ \F^{ms}(\Gamma)$ is log-canonical and describe the corresponding quiver. 

\paragraph{Poisson algebras, Poisson varieties, and Poisson-Lie groups.}\label{poissonbasics} The material discussed in this section is fairly standard and can be found, e.g., in \cite{laurent2012poisson}. A \emph{Poisson algebra} is a unital commutative associative algebra endowed with a Lie bracket $\{\,, \}$, called the \emph{Poisson bracket}, which is a derivation of the associative product: $\{ab,c\} = a\{b,c\} + \{a,c\}b$. A homomorphism of Poisson algebras is a map which is both an associative and Lie algebra homomorphism. An \emph{affine Poisson variety} is an affine variety whose coordinate ring carries a structure of a Poisson algebra. A regular map $f \colon X \to Y$ between affine Poisson varieties is Poisson if the corresponding map $f^\sharp \colon \C[Y] \to \C[X]$ between their coordinate rings is a homomorphism of Poisson algebras. Given a Poisson variety $X$, it is common to refer to the associated bracket as a Poisson bracket on $X$, even though it is actually an operation on $\C[X]$.


A Poisson algebra structure on a (unital commutative)  finitely generated associative algebra is determined by pairwise Poisson brackets of generators. Similarly, if $A$ is generated by $x_1, \dots, x_n$ and $x_1^{-1}, \dots, x_n^{-1}$, then it still suffices to define Poisson brackets of the form $\{x_i,x_j\}$, since $\{x^{-1}, y\} = -x^{-2}\{x, y\}$ for any unit $x$ and any $y \in A$.

In the case when $A$ is freely generated, i.e., when $A$ is the polynomial ring $\C[x_1, \dots, x_n]$, any skew-symmetric bracket on generators $x_i$  extends to a bracket on $A$ which has all properties of a Poisson bracket except, possibly, the Jacobi identity. To verify the Jacobi identity, it suffices to check it for generators. Furthermore, any Poisson bracket on the polynomial ring $A = \C[x_1, \dots, x_n]$ uniquely extends to the Laurent polynomial ring $\C[x_1^{\pm 1}, \dots, x_n^{\pm 1}]$, as well as any other localization of $A$.

\begin{example}\label{ex:torus}
Let $(q_{ij})$ be an $n \times n$ skew-symmetric matrix. Poisson brackets of the form
$
\{x_i, x_j\} = q_{ij}x_ix_j
$
are called \emph{log-canonical}. It is easy to see that any log-canonical bracket satisfies the Jacobi identity and hence gives a Poisson structure on the Laurent polynomial ring $\C[x_1^{\pm 1}, \dots, x_n^{\pm 1}]$. It can be thought of as a Poisson bracket on the algebraic torus $(\C^*)^n = \mathrm{Spec} \,\C[x_1^{\pm 1}, \dots, x_n^{\pm 1}]$.




%
\end{example}
The tensor product of Poisson algebras is naturally a Poisson algebra. The Poisson bracket on the tensor product is defined by $$\{a_1 \otimes b_1, a_2 \otimes b_2\} =  a_1a_2\otimes \{b_1,  b_2\}  + \{a_1,  a_2\}\otimes b_1b_2. $$ In particular, if $X$, $Y$ are Poisson varieties, then $X \times Y$ is also naturally a Poisson variety, since $\C[X \times Y] = \C[X] \otimes \C[Y]$. The corresponding Poisson structure on  $X \times Y$ is called the \emph{product Poisson structure}.


Let $G$ be a Lie group (to remain within the setting of affine varieties, we assume that $G$ is algebraic; however, the whole discussion easily extends to the setting of differentiable manifolds and arbitrary Lie groups). A Poisson structure on $G$ is called \emph{multiplicative} if the multiplication map $G \times G \to G$ is Poisson. Here $G \times G$ is endowed with the product Poisson structure. A Lie group with a multiplicative Poisson structure is called a \emph{Poisson-Lie group}. 

\begin{example}
The \emph{standard Poisson-Lie structure} (described below for any reductive group) on $\GL{2}$ is defined on matrix elements $a_{ij} \in \C[\GL{2}]$ by
\begin{gather*}
\{a_{11}, a_{12}\} = \frac{1}{2}a_{11}a_{12}, \quad \{a_{11}, a_{21}\} = \frac{1}{2}a_{11}a_{21}, \quad \{a_{11}, a_{22}\} = a_{12}a_{21}, \\  \{a_{12}, a_{21}\} = 0, \quad   \{a_{12}, a_{22}\}  = \frac{1}{2}a_{12}a_{22},   \quad \{a_{21}, a_{22}\}  = \frac{1}{2}a_{21}a_{22},
\end{gather*}
Multiplicativity of this bracket means that the co-multiplication map $$\Delta \colon  \C[\GL{2}] \to  \C[\GL{2}] \otimes  \C[\GL{2}], \quad \Delta(a_{ij}) = a_{i1} \otimes a_{1j} + a_{i2} \otimes a_{2j},
$$
 is a homomorphism of Poisson algebras. This can be verified by a direct calculation. 
%
\end{example}

\paragraph{Standard Poisson structure on a reductive Lie group.} Here we recall (see, e,g, \cite[Section 1.2]{hoffmann2000factorization} and \cite[Section 3.1]{reshetikhin2003integrability}) the definition of  the {standard} Poisson structure on an arbitrary complex reductive Lie group. Let $\g$ be a complex reductive Lie algebra, and $G$ be a Lie group with Lie algebra $\g$. To define the standard Poisson structure on $G$, one needs to fix a Cartan subalgebra $\mathfrak t \subset \g$, a system of positive roots, and an invariant inner product. Assuming that $[\g, \g]$ is simple, different choices lead to Poisson structures related by an inner automorphism and/or rescaling. 
Furthermore, for $G=\GL{n}$ (with a fixed defining representation), there is a canonical choice: the Cartan subalgebra is given by diagonal matrices, positive roots are those whose associated root vectors are upper-triangular matrices, and the inner product is given by the trace form $\tr(XY)$. Moreover, all the same choices work in types $B$, $C$, $D$ if we realize the corresponding groups by operators preserving a given symmetric or skew-symmetric \emph{anti-diagonal} bilinear form $\Omega$. All groups considered in this paper are defined in this way and hence carry a canonical Poisson structure.

Upon fixing the above data, consider the corresponding nilpotent subalgebras $\mathfrak n_\pm$. Any $\xi \in \g$ can be written as $\xi = \xi_+ + \xi_0 + \xi_-$, where $\xi_\pm \in \mathfrak n_\pm$, $\xi_0 \in \mathfrak t$ (recall that $\mathfrak t$ denotes the Cartan subalgebra). 
The operator $r \colon \g \to \g$ defined by $r(\xi) :=\xi_+ - \xi_-$ is called the \emph{standard $r$-matrix}. The \emph{standard Poisson bracket} on $G$ is given by
\begin{equation*}
\{f_1, f_2\} (g) := \frac{1}{2}\left(\langle r(g^{-1} \grad f_1(g)),g^{-1}\grad f_2(g) \rangle - \langle r( \grad f_1(g)g^{-1}),\grad f_2(g)g^{-1} \rangle\right)
\end{equation*}
for any regular functions $f_1, f_2 \in \C[G]$ on $G$ and any $g \in G$. Here $\langle \,, \rangle$ is the invariant inner product on $\g$, the gradients $ \grad f_i(g) \in T_gG$ are defined using the bi-invariant metric on $G$ determined by that inner product, while $g^{-1} \grad f_i(g) $ and $ \grad f_i(g)g^{-1} $ are the gradients translated to the identity using left and right shifts respectively. 
The so-defined  Poisson bracket is multiplicative, thus turning $G$ into a {Poisson-Lie group}. 



\paragraph{Induced Poisson structure on the fixed locus of a Poisson involution.}

We now recall (see, e.g., \cite{fernandes2001hyperelliptic, xu2003dirac}) the construction of the induced Poisson structure on the fixed point set of a Poisson involution.  In what follows, we will use that construction to deduce the Poisson property of the boundary measurement map for $B$ and $C$ types from the corresponding result in type $A$. 



Let $X$ be an affine Poisson variety, $\tau \colon X \to X$ be a biregular involution, and $Y := \mathrm{Fix}(\tau) = \{x \in X \mid \tau(x) = x\}$ be the fixed point set of $\tau$. Then, generally speaking, $Y$ is not a Poisson subvariety of $X$, which, in the language of functions, means that the set {{}$\I(Y) \subset\C[X]$ of regular functions on $X$ vanishing on $Y$ is not a Poisson ideal. However, this does hold for $\tau$-invariant functions, namely $\I(Y)^\tau := \{ f \in \I(Y) \mid \tau^*f = f\}$ is a Poisson ideal in $\C[X]^\tau = \{ f \in\C[X] \mid \tau^*f = f\}$, see  \cite[Proof of Proposition 3.4]{fernandes2001hyperelliptic}. Identifying the coordinate ring $\C[Y]$ with the quotient  $\C[X]^\tau / \,\I(Y)^\tau$, one gets the \emph{induced} Poisson structure on $Y$. This construction is natural in the sense that going from $(X, \tau)$ to the fixed point set of $\tau$ with its induced Poisson structure is a functor from the category of affine Poisson varieties with a Poisson involution to the category of affine Poisson varieties. In other words, we have the following.

\begin{proposition}\label{prop:functor}
Suppose we have a commutative diagram of affine Poisson varieties and regular Poisson maps
\begin{equation*}
\begin{tikzcd}
X_1 \arrow[d,"\psi", swap] \arrow[r,"\tau_1"]
&X_1 \arrow[d, "\psi"]  \\ 
X_2  \arrow[r,"\tau_2"] &
X_2,
\end{tikzcd}
\end{equation*}
where horizontal arrows are involutions. For $i  \in \{1,2\}$, let $Y_i = \mathrm{Fix}(\tau_i)\subset X_i$ be the fixed locus of $\tau_i$, endowed with the induced Poisson structure. Then the map $\psi\vert_{Y_1} \colon Y_1 \to Y_2$ is Poisson.
\end{proposition}
\begin{proof}
Since $\psi^* \colon\C[X_2] \to\C[X_1]$ is a homomorphism of Poisson algebras, so is the induced map $\C[X_2]^{\tau_2} / \,\I(Y_2)^{\tau_2} \to\C[X_1]^{\tau_1} /\, \I(Y_1)^{\tau_1}$. 
\end{proof}

\paragraph{Standard Poisson structures in types $B$, $C$, $D$ from type $A$.} \begin{proposition}\label{PLsubgroup} Suppose $G$ is a Poisson-Lie group equipped with an involutive Poisson automorphism $\tau$. Then the fixed point subgroup $ \{ g \in G \mid \tau(g) = g\}$ with the induced Poisson structure is also a Poisson-Lie group.\end{proposition}\begin{proof}
Consider the commutative diagram
\begin{equation*}
\begin{tikzcd}[column sep=2.9cm]
G \times G \arrow[d,"(g_1{,}\, g_2)\, \mapsto\, g_1g_2", swap] \arrow[r,"(g_1{,}\, g_2)\, \mapsto \,(\tau(g_1){,}\, \tau(g_2))"]
&G \times G \arrow[d, "(g_1{,}\, g_2)\, \mapsto\, g_1g_2"]  \\ 
G   \arrow[r,"\tau"] &
G
\end{tikzcd}
\end{equation*}
and apply Proposition \ref{prop:functor}.
\end{proof}
 This result makes it possible to recover standard Poisson structures  in types $B$, $C$, $D$ from the one in type~$A$.

\begin{proposition}\label{prop:induce}
Let $\Omega$ be a non-degenerate symmetric or skew-symmetric anti-diagonal matrix, and let $\tau \colon \GL{n} \to \GL{n}$ be given by $\tau(A) := \Omega A^{-t} \Omega^{-1}$. Then $\tau$ preserves the standard Poisson structure on $\GL{n}$, and the induced Poisson structure on the fixed point set $G(\Omega)$ of $\tau$ coincides with the standard Poisson structure on  $G(\Omega)$.
\end{proposition}
\begin{proof}
The differential $d\tau(\Id) \colon  \GL{n} \to \GL{n}$ of $\tau$ at the identity reads $A \mapsto -\Omega A^{t} \Omega^{-1}$. Since transposition switches upper- and lower-triangular matrices, while conjugation by $\Omega$ switches them back, this map commutes with the $r$-matrix. Furthermore, since the differential of $\tau$ preserves the trace form, it also preserves the bivector $r^{\sharp}$ obtained from the operator $r$ by means of identification $\gl_{n} \simeq \gl_{n}^*$ via the trace form. Any automorphism preserving $r^{\sharp}$ preserves the associated Poisson bracket \cite[Proposition 2.12]{izosimov2022lattice}, so $\tau$ is indeed a Poisson map.


Further, observe that, being an automorphism, $\tau$ preserves the bi-invariant metric on $\GL{n}$. So, for any fixed point $A \in G(\Omega)$ of $\tau$ we have an orthogonal vector space decomposition $T_A \GL{n} = T_AG(\Omega) \oplus V_A$, where $V_A$ is the $(-1)$-eigenspace  of $d\tau(A)$. 
Take any $\tau$-invariant function $f \in \C[\GL{n}]^\tau$. Then the differential $df(A) \in T^*_A \GL{n}$ belongs to the $(+1)$-eigenspace  of $(d\tau(A))^*$ and hence vanishes on the subspace $V_A \subset T_A\GL{n}$. But that means that the gradient $\grad f(A)$ is orthogonal to $V_A$ and hence tangent to $G(\Omega)$. Therefore, that gradient is equal to the gradient at $A$ of the restriction  $f\vert_{G(\Omega)}$. So, the standard bracket of two $\tau$-invariant functions $f_1, f_2 \in \C[\GL{n}]^\tau$ computed at a point $A \in G(\Omega)$ is equal to the standard bracket of their restrictions to $G(\Omega)$, computed at the same point (here we also use that the restriction of the standard $r$-matrix from $\gl_{n}$ to the Lie algebra $\g(\Omega) $ of $G(\Omega)$ gives the standard $r$-matrix on the latter, and the same holds for the invariant inner product). But that precisely means that the standard bracket on $\GL{n}$ induces the standard bracket on $G(\Omega)$, q.e.d.
\end{proof}

\paragraph{Construction of the Poisson bracket on move-symmetric weightings.}
  \begin{figure}[t]
 \centering
\begin{tikzpicture}[scale = 0.7]
\node () at (0,0){
\begin{tikzpicture}[scale = 0.5]
\node [draw,circle,color=black, fill=black,inner sep=0pt,minimum size=3pt, opacity = 0.3] (A) at (0.3,0) {};
\node [draw,circle,color=black, fill=white,inner sep=0pt,minimum size=3pt, opacity = 0.3] (B) at (1,0.7) {};
\node [draw,circle,color=black, fill=black,inner sep=0pt,minimum size=3pt, opacity = 0.3] (C) at (1.7,0) {};
\node [draw,circle,color=black, fill=white,inner sep=0pt,minimum size=3pt, opacity = 0.3] (D) at (1,-0.7) {};
\node [draw,circle,color=black, fill=white,inner sep=0pt,minimum size=3pt, opacity = 0.3] (E) at (1,2) {};
\node [draw,circle,color=black, fill=black,inner sep=0pt,minimum size=3pt, opacity = 0.3] (F) at (1,-2) {};
 \node [draw,circle,color=black, fill=black,inner sep=0pt,minimum size=3pt]  (1) at (1,0){};
  \node [draw,circle,color=black, fill=black,inner sep=0pt,minimum size=3pt]  (1') at (4,0){};
 \node [draw,circle,color=black, fill=black,inner sep=0pt,minimum size=3pt]  (Y2) at (-0.5,1.25){};
  \node [draw,circle,color=black, fill=black,inner sep=0pt,minimum size=3pt]  (2Y2) at (-0.5,-1.25){};
   \node [draw,circle,color=black, fill=black,inner sep=0pt,minimum size=3pt]  (Y3) at (2.5,1.25){};
    \node [draw,circle,color=black, fill=black,inner sep=0pt,minimum size=3pt]  (12Y3) at (2.5,-1.25){};
       \node [draw,circle,color=black, fill=black,inner sep=0pt,minimum size=3pt]  (Y4) at (5.5,1.25){};
    \node [draw,circle,color=black, fill=black,inner sep=0pt,minimum size=3pt]  (12Y4) at (5.5,-1.25){};
     \node [draw,circle,color=black, fill=black,inner sep=0pt,minimum size=3pt]  (Y1) at (2.5,2.75){};
        \node [draw,circle,color=black, fill=black,inner sep=0pt,minimum size=3pt]  (Y1') at (2.5,-2.75){};
        \draw [dotted, ->-, thick] (Y1) to [bend right]  (Y2);
                \draw [->, thick] (Y3) -- (Y1);
                     \draw [dotted, ->, thick] (2Y2)  to [bend left = 15] (Y2);
                             \draw [dotted, ->-,thick ] (Y1') to [bend left]  (2Y2);
                                   \draw [->, thick] (12Y3) -- (Y1');
                                           \draw [dotted, ->-, thick] (Y1) to [bend left]  (Y4);
                     \draw [dotted, ->,thick ] (Y4)  to [bend left = 15]  (12Y4);
                             \draw [dotted, ->-, thick] (Y1') to [bend right]  (12Y4);
                                                     \draw [ ->-,thick ] (1) -- (Y3);
                                                        \draw [ ->-,thick ] (1) -- (2Y2);
                                                           \draw [ ->-, thick] (Y2) -- (1);
                                                                          \draw [ ->-, thick] (12Y3) -- (1);
                                                                                                                               \draw [ ->-,thick ] (1') -- (Y4);
                                                        \draw [ ->-, thick] (1') -- (12Y3);
                                                           \draw [ ->-, thick] (Y3) -- (1');
                                                                          \draw [ ->-, thick] (12Y4) -- (1');
                                                                                  \draw [->, thick] (Y4)-- (Y3);
                                                                                             \draw [->, thick] (2Y2) --  (12Y3);
\draw [opacity = 0.3] (A) -- (B) -- (C) -- (D) -- (A);
\draw  [opacity = 0.3](A) -- +(-1.8,0);
\draw  [opacity = 0.3](B) -- (E);
\draw  [opacity = 0.3](E) -- +(2.5,0);
\draw  [opacity = 0.3](E) -- +(-2.5,0);
\draw [opacity = 0.3] (C) -- +(1,0);
\draw  [opacity = 0.3](D) -- (F);
\draw [opacity = 0.3] (F) -- +(2.5,0);
\draw  [opacity = 0.3](F) -- +(-2.5,0);
\def\n {-2}
\node [draw,circle,color=black, fill=black,inner sep=0pt,minimum size=3pt, opacity = 0.3] (A) at (5.3+\n,0) {};
\node [draw,circle,color=black, fill=white,inner sep=0pt,minimum size=3pt, opacity = 0.3] (B) at (6 + \n,0.7) {};
\node [draw,circle,color=black, fill=black,inner sep=0pt,minimum size=3pt, opacity = 0.3] (C) at (6.7 + \n,0) {};
\node [draw,circle,color=black, fill=white,inner sep=0pt,minimum size=3pt, opacity = 0.3] (D) at (6 +\n,-0.7) {};
\node  [draw,circle,color=black, fill=black,inner sep=0pt,minimum size=3pt, opacity = 0.3] (E) at (6+\n,2) {};
\node  [draw,circle,color=black, fill=white,inner sep=0pt,minimum size=3pt, opacity = 0.3] (F) at (6+\n,-2) {};
\draw [opacity = 0.3] (A) -- (B) -- (C) -- (D) -- (A);
\draw  [opacity = 0.3](A) -- +(-1,0);
\draw  [opacity = 0.3](B) -- (E);
\draw  [opacity = 0.3](E) -- +(2.5,0);
\draw  [opacity = 0.3](E) -- +(-2.5,0);
\draw [opacity = 0.3] (C) -- +(1.8,0);
\draw  [opacity = 0.3](D) -- (F);
\draw [opacity = 0.3] (F) -- +(2.5,0);
\draw  [opacity = 0.3](F) -- +(-2.5,0);
\draw [dashed] (-1.5,-3.5) -- (-1.5,3.5) -- (6.5,3.5) -- (6.5,-3.5) -- cycle;
\end{tikzpicture}};

\node () at (9,0) {

\begin{tikzpicture}[xscale = 0.5, yscale = -0.5]
\node [draw,circle,color=black, fill=white,inner sep=0pt,minimum size=3pt, opacity = 0.3] (A) at (0.3,0) {};
\node [draw,circle,color=black, fill=black,inner sep=0pt,minimum size=3pt, opacity = 0.3] (B) at (1,0.7) {};
\node [draw,circle,color=black, fill=white,inner sep=0pt,minimum size=3pt, opacity = 0.3] (C) at (1.7,0) {};
\node [draw,circle,color=black, fill=black,inner sep=0pt,minimum size=3pt, opacity = 0.3] (D) at (1,-0.7) {};
\node [draw,circle,color=black, fill=black,inner sep=0pt,minimum size=3pt, opacity = 0.3] (E) at (1,2) {};
\node [draw,circle,color=black, fill=white,inner sep=0pt,minimum size=3pt, opacity = 0.3] (F) at (1,-2) {};
 \node [draw,circle,color=black, fill=black,inner sep=0pt,minimum size=3pt]  (1) at (1,0){};
  \node [draw,circle,color=black, fill=black,inner sep=0pt,minimum size=3pt]  (1') at (4,0){};
 \node [draw,circle,color=black, fill=black,inner sep=0pt,minimum size=3pt]  (Y2) at (-0.5,1.25){};
  \node [draw,circle,color=black, fill=black,inner sep=0pt,minimum size=3pt]  (2Y2) at (-0.5,-1.25){};
   \node [draw,circle,color=black, fill=black,inner sep=0pt,minimum size=3pt]  (Y3) at (2.5,1.25){};
    \node [draw,circle,color=black, fill=black,inner sep=0pt,minimum size=3pt]  (12Y3) at (2.5,-1.25){};
       \node [draw,circle,color=black, fill=black,inner sep=0pt,minimum size=3pt]  (Y4) at (5.5,1.25){};
    \node [draw,circle,color=black, fill=black,inner sep=0pt,minimum size=3pt]  (12Y4) at (5.5,-1.25){};
     \node [draw,circle,color=black, fill=black,inner sep=0pt,minimum size=3pt]  (Y1) at (2.5,2.75){};
        \node [draw,circle,color=black, fill=black,inner sep=0pt,minimum size=3pt]  (Y1') at (2.5,-2.75){};
        \draw [dotted, ->-, thick] (Y1) to [bend right]  (Y2);
                \draw [->, thick] (Y3) -- (Y1);
                     \draw [dotted, ->, thick] (2Y2)  to [bend left = 15] (Y2);
                             \draw [dotted, ->-,thick ] (Y1') to [bend left]  (2Y2);
                                   \draw [->, thick] (12Y3) -- (Y1');
                                           \draw [dotted, ->-, thick] (Y1) to [bend left]  (Y4);
                     \draw [dotted, ->,thick ] (Y4)  to [bend left = 15]  (12Y4);
                             \draw [dotted, ->-, thick] (Y1') to [bend right]  (12Y4);
                                                     \draw [ ->-,thick ] (1) -- (Y3);
                                                        \draw [ ->-,thick ] (1) -- (2Y2);
                                                           \draw [ ->-, thick] (Y2) -- (1);
                                                                          \draw [ ->-, thick] (12Y3) -- (1);
                                                                                                                               \draw [ ->-,thick ] (1') -- (Y4);
                                                        \draw [ ->-, thick] (1') -- (12Y3);
                                                           \draw [ ->-, thick] (Y3) -- (1');
                                                                          \draw [ ->-, thick] (12Y4) -- (1');
                                                                                  \draw [->, thick] (Y4)-- (Y3);
                                                                                             \draw [->, thick] (2Y2) --  (12Y3);
\draw [opacity = 0.3] (A) -- (B) -- (C) -- (D) -- (A);
\draw  [opacity = 0.3](A) -- +(-1.8,0);
\draw  [opacity = 0.3](B) -- (E);
\draw  [opacity = 0.3](E) -- +(2.5,0);
\draw  [opacity = 0.3](E) -- +(-2.5,0);
\draw [opacity = 0.3] (C) -- +(1,0);
\draw  [opacity = 0.3](D) -- (F);
\draw [opacity = 0.3] (F) -- +(2.5,0);
\draw  [opacity = 0.3](F) -- +(-2.5,0);
\def\n {-2}
\node [draw,circle,color=black, fill=white,inner sep=0pt,minimum size=3pt, opacity = 0.3] (A) at (5.3+\n,0) {};
\node [draw,circle,color=black, fill=black,inner sep=0pt,minimum size=3pt, opacity = 0.3] (B) at (6 + \n,0.7) {};
\node [draw,circle,color=black, fill=white,inner sep=0pt,minimum size=3pt, opacity = 0.3] (C) at (6.7 + \n,0) {};
\node [draw,circle,color=black, fill=black,inner sep=0pt,minimum size=3pt, opacity = 0.3] (D) at (6 +\n,-0.7) {};
\node  [draw,circle,color=black, fill=white,inner sep=0pt,minimum size=3pt, opacity = 0.3] (E) at (6+\n,2) {};
\node  [draw,circle,color=black, fill=black,inner sep=0pt,minimum size=3pt, opacity = 0.3] (F) at (6+\n,-2) {};
\draw [opacity = 0.3] (A) -- (B) -- (C) -- (D) -- (A);
\draw  [opacity = 0.3](A) -- +(-1,0);
\draw  [opacity = 0.3](B) -- (E);
\draw  [opacity = 0.3](E) -- +(2.5,0);
\draw  [opacity = 0.3](E) -- +(-2.5,0);
\draw [opacity = 0.3] (C) -- +(1.8,0);
\draw  [opacity = 0.3](D) -- (F);
\draw [opacity = 0.3] (F) -- +(2.5,0);
\draw  [opacity = 0.3](F) -- +(-2.5,0);
\draw [dashed] (-1.5,-3.5) -- (-1.5,3.5) -- (6.5,3.5) -- (6.5,-3.5) -- cycle;
\end{tikzpicture}};

\end{tikzpicture}
\caption{ If two plabic graphs are related by reflection followed by vertex recoloring, then their dual quivers are reflections of each other.}\label{fig:mirror}
    
\end{figure}
Let $\Gamma$ be an $n\times n$  move-symmetric plabic graph, and $ \F^{ms}(\Gamma)$ be the space of move-symmetric weightings on $\Gamma$. We aim to equip the latter space with a multiplicative Poisson structure such that the boundary measurement map   $ \F^{ms}(\Gamma)\dashrightarrow G(\Omega_n) $ is Poisson. To that end, recall that $ \F^{ms}(\Gamma)$ is the fixed point set of a birational involution $\sigma \colon \F_{}(\Gamma) \dashedrightarrow \F_{}(\Gamma)$ defined in Section \ref{sec2}. The involution $\sigma$ can be represented as composition of square moves turning the graph $\Gamma$ into the graph $\Gamma'$ and reflection which maps $\Gamma'$ back onto $\Gamma$ but with opposite color vertices. Since the reflection map reverses both the orientation and vertex colors, it maps the dual quiver $\Q'$ of $\Gamma'$ isomorphically onto the dual quiver $\Q$ of $\Gamma$, see Figure \ref{fig:mirror}.  
 So, since the Poisson structure on the face weight space is determined by the quiver, the induced map on face weights is Poisson. Given that square moves are Poisson as well, it follows that $\sigma$ is a birational Poisson involution. Moreover, $\sigma$ is biregular when viewed as a map $  \F_0(\Gamma) \to \F_0(\Gamma) $, where $\F_{0}(\Gamma) \subset  \F_{}(\Gamma)$ is an open subvariety defined by the condition that the faces where we perform square moves have weights not equal to $-1$. So, the fixed point set $ \F^{ms}(\Gamma) \subset \F_0(\Gamma)$ of $\sigma$ carries an induced Poisson structure. 
 
 Now, we show that the so-constructed Poisson structure on $ \F^{ms}(\Gamma)$ is multiplicative. Let $\Gamma_1$, $\Gamma_2$ be $n \times n$ move-symmetric plabic graphs, and let $\Gamma_1 \times \Gamma_2$ be their concatenation. Then we have a Poisson map $ \F_{}(\Gamma_1) \times  \F_{}(\Gamma_2) \to  \F_{}(\Gamma_1 \times \Gamma_2)$ which sends weightings $\mathcal Y_1 \in \F_{}(\Gamma_1), \mathcal Y_2 \in \F_{}(\Gamma_2) $ to the weighting $\mathcal Y_1\mathcal Y_2 \in \F_{}(\Gamma_1 \times \Gamma_2)$ obtained by concatenating weighted graphs $(\Gamma_1, \mathcal Y_1 )$, $(\Gamma_2, \mathcal Y_2 )$. As before, for a move-symmetric graph $\Gamma$, denote by $\F_0(\Gamma)$ the set of weightings such that the weights of faces with two vertices on the midline are not equal to $-1$. Then we have the following commutative diagram
 \begin{equation*}
\begin{tikzcd}[column sep=3cm]
 \F_{0}(\Gamma_1) \times  \F_{0}(\Gamma_2)  \arrow[d,"(\mathcal Y_1{,}\, \mathcal Y_2)\, \mapsto\, \mathcal Y_1\mathcal Y_2", swap] \arrow[r,"(\mathcal Y_1{,}\, \mathcal Y_2)\, \mapsto \,(\sigma(\mathcal Y_1){,}\, \sigma(\mathcal Y_2))"]
& \F_{0}(\Gamma_1) \times  \F_{0}(\Gamma_2)\arrow[d, "(\mathcal Y_1{,}\, \mathcal Y_2)\, \mapsto\, \mathcal Y_1\mathcal Y_2"]  \\ 
\F_{0}(\Gamma_1 \times \Gamma_2)   \arrow[r,"\sigma"] &
\F_{0}(\Gamma_1 \times \Gamma_2)  .
\end{tikzcd}
\end{equation*}
So, by Proposition \ref{prop:functor}, the concatenation map $\F^{ms}(\Gamma_1) \times  \F^{ms}(\Gamma_2) \to  \F^{ms}(\Gamma_1 \times \Gamma_2)$ is also Poisson, as needed (cf. the proof of Proposition \ref{PLsubgroup}).

Finally, let us show that, for a move-symmetric  perfectly orientable $n\times n$ plabic graph $\Gamma$, the boundary measurement map $ \F^{ms}(\Gamma) \dashrightarrow G(\Omega_n)$ is Poisson (away from its indeterminacy locus). Let $\F_{00}(\Gamma) \subset \F_0(\Gamma)$ be the subset obtained by removing the  indeterminacy locus of the boundary measurement map. Then, from \eqref{mainDiag}, we get the following commutative diagram:
\begin{equation*}
\begin{tikzcd}
\F_{00}(\Gamma) \arrow[d,"\mes", swap] \arrow[r,"\sigma"]
&\F_{00}( \Gamma) \arrow[d, "\mes"]  \\ 
\GL{n}  \arrow[r,"\tau"] &
\GL{n}.
\end{tikzcd}
\end{equation*}
Applying Proposition \ref{prop:functor}, we get that the restriction of the boundary measurement map $\mes$ to the fixed point set of $\sigma$ is a Poisson map to the fixed point set of $\tau$. Put differently, the pushforward of the induced Poisson structure on $\mathrm{Fix}(\sigma) = \F^{ms}(\Gamma)$ by the boundary measurement map $\mes \colon  \F^{ms}(\Gamma) \dashrightarrow G(\Omega_n)$ is the induced Poisson structure on $G(\Omega_n) = \mathrm{Fix}(\tau)$. At the same time, by Proposition~\ref{prop:induce}, the induced Poisson structure on $G(\Omega_n)$ coincides with the standard Poisson structure, hence the result.

 
\medskip
\section{Calculation of the Poisson bracket on move-symmetric weightings}\label{sec:pbc}
Let $\Gamma$ be an $n \times n$  move-symmetric plabic graph, and $ \F^{ms}(\Gamma)$ be the space of move-symmetric weightings on $\Gamma$. We will now explicitly compute the Poisson bracket on $ \F^{ms}(\Gamma)$ constructed in the previous section. In particular, we will show that this bracket is log-canonical in terms of face weights, thus completing the proof of Theorem~\ref{thm1}. 

\paragraph{The symplectic case.} First assume that $n$ is even, so that the associated boundary measurement matrix is symplectic. Then $\Gamma$ is necessarily symmetric. Let $I$ be an index set labeling faces of $\Gamma$, 
and let $I_- := \{ i \in I \mid i\mbox{th face is below the midline}\}$, $I_+ := \{ i \in I \mid i\mbox{th face is above the midline}\}$,  $I_0 := \{ i \in I \mid i\mbox{th face is on the midline}\}$. When saying that a face lies on the midline, we mean that its interior has a non-empty intersection with that line. For any $i \in I$, let $i' \in I$ be such that the face $i$, $i'$ are symmetric to each other with respect to the midline. 

{On $ \F^{ms}(\Gamma)$, we have $y_{i'} = y_i$. So, 
the coordinate ring $\C[ \F^{ms}(\Gamma)] $ of  $ \F^{ms}(\Gamma)$ is generated by 
$  y_i^{\pm 1}$, $i \in I_0 \sqcup I_+$, subject to the relation $\prod\nolimits_{i \in I_0}y_i \prod\nolimits_{i \in I_+}y_i^2= 1$.
Therefore, to describe a Poisson structure on $\F^{ms}(\Gamma)$, it suffices to find pairwise Poisson brackets of those $y_i$. 
Denote by $q_{ij}$ the $(i,j)$ entry of the adjacency matrix of the quiver $\Q$ dual to $\Gamma$ (note that the vertices of $\Q$ are faces of $\Gamma$ and so are also indexed by $I$). 

\begin{proposition}\label{prop:folding}
The Poisson bracket on $ \F^{ms}(\Gamma)$ constructed in Section \ref{sec3} reads
\begin{equation}\label{eq:folding}
\begin{gathered}
\{ y_i,  y_j\} = q_{ij} y_i  y_j, \quad \mbox{if $i \in I_0$ or $j \in I_0$},\quad
\{ y_i,  y_j\} = \frac{1}{2}q_{ij} y_i  y_j,  \quad \mbox{if }  i,j \in I_+.
\end{gathered}
\end{equation}
\end{proposition}
\begin{remark}
  \begin{figure}
 \centering
\begin{tikzpicture}[scale = 1.1]
\centering
\node () at (0,0) {
\begin{tikzpicture}[scale = 1.0]
\node [draw,circle,color=black, fill=black,inner sep=0pt,minimum size=3pt, opacity = 0.3] (A) at (0,1) {};
\node [draw,circle,color=black, fill=white,inner sep=0pt,minimum size=3pt, opacity = 0.3] (B) at (1,1) {};
\node [draw,circle,color=black, fill=white,inner sep=0pt,minimum size=3pt, opacity = 0.3] (C) at (0,2) {};
\node [draw,circle,color=black, fill=black,inner sep=0pt,minimum size=3pt, opacity = 0.3] (D) at (1,2) {};
\node [draw,circle,color=black, fill=black,inner sep=0pt,minimum size=3pt, opacity = 0.3] (E) at (0,3) {};
\node [draw,circle,color=black, fill=white,inner sep=0pt,minimum size=3pt, opacity = 0.3] (F) at (1,3) {};
\node [draw,circle,color=black, fill=white,inner sep=0pt,minimum size=3pt, opacity = 0.3] (G) at (0,0) {};
\node [draw,circle,color=black, fill=black,inner sep=0pt,minimum size=3pt, opacity = 0.3] (H) at (1,0) {};
 \node [draw,circle,color=black, fill=black,inner sep=0pt,minimum size=3pt, label = {[label distance=-2]  above: \footnotesize$ y_1$}]  (Y1) at (0.5,3.3){};
\node  [draw,circle,color=black, fill=black,inner sep=0pt,minimum size=3pt, label = {[label distance=-2]  left: \footnotesize$ y_2$}]  (Y2) at (-0.5,2.5){ };
\node  [draw,circle,color=black, fill=black,inner sep=0pt,minimum size=3pt, label = {[label distance=-2]  45: \footnotesize$ y_3$}]  (Y3) at(0.5,2.5) {};
\node  [draw,circle,color=black, fill=black,inner sep=0pt,minimum size=3pt, label = {[label distance=-2]  right: \footnotesize$ y_4$}]  (Y4) at (1.5,2.5) {};
\node [draw,circle,color=black, fill=black,inner sep=0pt,minimum size=3pt, label = {[label distance=-2]  left: \footnotesize$ y_5$}]  (Y5) at (-0.5,1.5){ };
\node  [draw,circle,color=black, fill=black,inner sep=0pt,minimum size=3pt, label = {[label distance=-2]  45: \footnotesize$ y_6$}]  (Y6) at(0.5,1.5) {};
\node   [draw,circle,color=black, fill=black,inner sep=0pt,minimum size=3pt, label = {[label distance=-2]  right: \footnotesize$ y_7$}] (Y7) at (1.5,1.5) {};
\node [draw,circle,color=black, fill=black,inner sep=0pt,minimum size=3pt, label = {[label distance=-2]  left: \footnotesize$ y_2$}] (Y2') at (-0.5,0.5){ };
\node  [draw,circle,color=black, fill=black,inner sep=0pt,minimum size=3pt, label = {[label distance=-2]  45: \footnotesize$ y_3$}]  (Y3') at(0.5,0.5) {};
\node  [draw,circle,color=black, fill=black,inner sep=0pt,minimum size=3pt, label = {[label distance=-2]  right: \footnotesize$ y_4$}]  (Y4') at (1.5,0.5) {};
\node [draw,circle,color=black, fill=black,inner sep=0pt,minimum size=3pt, label = {[label distance=-2]  below: \footnotesize$ y_1$}]  (Y1') at(0.5,-0.3){};
\draw [thick, ->] (Y1) -- (Y3);
\draw [thick, ->] (Y7) -- (Y6);
\draw [thick, ->] (Y6) -- (Y3);
\draw [thick, ->] (Y3) -- (Y4);
\draw [thick, ->, dotted] (Y4) -- (Y7);
\draw [thick, ->] (Y3) -- (Y2);
\draw [thick, ->, dotted] (Y2) -- (Y5);
\draw [thick, ->] (Y5) -- (Y6);
\draw [thick,dotted,->] (Y2) to [bend left] (Y1);
\draw [thick,dotted, ->] (Y4) to [bend right]  (Y1);

\draw [thick, ->] (Y1') -- (Y3');
\draw [thick, ->] (Y6) -- (Y3');
\draw [thick, ->] (Y3') -- (Y4');
\draw [thick, ->, dotted] (Y4') -- (Y7);
\draw [thick, ->] (Y3') -- (Y2');
\draw [thick, ->, dotted] (Y2') -- (Y5);
\draw [thick,dotted,->] (Y2') to [bend right] (Y1');
\draw [thick,dotted, ->] (Y4') to [bend left]  (Y1');
\draw[ opacity = 0.3]  (-1, 1) -- (A)   node[midway, below] {};
\draw[ opacity = 0.3]  (B) -- (A) -- (G) -- (H) -- (B);
\draw[ opacity = 0.3]  (A) -- (C) node[midway, left] {};
\draw[ opacity = 0.3] (C) -- (D) -- (F) -- (E) -- (C);
\draw[ opacity = 0.3]   (C) -- +(-1,0) node[midway, above] {};
\draw[ opacity = 0.3]  (E) -- +(-1,0) node[midway, above] {};
\draw[ opacity = 0.3]  (G) -- +(-1,0) node[midway, above] {};
\draw[ opacity = 0.3]  (B) -- +(1,0) node[midway, below] {};
\draw[ opacity = 0.3]  (F) -- +(1,0) node[midway, below] {};
\draw[ opacity = 0.3] (H) -- +(1,0) node[midway, below] {};
\draw[ opacity = 0.3]  (D) -- (2,2) node[midway, above] {};
\draw[ opacity = 0.3]  (D) -- (B) node[midway, left] {};

\draw [dashed, opacity = 0.3] (-1,3.8) -- (2,3.8);
\draw [dashed, opacity = 0.3] (2,-0.8) -- (2,3.8) ; 
\draw [ dashed, opacity = 0.3] (-1,-0.8) -- (2,-0.8) ;
\draw [ dashed, opacity = 0.3]  (-1,-0.8) -- (-1,3.8);

\end{tikzpicture}};
\node  () at (3,0) {$\longrightarrow$};
\node () at (6,0) {
\begin{tikzpicture}[scale = 1.0]
 \node [draw,circle,color=black, fill=black,inner sep=0pt,minimum size=3pt, label = {[label distance=-2]  above: \footnotesize$ y_1$}]  (Y1) at (0.5,3.3){};
\node  [draw,circle,color=black, fill=black,inner sep=0pt,minimum size=3pt, label = {[label distance=-2]  left: \footnotesize$ y_2$}]  (Y2) at (-0.5,2.5){ };
\node  [draw,circle,color=black, fill=black,inner sep=0pt,minimum size=3pt, label = {[label distance=-2]  45: \footnotesize$ y_3$}]  (Y3) at(0.5,2.5) {};
\node  [draw,circle,color=black, fill=black,inner sep=0pt,minimum size=3pt, label = {[label distance=-2]  right: \footnotesize$ y_4$}]  (Y4) at (1.5,2.5) {};
\node [draw,circle,color=black, fill=black,inner sep=0pt,minimum size=3pt, label = {[label distance=-2]  left: \footnotesize$ y_5$}]  (Y5) at (-0.5,1.5){ };
\node  [draw,circle,color=black, fill=black,inner sep=0pt,minimum size=3pt, label = {[label distance=-2]  below: \footnotesize$ y_6$}]  (Y6) at(0.5,1.5) {};
\node   [draw,circle,color=black, fill=black,inner sep=0pt,minimum size=3pt, label = {[label distance=-2]  right: \footnotesize$ y_7$}] (Y7) at (1.5,1.5) {};
\draw [thick, ->] (Y1) -- (Y3);
\draw [thick, ->-] (Y7) to [bend left = 20] (Y6);
\draw [thick, ->-] (Y6) to [bend left = 20] (Y3);
\draw [thick, ->-] (Y7) to [bend right = 20] (Y6);
\draw [thick, ->-] (Y6) to [bend right = 20] (Y3);
\draw [thick, ->] (Y3) -- (Y4);
\draw [thick, ->] (Y4) -- (Y7);
\draw [thick, ->] (Y3) -- (Y2);
\draw [thick, ->] (Y2) -- (Y5);
\draw [thick, ->-] (Y5)  to [bend left = 20]  (Y6);
\draw [thick, ->-] (Y5)  to [bend right = 20]  (Y6);
\draw [thick,dotted,->] (Y2) to [bend left] (Y1);
\draw [thick,dotted, ->] (Y4) to [bend right]  (Y1);

\end{tikzpicture}};
\end{tikzpicture}
\caption{The quiver defining a Poisson bracket on symmetric weightings is constructed by means of folding. The  parameters $y_1, \dots, y_7$ are subject to the relation $(y_1y_2y_3y_4)^2y_5y_6y_7 = 1$.}\label{fig:fold}
    
\end{figure}
One can understand this bracket as a \emph{folding} of the bracket on $ \F(\Gamma)$. In the quiver language, to obtain the bracket on $ \F^{ms}(\Gamma)$, one folds the dual quiver $\Q$  of $\Gamma$ along the midline and then doubles all the arrows incident to vertices on the midline, see Figure~\ref{fig:fold}. This gives a Poisson structure on  $ \F^{ms}(\Gamma)$  which is twice the bracket  \eqref{eq:folding}. This procedure is analogous to Dynkin diagram folding $A_{2k-1} \to C_k$ and, in the cluster algebra setting, can be understood as exchange matrix folding followed by skew-symmetrization, cf.~\cite{felikson2012cluster}.
\end{remark}


\begin{proof}[\rm\textbf{Proof of Proposition \ref{prop:folding}}]
By construction, the Poisson bracket of two functions on $ \F^{ms}(\Gamma)$ can be found by extending them to $\sigma$-invariant functions on $\F_{}(\Gamma)$, taking the Poisson bracket of extensions, and then restricting the result back to  $ \F^{ms}(\Gamma)$. Let $\{\,,\}$ be the bracket on $\F^{ms}(\Gamma)$ and   $\{\,,\}'$ be the bracket on $\F_{}(\Gamma)$. Note that, for any $i \in I$, the function $y_iy_{i'}$ is invariant under the involution $\sigma$ on  $\F_{}(\Gamma)$. Therefore, at any point of $ \F^{ms}(\Gamma)$ we have
\begin{equation*}
\{y_i^2, y_j^2 \} = \{y_iy_{i'}, y_j y_{j'} \}'= (q_{ij} + q_{i'j'} + q_{i'j} + q_{ij'})y_iy_{i'} y_j y_{j'} = 2(q_{ij} + q_{i'j})y_i^2y_j^2,
\end{equation*}
where we used that $q_{ij} = q_{i'j'}$, $q_{ij'} = q_{i'j}$ due to symmetry and that $y_{i'} = y_i, y_{j'}= y_j$ on $ \F^{ms}(\Gamma)$. Furthermore, if $i \in I_0$ or $j \in I_0$, then $q_{i'j} = q_{ij}$, so the above formula simplifies to $
\{y_i^2, y_j^2 \} = 4q_{ij} y_i^2y_j^2
$
and hence $
\{y_i, y_j \} = q_{ij} y_iy_j.
$
As for the case $i,j \in I_+$, in that situation the faces $i$ and $j'$ cannot be adjacent (since $i$th face is above the midline, and face $j'$ is below), so $q_{ij'} = 0$, and we get 
$
\{ y_i, y_j \} = \frac{1}{2} q_{ij} y_iy_j.
$ 
\end{proof}
\begin{example}
Consider the graph shown in Figure \ref{fig:dual}. Its (move-)symmetric weightings are parametrized by $y_1, y_2, y_3,y_4$, subject to the relation $y_1^2y_2y_3y_4 = 1$. Let us, for example, find the Poisson bracket of $y_1$ and $y_3$. As above, let $\{\,,\}$ be the bracket on move-symmetric weightings and   $\{\,,\}'$ be the bracket on all weightings. The former is induced by the latter as the bracket on the fixed locus of Poisson involution $y_1 \leftrightarrow y_5$. The function $y_1y_5$ is invariant under this involution and, when restricted to move-symmetric weightings, evaluates to $y_1^2$. The function $y_3$ is also invariant. So, we can compute the bracket $\{y_1^2, y_3\} $ by first finding $ \{y_1y_5,y_3\}'$ and then restricting to move-symmetric weightings. One has
$
 \{y_1y_5,y_3\}' = -2y_1y_3y_5
$
which, on move-symmetric weightings, gives $-2y_1^2y_3$. So,
$
\{y_1,y_3\} = -y_1y_3.
$
\end{example}

\paragraph{The orthogonal case.} Now assume that $n$ (the number of sources/sinks of the graph $\Gamma$) is odd, so that the associated boundary measurement matrix is orthogonal. Then $\Gamma$ necessarily contains a chain of squares on the midline as in Figure \ref{fig:ave}.
  \begin{figure}[t]
 \centering
\begin{tikzpicture}[scale = 0.7]
\node () at (0,0){
\begin{tikzpicture}[scale = 0.65]
\node [draw,circle,color=black, fill=black,inner sep=0pt,minimum size=3pt, opacity = 0.3] (A) at (0.3,0) {};
\node [draw,circle,color=black, fill=white,inner sep=0pt,minimum size=3pt, opacity = 0.3] (B) at (1,0.7) {};
\node [draw,circle,color=black, fill=black,inner sep=0pt,minimum size=3pt, opacity = 0.3] (C) at (1.7,0) {};
\node [draw,circle,color=black, fill=white,inner sep=0pt,minimum size=3pt, opacity = 0.3] (D) at (1,-0.7) {};
\node [draw,circle,color=black, fill=white,inner sep=0pt,minimum size=3pt, opacity = 0.3] (E) at (1,2) {};
\node [draw,circle,color=black, fill=black,inner sep=0pt,minimum size=3pt, opacity = 0.3] (F) at (1,-2) {};
 \node [draw,circle,color=black, fill=black,inner sep=0pt,minimum size=3pt, label = { [label distance=-1]  above: \footnotesize$ 1$}]  (1) at (1,0){};
  \node [draw,circle,color=black, fill=black,inner sep=0pt,minimum size=3pt, label = { [label distance=-1]  above: \footnotesize$ 1$}]  (1') at (4,0){};
 \node [draw,circle,color=black, fill=black,inner sep=0pt,minimum size=3pt, label = {[label distance=-2]  left: \footnotesize$ y_2$}]  (Y2) at (-0.5,1.25){};
  \node [draw,circle,color=black, fill=black,inner sep=0pt,minimum size=3pt, label = {[label distance=-2]  left: \footnotesize$ 2y_2$}]  (2Y2) at (-0.5,-1.25){};
   \node [draw,circle,color=black, fill=black,inner sep=0pt,minimum size=3pt, label = {[label distance=-5]  120: \footnotesize$ y_3$}]  (Y3) at (2.5,1.25){};
    \node [draw,circle,color=black, fill=black,inner sep=0pt,minimum size=3pt, label = {[label distance=-5]  -60: \footnotesize$y_3$}]  (12Y3) at (2.5,-1.25){};
       \node [draw,circle,color=black, fill=black,inner sep=0pt,minimum size=3pt, label = {[label distance=-2]  right: \footnotesize$ 2y_4$}]  (Y4) at (5.5,1.25){};
    \node [draw,circle,color=black, fill=black,inner sep=0pt,minimum size=3pt, label = {[label distance=-2]  right: \footnotesize$y_4$}]  (12Y4) at (5.5,-1.25){};
     \node [draw,circle,color=black, fill=black,inner sep=0pt,minimum size=3pt, label = {[label distance=-2]  above: \footnotesize$ y_1$}]  (Y1) at (2.5,2.75){};
        \node [draw,circle,color=black, fill=black,inner sep=0pt,minimum size=3pt, label = {[label distance=-2]  below: \footnotesize$ y_1$}]  (Y1') at (2.5,-2.75){};
        \draw [dotted, ->-, thick] (Y1) to [bend right]  (Y2);
                \draw [->, thick] (Y3) -- (Y1);
                     \draw [dotted, ->, thick] (2Y2)  to [bend left = 15] (Y2);
                             \draw [dotted, ->-,thick ] (Y1') to [bend left]  (2Y2);
                                   \draw [->, thick] (12Y3) -- (Y1');
                                           \draw [dotted, ->-, thick] (Y1) to [bend left]  (Y4);
                     \draw [dotted, ->,thick ] (Y4)  to [bend left = 15]  (12Y4);
                             \draw [dotted, ->-, thick] (Y1') to [bend right]  (12Y4);
                                                     \draw [ ->-,thick ] (1) -- (Y3);
                                                        \draw [ ->-,thick ] (1) -- (2Y2);
                                                           \draw [ ->-, thick] (Y2) -- (1);
                                                                          \draw [ ->-, thick] (12Y3) -- (1);
                                                                                                                               \draw [ ->-,thick ] (1') -- (Y4);
                                                        \draw [ ->-, thick] (1') -- (12Y3);
                                                           \draw [ ->-, thick] (Y3) -- (1');
                                                                          \draw [ ->-, thick] (12Y4) -- (1');
                                                                                  \draw [->, thick] (Y4)-- (Y3);
                                                                                             \draw [->, thick] (2Y2) --  (12Y3);
\draw [opacity = 0.3] (A) -- (B) -- (C) -- (D) -- (A);
\draw  [opacity = 0.3](A) -- +(-1.8,0);
\draw  [opacity = 0.3](B) -- (E);
\draw  [opacity = 0.3](E) -- +(2.5,0);
\draw  [opacity = 0.3](E) -- +(-2.5,0);
\draw [opacity = 0.3] (C) -- +(1,0);
\draw  [opacity = 0.3](D) -- (F);
\draw [opacity = 0.3] (F) -- +(2.5,0);
\draw  [opacity = 0.3](F) -- +(-2.5,0);
\def\n {-2}
\node [draw,circle,color=black, fill=black,inner sep=0pt,minimum size=3pt, opacity = 0.3] (A) at (5.3+\n,0) {};
\node [draw,circle,color=black, fill=white,inner sep=0pt,minimum size=3pt, opacity = 0.3] (B) at (6 + \n,0.7) {};
\node [draw,circle,color=black, fill=black,inner sep=0pt,minimum size=3pt, opacity = 0.3] (C) at (6.7 + \n,0) {};
\node [draw,circle,color=black, fill=white,inner sep=0pt,minimum size=3pt, opacity = 0.3] (D) at (6 +\n,-0.7) {};
\node  [draw,circle,color=black, fill=black,inner sep=0pt,minimum size=3pt, opacity = 0.3] (E) at (6+\n,2) {};
\node  [draw,circle,color=black, fill=white,inner sep=0pt,minimum size=3pt, opacity = 0.3] (F) at (6+\n,-2) {};
\draw [opacity = 0.3] (A) -- (B) -- (C) -- (D) -- (A);
\draw  [opacity = 0.3](A) -- +(-1,0);
\draw  [opacity = 0.3](B) -- (E);
\draw  [opacity = 0.3](E) -- +(2.5,0);
\draw  [opacity = 0.3](E) -- +(-2.5,0);
\draw [opacity = 0.3] (C) -- +(1.8,0);
\draw  [opacity = 0.3](D) -- (F);
\draw [opacity = 0.3] (F) -- +(2.5,0);
\draw  [opacity = 0.3](F) -- +(-2.5,0);
\draw [dashed] (-1.5,-3.75) -- (-1.5,3.75) -- (6.5,3.75) -- (6.5,-3.75) -- cycle;
\end{tikzpicture}};

\node  () at (6,0) {$\longrightarrow$};
\node () at (12,0) {
\begin{tikzpicture}[scale = 0.65]
 \node [draw,circle,color=black, fill=black,inner sep=0pt,minimum size=3pt, label = {[label distance=-2]  left: \footnotesize$\sqrt{2} y_2$}]  (Y2) at (-0.5,1.25){};
   \node [draw,circle,color=black, fill=black,inner sep=0pt,minimum size=3pt, label = {[label distance=-2]  below: \footnotesize$ y_3$}]  (Y3) at (2.5,1.25){};
       \node [draw,circle,color=black, fill=black,inner sep=0pt,minimum size=3pt, label = {[label distance=-2]  right: \footnotesize$ \sqrt{2} y_4$}]  (Y4) at (5.5,1.25){};
     \node [draw,circle,color=black, fill=black,inner sep=0pt,minimum size=3pt, label = {[label distance=-2]  above: \footnotesize$ y_1$}]  (Y1) at (2.5,2.75){};
        \draw [dotted, ->-, thick] (Y1) to [bend right]  (Y2);
                \draw [->, thick] (Y3) -- (Y1);
                                           \draw [dotted, ->-, thick] (Y1) to [bend left]  (Y4);
                                                                                  \draw [->, dotted,thick ] (Y4) -- (Y3);
                                                                                             \draw [->, dotted, thick] (Y2) --   (Y3);
\end{tikzpicture}
};

\end{tikzpicture}
\caption{ The quiver representing the Poisson bracket on move-symmetric weightings (multiplied by a factor of $2$). The  parameters $y_1, \sqrt{2}y_2, y_3, \sqrt{2}y_4$ are subject to the relation $y_1(\sqrt{2}y_2)y_3(\sqrt{2}y_4) = \pm 1$.}\label{fig:ave}
    
\end{figure}
 The number of squares may be arbitrary, and the colors of vertices of each square might be opposite to those shown in the figure. 
 
 We keep the same labeling of faces as in the symplectic case. The move symmetry condition implies $y_i = 1$ for any $i \in I_0$, while for any $i \in I_+$ we have $y_{i'} = cy_i$ where  $c_i \in \R_+$ depends only on the combinatorics of the graph (as before, faces $i'$ and $i$ are symmetric with respect to the midline). Consider symmetrized coordinates  $\sqrt{y_iy_{i'}}$, which, by definition, are understood as $\sqrt{c_i}y_i$. Then the coordinate ring $\C[ \F^{ms}(\Gamma)] $ of  $ \F^{ms}(\Gamma)$  is generated by
$(\sqrt{y_iy_{i'}})^{\pm 1}$, $i \in I_+$, subject to the relation  $\prod\nolimits_{i \in I_+}\sqrt{y_iy_{i'}}  = \pm 1$.
To describe a Poisson structure on  $ \F^{ms}(\Gamma)$, it suffices to find pairwise Poisson brackets of  functions $\sqrt{y_iy_{i'}}$. 

\begin{proposition}\label{prop:ipbms}
Up to a scalar factor,  the number of arrows from vertex $i$ to vertex $j$ in the quiver defining the bracket on $ \F^{ms}(\Gamma)$ constructed in Section \ref{sec3} is the average between the corresponding numbers in the upper and lower parts of $\Q$, see Figure~\ref{fig:ave}. Specifically, the bracket on $ \F^{ms}(\Gamma)$ is given by
\begin{gather*}
\{\sqrt{y_iy_{i'}}, \sqrt{y_jy_{j'}}\} = \frac{1}{4} (q_{ij} + q_{i'j'})\sqrt{y_iy_{i'}}\sqrt{y_jy_{j'}}.
\end{gather*}
\end{proposition}
\begin{remark}
Another possible choice of coordinates on $ \F^{ms}(\Gamma)$ is given by $y_i$, $i \in I_+$. Those coordinates are related to the above ones by the rule $\sqrt{y_iy_{i'}} = \sqrt{c_i}y_i$, so the Poisson brackets of $y_i$'s have the same form $\{y_i, y_j\} =  \frac{1}{4} (q_{ij} + q_{i'j'})y_iy_j$. Our motivation for considering the symmetrized coordinates $\sqrt{y_iy_{i'}}$ is that they turn out to be cluster coordinates, see Section \ref{sec:dbc} and, in particular, Example \ref{fgex}.
\end{remark}

\begin{proof}[\bf{Proof of Proposition \ref{prop:ipbms}}]
In this case, the involution $\sigma$ is a composition of square moves and reflection, performed in arbitrary order.  The quantity $y_iy_{i'}$, where $i \in I_+$, is invariant under reflection, so the effect of $\sigma$ on that function boils down to square moves. Using the square move formulas, along with the relation $q_{i'k} = -q_{ik}$ which, due to move-symmetry, holds for any $i \in I_+$ and $k \in I_0$, we get that
$\sigma^* (y_iy_{i'}) = \eta_i y_iy_{i'}$ where $\eta_i := \prod_{k \in I_0} y_k^{|q_{ik}|}.$ Furthermore, the action of $\sigma$ on $\eta_i$ is given by $\sigma^* \eta_i = \eta_i^{-1}$. Therefore, the function $\eta_i(y_iy_{i'})^2$ is invariant under $\sigma$. Also using that $\eta_i =  1$ on $ \F^{ms}(\Gamma)$ for all $i \in I_0$, at any point of $ \F^{ms}(\Gamma)$ we get
$$
\{(y_iy_{i'})^2, (y_jy_{j'})^2\} = \{\eta_i(y_iy_{i'})^2, \eta_j(y_jy_{j'})^2\}'
$$
for all $ i,j \in I_+$. Further, once again using that $q_{ik} = -q_{i'k}$ for any $i \in I_+$, $k \in I_0$, we get that $\{y_iy_{i'}, \eta_j\} = 0$ for all $i,j \in I_+$. Also, since all vertices in $I_0$ are disjoint, we have $\{\eta_i, \eta_j\} = 0$. So, 
$$
 \{(y_iy_{i'})^2, (y_jy_{j'})^2\} =  \{(y_iy_{i'})^2, (y_jy_{j'})^2\}' = 4(q_{ij} + q_{i'j'} + q_{i'j} + q_{ij'})(y_iy_{i'} )^2(y_j y_{j'})^2. 
$$
Finally, observe that the faces $i$ and $j'$ cannot be adjacent (unless $i = j$ in which case the proposition is trivial), and so $q_{i'j} = 0$. Analogously, $q_{ij'} = 0$. The result follows. \qedhere
\end{proof}

  \begin{figure}[t]
 \centering
\begin{tikzpicture}[scale = 0.75]
\node () at (0,0){\begin{tikzpicture}[scale = 0.8]
\node [draw,circle,color=black, fill=black,inner sep=0pt,minimum size=3pt, opacity = 1] (A) at (0.3,0) {};
\node [draw,circle,color=black, fill=white,inner sep=0pt,minimum size=3pt, opacity = 1] (B) at (1,0.7) {};
\node [draw,circle,color=black, fill=black,inner sep=0pt,minimum size=3pt, opacity = 1] (C) at (1.7,0) {};
\node [draw,circle,color=black, fill=white,inner sep=0pt,minimum size=3pt, opacity = 1] (D) at (1,-0.7) {};
\node [draw,circle,color=black, fill=white,inner sep=0pt,minimum size=3pt, opacity = 1] (E) at (1,2) {};
\node [draw,circle,color=black, fill=black,inner sep=0pt,minimum size=3pt, opacity = 1] (F) at (1,-2) {};
\node () at (1,0) {\footnotesize$y_4$};
\node () at (0,1) {\footnotesize$y_2$};
\node () at (0,-1) {\footnotesize$y_5$};
\node () at (2,1) {\footnotesize$y_3$};
\node () at (2,-1) {\footnotesize$y_6$};
\node () at (1,2.5) {\footnotesize$y_1$};
\node () at (1,-2.5) {\footnotesize$y_7$};
\draw [] (A) -- (B) -- (C) -- (D) -- (A);
\draw (A) -- +(-1.3,0);
\draw (B) -- (E);
\draw (E) -- +(2,0);
\draw (E) -- +(-2,0);
\draw (C) -- +(1.3,0);
\draw (D) -- (F);
\draw (F) -- +(2,0);
\draw (F) -- +(-2,0);
\draw [dashed] (-1,-3) -- (-1,3) -- (3,3) -- (3,-3) -- cycle;
\end{tikzpicture}};
\node () at (4,0) {$\xrightarrow{\mbox{square move}}$};

\node () at (8,0){\begin{tikzpicture}[scale = 0.8]

\draw [dashed] (-1,-3) -- (-1,3) -- (3,3) -- (3,-3) -- cycle;
\node [draw,circle,color=black, fill=white,inner sep=0pt,minimum size=3pt, opacity = 1] (A) at (0.3,0) {};
\node [draw,circle,color=black, fill=black,inner sep=0pt,minimum size=3pt, opacity = 1] (B) at (1,0.7) {};
\node [draw,circle,color=black, fill=white,inner sep=0pt,minimum size=3pt, opacity = 1] (C) at (1.7,0) {};
\node [draw,circle,color=black, fill=black,inner sep=0pt,minimum size=3pt, opacity = 1] (D) at (1,-0.7) {};
\node [draw,circle,color=black, fill=white,inner sep=0pt,minimum size=3pt, opacity = 1] (E) at (1,2) {};
\node [draw,circle,color=black, fill=black,inner sep=0pt,minimum size=3pt, opacity = 1] (F) at (1,-2) {};
\node () at (1,0) {\footnotesize$y_4^{-1}$};
\node () at (0,1) {\footnotesize$y_2(1+y_4)$};
\node () at (0,-1) {\footnotesize$\dfrac{y_5}{1+y_4^{-1}}$};
\node () at (2,1) {\footnotesize$\dfrac{y_3}{1+y_4^{-1}}$};
\node () at (2,-1) {\footnotesize$y_6(1+y_4)$};
\node () at (1,2.5) {\footnotesize$y_1$};
\node () at (1,-2.5) {\footnotesize$y_7$};
\draw [] (A) -- (B) -- (C) -- (D) -- (A);
\draw (A) -- +(-1.3,0);
\draw (B) -- (E);
\draw (E) -- +(2,0);
\draw (E) -- +(-2,0);
\draw (C) -- +(1.3,0);
\draw (D) -- (F);
\draw (F) -- +(2,0);
\draw (F) -- +(-2,0);

\end{tikzpicture}};
\node () at (12,0) {$\xrightarrow{\mbox{reflection}}$};

\node () at (16,0){\begin{tikzpicture}[xscale = 0.8, yscale = -0.8]

\draw [dashed] (-1,-3) -- (-1,3) -- (3,3) -- (3,-3) -- cycle;
\node [draw,circle,color=black, fill=white,inner sep=0pt,minimum size=3pt, opacity = 1] (A) at (0.3,0) {};
\node [draw,circle,color=black, fill=black,inner sep=0pt,minimum size=3pt, opacity = 1] (B) at (1,0.7) {};
\node [draw,circle,color=black, fill=white,inner sep=0pt,minimum size=3pt, opacity = 1] (C) at (1.7,0) {};
\node [draw,circle,color=black, fill=black,inner sep=0pt,minimum size=3pt, opacity = 1] (D) at (1,-0.7) {};
\node [draw,circle,color=black, fill=white,inner sep=0pt,minimum size=3pt, opacity = 1] (E) at (1,2) {};
\node [draw,circle,color=black, fill=black,inner sep=0pt,minimum size=3pt, opacity = 1] (F) at (1,-2) {};
\node () at (1,0) {\footnotesize$y_4^{-1}$};
\node () at (0,1) {\footnotesize$y_2(1+y_4)$};
\node () at (0,-1) {\footnotesize$\dfrac{y_5}{1+y_4^{-1}}$};
\node () at (2,1) {\footnotesize$\dfrac{y_3}{1+y_4^{-1}}$};
\node () at (2,-1) {\footnotesize$y_6(1+y_4)$};
\node () at (1,2.5) {\footnotesize$y_1$};
\node () at (1,-2.5) {\footnotesize$y_7$};
\draw [] (A) -- (B) -- (C) -- (D) -- (A);
\draw (A) -- +(-1.3,0);
\draw (B) -- (E);
\draw (E) -- +(2,0);
\draw (E) -- +(-2,0);
\draw (C) -- +(1.3,0);
\draw (D) -- (F);
\draw (F) -- +(2,0);
\draw (F) -- +(-2,0);

\end{tikzpicture}};
\end{tikzpicture}
\caption{ Involution $\sigma$ on the face weight space of a move-symmetric plabic graph. 
}\label{fig:msinv}
    
\end{figure}

\begin{example}
Figure \ref{fig:msinv} shows the action of the involution $\sigma$ on the face weight space of the move-symmetric plabic graph from Figure \ref{fig:ms}. The functions $y_4^{-1}(y_2y_5)^2$,  $y_4^{-1}(y_3y_6)^2$ are invariant under $\sigma$, so the Poisson bracket of their restrictions to move-symmetric weightings is equal to the restriction of their Poisson bracket. This gives 
$
\{(y_2y_5)^2, (y_3y_6)^2\} = 4(y_2y_5)^2(y_3y_6)^2,
$
so
$
\{\sqrt{y_2y_5}, \sqrt{y_3y_6}\} = \frac{1}{4}\sqrt{y_2y_5} \sqrt{y_3y_6}.
$
\end{example}

\medskip

\section{Plabic graphs, double Bruhat cells, and cluster coordinates}\label{sec:dbc}
\paragraph{Double Bruhat cells and double words.} Let $G$ be a complex connected reductive Lie group, $T \subset G$ be a maximal torus, and $B_\pm$ be a pair of opposite Borel subgroups with $B_+ \cap B_- = T$. Let also $W = N_G(T) / T$, where  $N_G(T)$  is the normalizer of $T$ in $G$, be the Weyl group of $G$. For any $u,v \in W$, the corresponding \emph{double Bruhat cell} $G^{u,v}$ is defined as the intersection of opposite Bruhat cells $B_+ u B_+ \cap B_- v B_-$ \cite{fomin1999double}. Here a double coset $BwB$ of an element of $w \in W$ with respect to a Borel subgroup $B \subset G$ is defined as $B\dot wB$, where $\dot w \in N_G(T)$ is any representative of $w$ in $N_G(T)$.

Any complex connected reductive Lie group $G$ is a disjoint union of its double Bruhat cells. Those cells can also be described as \emph{$T$-leaves}, i.e. orbits of symplectic leaves of the standard Poisson structure on $G$ under the (left or right) action of the maximal torus \cite{hoffmann2000factorization, kogan2002symplectic}. 


Let $\{\alpha_1, \dots, \alpha_k\}$ be the simple roots of $G$. 
Then elements of the Weyl group $W$ can be represented by words in the alphabet $[1,k] := \{1, \dots, k\}$. By definition, the Weyl group element corresponding to the word $i_1\dots i_m$ is $s_{i_1} \cdots s_{i_m}$, where $s_i$ is the simple reflection corresponding to the root $\alpha_i$. Accordingly, elements of the direct product $W \times W$ of two copies of $W$ can be represented by words in the alphabet  $ [-k,-1] \cup [1,k]$ where symbols $i \in [-k,-1]$ correspond to generators of the first copy of $W$, and symbols $i \in [1,k]$ to generators of the second copy of $W$. Such words are called \emph{double words in $W$.} 
A double word is \emph{reduced} 
 if it has the shortest length among all words representing the given element of $W \times W$ (equivalently, if it is a shuffle of a reduced word in the first copy of $W$ and a reduced word in the second copy of $W$).

\begin{figure}[p]

$${B_k} \quad \dynkin[ labels = { 1,2,k-1,k}, root radius=.035, arrow shape/.style={-{Bourbaki[length=3pt]}}, scale = 2] B{*2.*2} \qquad \qquad
 C_k \quad \dynkin[labels = {1,2,k-1,k}, root radius=.035,arrow shape/.style={-{Bourbaki[length=3pt]}}, scale = 2] C{*2.*2} $$
%
 \caption{Dynkin diagrams of types $B_k$ and $C_k$. We label  the roots so that the double edge is between the nodes labeled $k-1$ and $k$. }\label{fig:dynkin}
\end{figure}

\begin{figure}[p]
\centering
\begin{tikzpicture}[scale = 1]
\centering
\node () at (0.5,0.25){ \footnotesize $$};
\node () at (0.5,2.25){ \footnotesize $$};
\node () at (0.5,1.25){ \footnotesize $$};

\draw [dashed] (0,3) -- (1,3);
\draw [dashed] (1,0) -- (1,3) node[pos = 1/6, right] {\footnotesize$i$} node[pos = 1/3, right] {\footnotesize$i+1$} node[pos = 5/6, right] {\footnotesize$n+1 - i$} node[pos = 2/3, right] {\footnotesize$n-i$};;; ;
\draw [ dashed] (0,0) -- (1,0) ;
\draw [ dashed]  (0,0) -- (0,3) node[pos = 1/6, left] {\footnotesize$i$} node[pos = 1/3, left] {\footnotesize$i+1$} node[pos = 5/6, left] {\footnotesize$n+1 - i$} node[pos = 2/3, left] {\footnotesize$n-i$};;; ;;
\node [draw,circle,color=black, fill=black,inner sep=0pt,minimum size=3pt, opacity = 1] (A) at (0.5,2.5) {};
\node [draw,circle,color=black, fill=white,inner sep=0pt,minimum size=3pt, opacity = 1] (B) at (0.5,2) {};
\draw (A) -- (B);
\draw (0,2.5) -- (A) -- (1,2.5);
\draw (0,2) -- (B) -- (1,2);
\node [draw,circle,color=black, fill=black,inner sep=0pt,minimum size=3pt, opacity = 1] (A) at (0.5,1) {};
\node [draw,circle,color=black, fill=white,inner sep=0pt,minimum size=3pt, opacity = 1] (B) at (0.5,0.5) {};
\draw (A) -- (B);
\draw (0,1) -- (A) -- (1,1);
\draw (0,0.5) -- (B) -- (1,0.5);
\node (B) at (-1.5,2.8) {$\Gamma_{i}:$};
\end{tikzpicture}
\qquad
\begin{tikzpicture}[scale = 1]
\centering
\node () at (0.5,0.25){ \footnotesize $$};
\node () at (0.5,2.25){ \footnotesize $$};
\node () at (0.5,1.25){ \footnotesize $$};

\draw [dashed] (0,3) -- (1,3);
\draw [dashed] (1,0) -- (1,3) node[pos = 1/6, right] {\footnotesize$i$} node[pos = 1/3, right] {\footnotesize$i+1$} node[pos = 5/6, right] {\footnotesize$n+1 - i$} node[pos = 2/3, right] {\footnotesize$n-i$};;; ;
\draw [ dashed] (0,0) -- (1,0) ;
\draw [ dashed]  (0,0) -- (0,3) node[pos = 1/6, left] {\footnotesize$i$} node[pos = 1/3, left] {\footnotesize$i+1$} node[pos = 5/6, left] {\footnotesize$n+1 - i$} node[pos = 2/3, left] {\footnotesize$n-i$};;; ;;
\node [draw,circle,color=black, fill=white,inner sep=0pt,minimum size=3pt, opacity = 1] (A) at (0.5,2.5) {};
\node [draw,circle,color=black, fill=black,inner sep=0pt,minimum size=3pt, opacity = 1] (B) at (0.5,2) {};
\draw (A) -- (B);
\draw (0,2.5) -- (A) -- (1,2.5);
\draw (0,2) -- (B) -- (1,2);
\node [draw,circle,color=black, fill=white,inner sep=0pt,minimum size=3pt, opacity = 1] (A) at (0.5,1) {};
\node [draw,circle,color=black, fill=black,inner sep=0pt,minimum size=3pt, opacity = 1] (B) at (0.5,0.5) {};
\draw (A) -- (B);
\draw (0,1) -- (A) -- (1,1);
\draw (0,0.5) -- (B) -- (1,0.5);
\node (B) at (-1.5,2.8) {$\Gamma_{-i}:$};
\end{tikzpicture}
\caption{Plabic graphs associated to the root $ \alpha_i$ for $i \in [1,k-1]$ in types $B_k$ and $C_k$. Here $n$ is the dimension of the defining representation, i.e., $n = 2k$ in type $C_k$ and $n = 2k+1$ in type $B_k$. }\label{fig:graphroots}
\end{figure}

\begin{figure}[p]
\centering
\begin{tikzpicture}[scale = 1]
\centering
\node () at (0.5,0.25){ \footnotesize $$};
\node () at (0.5,2.25){ \footnotesize $$};
\node () at (0.5,1.25){ \footnotesize $$};

\draw [dashed] (0,3) -- (1,3);
\draw [dashed] (1,0) -- (1,3) node[pos = 3/4, right] {\footnotesize$k+2$} node[pos = 1/4, right] {\footnotesize$k$} node[pos = 1/2, right] {\footnotesize$k+1$}; 
\draw [ dashed] (0,0) -- (1,0) ;
\draw [ dashed]  (0,0) -- (0,3) node[pos = 3/4, left] {\footnotesize$k+2$} node[pos = 1/4, left] {\footnotesize$k$} node[pos = 1/2, left] {\footnotesize$k+1$}; 
\node [draw,circle,color=black, fill=white,inner sep=0pt,minimum size=3pt, opacity = 1] (A) at (0.25,1.5) {};
\node [draw,circle,color=black, fill=black,inner sep=0pt,minimum size=3pt, opacity = 1] (B) at (0.5,1.75) {};
\node [draw,circle,color=black, fill=white,inner sep=0pt,minimum size=3pt, opacity = 1] (C) at (0.75,1.5) {};
\node [draw,circle,color=black, fill=black,inner sep=0pt,minimum size=3pt, opacity = 1] (D) at (0.5,1.25) {};
\node [draw,circle,color=black, fill=black,inner sep=0pt,minimum size=3pt, opacity = 1] (E) at (0.5,2.25) {};
\node [draw,circle,color=black, fill=white,inner sep=0pt,minimum size=3pt, opacity = 1] (F) at (0.5,0.75) {};
\draw [] (A) -- (B) -- (C) -- (D) -- (A);
\draw (A) -- +(-0.25,0);
\draw (B) -- (E);
\draw (E) -- +(0.5,0);
\draw (E) -- +(-0.5,0);
\draw (C) -- +(0.25,0);
\draw (D) -- (F);
\draw (F) -- +(0.5,0);
\draw (F) -- +(-0.5,0);

\node (B) at (-1.5,2.8) {$\Gamma_{k}:$};
\end{tikzpicture}
\qquad
\begin{tikzpicture}[scale = 1]
\centering
\node () at (0.5,0.25){ \footnotesize $$};
\node () at (0.5,2.25){ \footnotesize $$};
\node () at (0.5,1.25){ \footnotesize $$};

\draw [dashed] (0,3) -- (1,3);
\draw [dashed] (1,0) -- (1,3) node[pos = 3/4, right] {\footnotesize$k+2$} node[pos = 1/4, right] {\footnotesize$k$} node[pos = 1/2, right] {\footnotesize$k+1$}; 
\draw [ dashed] (0,0) -- (1,0) ;
\draw [ dashed]  (0,0) -- (0,3) node[pos = 3/4, left] {\footnotesize$k+2$} node[pos = 1/4, left] {\footnotesize$k$} node[pos = 1/2, left] {\footnotesize$k+1$}; 
\node [draw,circle,color=black, fill=white,inner sep=0pt,minimum size=3pt, opacity = 1] (A) at (0.25,1.5) {};
\node [draw,circle,color=black, fill=black,inner sep=0pt,minimum size=3pt, opacity = 1] (B) at (0.5,1.75) {};
\node [draw,circle,color=black, fill=white,inner sep=0pt,minimum size=3pt, opacity = 1] (C) at (0.75,1.5) {};
\node [draw,circle,color=black, fill=black,inner sep=0pt,minimum size=3pt, opacity = 1] (D) at (0.5,1.25) {};
\node [draw,circle,color=black, fill=white,inner sep=0pt,minimum size=3pt, opacity = 1] (E) at (0.5,2.25) {};
\node [draw,circle,color=black, fill=black,inner sep=0pt,minimum size=3pt, opacity = 1] (F) at (0.5,0.75) {};
\draw [] (A) -- (B) -- (C) -- (D) -- (A);
\draw (A) -- +(-0.25,0);
\draw (B) -- (E);
\draw (E) -- +(0.5,0);
\draw (E) -- +(-0.5,0);
\draw (C) -- +(0.25,0);
\draw (D) -- (F);
\draw (F) -- +(0.5,0);
\draw (F) -- +(-0.5,0);

\node (B) at (-1.5,2.8) {$\Gamma_{-k}:$};
\end{tikzpicture}
\caption{Plabic graphs associated to the short root $ \alpha_k$ in type $B_k$.}\label{fig:graphrootsB}
\end{figure}

\begin{figure}[p]
\centering
\begin{tikzpicture}[scale = 1]
\centering
\node () at (0.5,0.25){ \footnotesize $$};
\node () at (0.5,2.25){ \footnotesize $$};
\node () at (0.5,1.25){ \footnotesize $$};

\draw [dashed] (0,2) -- (1,2);
\draw [dashed] (1,0) -- (1,2) node[pos = 1/3, right] {\footnotesize$k$} node[pos = 2/3, right] {\footnotesize$k+1$}; 
\draw [ dashed] (0,0) -- (1,0) ;
\draw [ dashed]  (0,0) -- (0,2) node[pos = 1/3, left] {\footnotesize$k$} node[pos = 2/3, left] {\footnotesize$k+1$}; 
\node [draw,circle,color=black, fill=black,inner sep=0pt,minimum size=3pt, opacity = 1] (E) at (0.5,4/3) {};
\node [draw,circle,color=black, fill=white,inner sep=0pt,minimum size=3pt, opacity = 1] (F) at (0.5,2/3) {};
\draw (E) -- +(0.5,0);
\draw (E) -- +(-0.5,0);
\draw (E) -- (F);
\draw (F) -- +(0.5,0);
\draw (F) -- +(-0.5,0);

\node (B) at (-1.5,1.8) {$\Gamma_{k}:$};
\end{tikzpicture}
\qquad
\begin{tikzpicture}[scale = 1]
\centering
\node () at (0.5,0.25){ \footnotesize $$};
\node () at (0.5,2.25){ \footnotesize $$};
\node () at (0.5,1.25){ \footnotesize $$};

\draw [dashed] (0,2) -- (1,2);
\draw [dashed] (1,0) -- (1,2) node[pos = 1/3, right] {\footnotesize$k$} node[pos = 2/3, right] {\footnotesize$k+1$}; 
\draw [ dashed] (0,0) -- (1,0) ;
\draw [ dashed]  (0,0) -- (0,2) node[pos = 1/3, left] {\footnotesize$k$} node[pos = 2/3, left] {\footnotesize$k+1$}; 
\node [draw,circle,color=black, fill=white,inner sep=0pt,minimum size=3pt, opacity = 1] (E) at (0.5,4/3) {};
\node [draw,circle,color=black, fill=black,inner sep=0pt,minimum size=3pt, opacity = 1] (F) at (0.5,2/3) {};
\draw (E) -- +(0.5,0);
\draw (E) -- +(-0.5,0);
\draw (E) -- (F);
\draw (F) -- +(0.5,0);
\draw (F) -- +(-0.5,0);

\node (B) at (-1.5,1.8) {$\Gamma_{-k}:$};
\end{tikzpicture}
\caption{Plabic graphs associated to the long root $ \alpha_k$ in type $C_k$.}\label{fig:graphrootsC}
\end{figure}

 \paragraph{Plabic graph parametrization of double Bruhat cells.} Now, let $G = \Sp_{2k}$ or $O_{2k+1}$ (note that although the latter group is disconnected, it still decomposes into double Bruhat cells, assuming we define the maximal torus and Borel subgroups as preimages of the corresponding subgroups in $\SO_{2k+1}$). In this section, we will give a parametrization of double Bruhat cells in $G$ by means of move-symmetric plabic graphs.

Label  the roots of $G$ so that the double edge in the corresponding type $B_k/C_k$ Dynkin diagram is between the nodes $k-1$ and $k$, see Figure \ref{fig:dynkin}. 
To every root $\alpha_i$, $i \in [1,k]$, we associate two move-symmetric plabic graph denoted by $\Gamma_{\pm i}$, so that there is a graph $\Gamma_j$ for any $j \in [1,k]  \cup [-k,-1]$. 
The graphs $\Gamma_{\pm i}$ associated to the roots $ \alpha_i$, $i < k$, are shown in Figure~\ref{fig:graphroots}. The graphs associated to the short root $\alpha_k$ in type $B_k$ are shown in Figure \ref{fig:graphrootsB} (the colors of vertices of the square face in the middle can be switched to the opposite ones; this does not affect the result below). Finally, the graphs associated to the long root $\alpha_k$ in type $C_k$ are shown in Figure \ref{fig:graphrootsC}. In all figures, all sources which are not explicitly shown are directly connected to the corresponding sinks.

Fix $(u,v) \in W \times W$ and some double reduced word $\word$  representing $(u,v)$. Replace every letter $i \in [1,k]  \cup [-k,-1]$ in that word with the associated graph  $\Gamma_{i}$ and take concatenation of those graphs in the order prescribed by the word $\word$. Denote the so-obtained move-symmetric plabic graph by $\Gamma_{\word}$. For the empty word $\word = \emptyset$, the corresponding graph $\Gamma_\emptyset$ is, by definition, the plabic graph with sources directly connected to sinks.

\begin{proposition}\label{dbc}
Let $G = G(\Omega_n)$ be the group $O_{2k+1}$ or $\Sp_{2k}$, and let $W$ be the corresponding Weyl group. 
Then, for any $(u,v) \in W \times W$ and any double reduced word $\word$  representing $(u,v)$, 
the boundary measurement mapping maps $ \F^{ms}(\Gamma_{\word})$ birationally onto $G^{u,v}$. 
\end{proposition}
\begin{proof}
This is a plabic graph counterpart of parametrization of $G^{u,v}$ by means of factorization coordinates. Recall that the latter is constructed as follows. Let $T \subset G$ be a maximal torus, $E_i, F_i$ be the corresponding Chevalley generators of the Lie algebra of $G$,  and let \begin{equation}\label{eq:x}X_i(t) := \exp(t E_i), \quad X_{-i}(t) := \exp(t F_i).\end{equation} 
Suppose $\word = i_1\dots i_{m}$ is a double reduced word representing  $(u,v) \in W\times W$. Define a map 
\begin{equation}\label{eq:phimap}
\phi_\word \colon T \times (\C^*)^{m} \to G, \quad \phi_\word(H, t_1, \dots, t_{m}) := H X_{i_1}(t_1) \cdots X_{i_{m}}( t_{m}).
\end{equation}
This map is birational, and the corresponding coordinates on (an open dense subset of) $G^{u,v}$ are known as \emph{factorization parameters} \cite{fomin1999double}.

Now return to the setting of plabic graphs. Consider the $C_k$ case. In our realization of $\Sp_{2k}$, Chevalley generators are
\begin{align}\label{cb1}
\begin{aligned}
E_i &= \begin{cases}
E_{i, i+1} + E_{2k-i, 2k-i + 1} \quad \mbox{for } i < k,\\
E_{k, k+1} \quad  \mbox{for } i = k,
\end{cases}\\
&F_i = E_i^t \quad \mbox{for all } i,
\end{aligned}
\end{align}
while the maximal torus is given by diagonal matrices with entries $x_1, \dots, x_k, x_k^{-1}, \dots, x_1^{-1}$, where $x_i \in \C^*$. For every $i \in [1,k]  \cup [-k,-1]$, define a mapping $\psi_i \colon \C^* \to \F^{ms}(\Gamma_i) $ as follows. 
The graph $\Gamma_i$ has a unique perfect orientation, turning it into a perfect network. Attach weight $t$ to vertical edges of that network (there are two such edges for $i \neq k$, and one such edge for $i = k$), and weight $1$ to the remaining edges. The corresponding face weight collection $\mathcal Y \in  \F^{ms}(\Gamma_i) $ is, by definition, $\psi_i(t)$. 
Further, define a map $\psi_0 \colon T \to \F^{ms}(\Gamma_\emptyset) $ by assigning to $H \in T$ the face weights of a network whose sources are directly connected to sinks and whose edge weights are given by diagonal entries of $H$, going from bottom to top. Note that $\mes( \psi_0(H)) = H$ and $\mes( \psi_i(t)) = X_i(t)$.
\begin{example}
\begin{figure}[t]
\centering
\begin{tikzpicture}[scale = 1.2]
\centering
\node () at (0.5,0.25){ \footnotesize $$};
\node () at (0.5,2.25){ \footnotesize $$};
\node () at (0.5,1.25){ \footnotesize $$};

\draw [dashed] (0,2) -- (1,2);
\draw [dashed] (1,0) -- (1,2) node[pos = 1/3, right] {\footnotesize$k$} node[pos = 2/3, right] {\footnotesize$k+1$}; 
\draw [ dashed] (0,0) -- (1,0) ;
\draw [ dashed]  (0,0) -- (0,2) node[pos = 1/3, left] {\footnotesize$k$} node[pos = 2/3, left] {\footnotesize$k+1$}; 
\node [draw,circle,color=black, fill=black,inner sep=0pt,minimum size=3pt, opacity = 1] (E) at (0.5,4/3) {};
\node [draw,circle,color=black, fill=white,inner sep=0pt,minimum size=3pt, opacity = 1] (F) at (0.5,2/3) {};
\draw [->] (E) -- +(0.5,0);
\draw [->] (0,4/3) -- (E);
\draw [->] (F) -- (E) node[pos = 1/2, right] {\footnotesize $t$};
\draw [->] (F) -- +(0.5,0);
\draw [->] (0,2/3) -- (F);

\end{tikzpicture}
\caption{Network $\psi_k(t)$.  All sources which are not explicitly shown are directly connected to the corresponding sinks. All omitted edge weights are equal to $1$. }\label{fig:psikex}
\end{figure}
For instance, let us show that
$\mes( \psi_k(t)) = X_k(t)$. 
By construction, the network $\psi_k(t)$ is a perfectly oriented version of the graph $\Gamma_k$ with weight $t$ on the vertical edge and weight $1$ on the horizontal ones, see Figure \ref{fig:psikex}. This network has a path of weight $1$ from each source to the corresponding sink, a path of weight $t$ from source $k$ to sink $k+1$, and no other paths. Denoting the identity matrix by $\Id$, we find that the boundary measurement matrix $\mes( \psi_k(t)) $  is
$
\Id + tE_{k,k+1} = \exp(tE_{k,k+1}) = X_k(t),
$
q.e.d.
\end{example}

Back to the proof, consider the map $$\psi_\word \colon T \times (\C^*)^{m} \to  \F^{ms}(\Gamma_\word), \quad \psi_\word(H, t_1, \dots, t_{m}) := \psi_0(H) \psi_{i_1}(t_1) \cdots \psi_{i_{m}}( t_{m}),
$$
where multiplication of weightings in the right-hand side is understood as concatenation of the corresponding weighted graphs. Since the boundary measurement map is a homomorphism of monoids, we have
\begin{gather*}
\mes(\psi_\word(H, t_1, \dots, t_{m}) ) = \mes( \psi_0(H) \psi_{i_1}(t_1) \cdots \psi_{i_{m}}( t_{m})) \\ = \mes( \psi_0(H))\mes( \psi_{i_1}(t_1)) \cdots \mes(\psi_{i_{m}}( t_{m}))   =  H X_{i_1}(t_1) \cdots X_{i_{m}}( t_{m}) = \phi_\word(H, t_1, \dots, t_{m}),
\end{gather*}
i.e., 
$
\phi_\word = \mes \circ \psi_\word.
$


Now, let us show that $\dim \F^{ms}(\Gamma_\word) = m+k$.
To that end, note that the right boundary of each face of $\Gamma_\word$ is either a vertical edge or a piece of the right boundary of the enclosing rectangle. At the same time, each of the $m$ graphs $\Gamma_i$ used to create the graph $\Gamma_\word$ has exactly one vertical edge on or above the midline. So, the total number of such vertical edges in  $\Gamma_\word$ is $m$. In addition to that, the right boundary of the rectangle enclosing $\Gamma_\word$ has $k+1$ pieces on or above the midline. So, $\Gamma_\word$  has $m+k+1$ faces  on or above the midline, and hence $\dim \F^{ms}(\Gamma_\word) = m+k$.

As a result, $\psi_\word$ is a monomial map between tori of the same dimension. Also, since the map $\phi_\word = \mes \circ \psi_\word$ is generically injective, the same must be true for $\psi_\word$. But a generically injective monomial map between tori of the same dimension is birational. So, the boundary measurement map can be expressed as $\mes = \phi_\word \circ \psi_\word^{-1}$ and hence maps $ \F^{ms}(\Gamma_{\word})$ birationally onto $G^{u,v}$.

The argument in the $B_k$ case is similar, with the following modifications. The Chevalley generators are
\begin{align}\label{cb2}
\begin{aligned}
E_i &= \begin{cases}
E_{i, i+1} + E_{2k-i + 1, 2k-i + 2} \quad \mbox{for } i < k,\\
\sqrt{2} (E_{k, k+1} + E_{k+1, k+2}) \quad  \mbox{for } i = k,
\end{cases}\\
&F_i = E_i^t \quad \mbox{for all } i,
\end{aligned}
\end{align}
while the Cartan subgroup $T$ consists of diagonal matrices with entries $x_1, \dots, x_k, \pm 1, x_k^{-1}, \dots, x_1^{-1}$. 
\begin{figure}
\centering
\begin{tikzpicture}[scale = 1.25]
\centering
\node () at (0.5,0.25){ \footnotesize $$};
\node () at (0.5,2.25){ \footnotesize $$};
\node () at (0.5,1.25){ \footnotesize $$};

\draw [dashed] (-0.25,3) -- (1.25,3);
\draw [dashed] (1.25,0) -- (1.25,3) node[pos = 5/6, right] {\footnotesize$k+2$} node[pos = 1/6, right] {\footnotesize$k$} node[pos = 1/2, right] {\footnotesize$k+1$}; 
\draw [ dashed] (-0.25,0) -- (1.25,0) ;
\draw [ dashed]  (-0.25,0) -- (-0.25,3)  node[pos = 5/6, left] {\footnotesize$k+2$} node[pos = 1/6, left] {\footnotesize$k$} node[pos = 1/2, left] {\footnotesize$k+1$} ; 
\node [draw,circle,color=black, fill=white,inner sep=0pt,minimum size=3pt, opacity = 1] (A) at (0.15,1.5) {};
\node [draw,circle,color=black, fill=black,inner sep=0pt,minimum size=3pt, opacity = 1] (B) at (0.5,1.85) {};
\node [draw,circle,color=black, fill=white,inner sep=0pt,minimum size=3pt, opacity = 1] (C) at (0.85,1.5) {};
\node [draw,circle,color=black, fill=black,inner sep=0pt,minimum size=3pt, opacity = 1] (D) at (0.5,1.15) {};
\node [draw,circle,color=black, fill=black,inner sep=0pt,minimum size=3pt, opacity = 1] (E) at (0.5,2.5) {};
\node [draw,circle,color=black, fill=white,inner sep=0pt,minimum size=3pt, opacity = 1] (F) at (0.5,0.5) {};
\draw [->] (A) -- (B);
\draw [->] (A) -- (D);
\draw [->] (C) -- (B);
\draw [->] (D) -- (C);

\draw [->] (-0.25,1.5) -- (A);
\draw [->] (B) -- (E) node[midway, right] {\!\scriptsize$ t/\sqrt{2}$};;
\draw [->] (E) -- +(0.75,0);
\draw [->] (-0.25, 2.5) -- (E);
\draw [->] (C) -- +(0.4,0);
\draw [->] (F) -- (D) node[midway, right] {\!\scriptsize$ t\sqrt{2}$};;
\draw [->] (F) -- +(0.75,0);
\draw [->] (-0.25,0.5) -- (F);

\node (B) at (-1.5,2.8) {$\psi_{k}:$};
\end{tikzpicture}
\qquad
\begin{tikzpicture}[scale = 1.25]
\centering
\node () at (0.5,0.25){ \footnotesize $$};
\node () at (0.5,2.25){ \footnotesize $$};
\node () at (0.5,1.25){ \footnotesize $$};

\draw [dashed] (-0.25,3) -- (1.25,3);
\draw [dashed] (1.25,0) -- (1.25,3) node[pos = 5/6, right] {\footnotesize$k+2$} node[pos = 1/6, right] {\footnotesize$k$} node[pos = 1/2, right] {\footnotesize$k+1$}; 
\draw [ dashed] (-0.25,0) -- (1.25,0) ;
\draw [ dashed]  (-0.25,0) -- (-0.25,3)  node[pos = 5/6, left] {\footnotesize$k+2$} node[pos = 1/6, left] {\footnotesize$k$} node[pos = 1/2, left] {\footnotesize$k+1$} ; 
\node [draw,circle,color=black, fill=white,inner sep=0pt,minimum size=3pt, opacity = 1] (A) at (0.15,1.5) {};
\node [draw,circle,color=black, fill=black,inner sep=0pt,minimum size=3pt, opacity = 1] (B) at (0.5,1.85) {};
\node [draw,circle,color=black, fill=white,inner sep=0pt,minimum size=3pt, opacity = 1] (C) at (0.85,1.5) {};
\node [draw,circle,color=black, fill=black,inner sep=0pt,minimum size=3pt, opacity = 1] (D) at (0.5,1.15) {};
\node [draw,circle,color=black, fill=white,inner sep=0pt,minimum size=3pt, opacity = 1] (E) at (0.5,2.5) {};
\node [draw,circle,color=black, fill=black,inner sep=0pt,minimum size=3pt, opacity = 1] (F) at (0.5,0.5) {};
\draw [->] (A) -- (B);
\draw [->] (A) -- (D);
\draw [->] (B) -- (C);
\draw [->] (C) -- (D);

\draw [->] (-0.25,1.5) -- (A);
\draw [->] (E) -- (B) node[midway, right] {\!\scriptsize$ t/\sqrt{2}$};;
\draw [->] (E) -- +(0.75,0);
\draw [->] (-0.25, 2.5) -- (E);
\draw [->] (C) -- +(0.4,0);
\draw [->] (D) -- (F) node[midway, right] {\!\scriptsize$ t\sqrt{2}$};;
\draw [->] (F) -- +(0.75,0);
\draw [->] (-0.25,0.5) -- (F);

\node (B) at (-1.5,2.8) {$\psi_{-k}:$};
\end{tikzpicture}
\caption{Maps $\psi_{\pm k} \colon \C^* \to \F^{ms}(\Gamma_{\pm k}) $ corresponding to the short root in type $B_k$.}\label{fig:graphrootsB2}
\end{figure}
The maps $\psi_{\pm k} \colon \C^* \to \F^{ms}(\Gamma_{\pm k}) $ corresponding to the short root are defined as shown in Figure \ref{fig:graphrootsB2}. The space $ \F^{ms}(\Gamma_\word) $ consists of two components $ \F^{ms}_\pm(\Gamma_\word) $ distinguished by the sign of the product of coordinates. The double Bruhat cell $G^{u,v}$ also consists of two components $G^{u,v}_\pm$ distinguished by the sign of the determinant. The same argument as in type $C$ shows that the map $\mes \colon  \F^{ms}_+(\Gamma_\word) \to G^{u,v}_+$ is birational, and so is the map $\mes \colon  \F^{ms}_-(\Gamma_\word) \to G^{u,v}_-$.\qedhere 

\end{proof}
\paragraph{Face weights as cluster coordinates.} Consider a double Bruhat cell $G^{u,v}$ in a \emph{simple} complex Lie group $G$. To every reduced double word $\word$ representing $(u,v)$, Fock and Goncharov associate a birational chart on the double Bruhat cell  $G^{u,v}$ \cite{fock2006cluster}. Those charts define a \emph{cluster structure} on $G^{u,v}$ in the sense that transition maps between charts corresponding to different double word representations of $(u,v)$ are $Y$-type cluster transformations. In this section, we will show that Fock-Goncharov coordinates associated with a double word $\word$ coincide with face weights of the corresponding move-symmetric plabic graph $\Gamma_\word$.

The construction of the chart associated with a double word $\word = i_1\dots i_m$ is as follows. 
Let $E_i, F_i$, $i \in [1,k]$, be Chevalley generators of the Lie algebra $\g$ of $G$. Let also $H_{i}$, $i \in [1,k]$, be the basis in the corresponding Cartan subalgebra $\mathfrak t \subset \g$ dual to the basis $\alpha_i \in \mathfrak t^* $ of simple roots. These $H_i$ are known as \emph{fundamental coweights}. They give rise to \emph{cocharacters} $Y_i(t) := \exp(H^i\log t)$, $i \in [1,k]$.
Let also $X_i :=   \exp(  E_i)$, $X_{-i}:=   \exp( F_i)$; in terms of \eqref{eq:x}, we have $X_i := X_i(1)$ for all $i \in [1,k] \cup [-k,-1]$. 
Define a map $(\C^*)^{m+k} \to G$ as follows. For any $i \in [1,k]$, let $m_i$ be the total number of occurrences of $\pm i$ in the word~$\word$. Consider the expression $X_{i_1}\cdots X_{i_m}$. Insert in this product $m_1 + 1$ instances of $Y_1$,  $m_2 + 1$ instances of $Y_2$, etc., according to the following rule: between any two instances of the same symbol $Y_i$ there is either exactly one $X_i$, or exactly one $X_{-i}$, but not both. Each of the $Y$ symbols depends on its own variable $t_j$, so that there are $\sum{(m_i + 1)} = m + k$ variables altogether. The so-obtained product is independent on the exact positions of $Y$'s and gives a birational chart on $G^{u,v}$.

At the same time, one can construct a birational chart on  $G^{u,v}$ associated with a double word $\word$ as follows. Denote by $\tilde \F^{ms}(\Gamma_\word)$ the space of projective move-symmetric weightings on $\Gamma_\word$, i.e., weightings with weights assigned to all faces except the upper- and lowermost one. Proposition \ref{dbc} implies that the boundary measurement mapping on $\tilde \F^{ms}(\Gamma_\word)$ is birational onto $G^{u,v}$ where $G = \PSp_{2k}$ or $\SO_{2k+1}$ is the adjoint group of $G(\Omega_n)$. Therefore, on $G^{u,v}$, one has birational coordinates given by face weights. More precisely, the coordinates in the $C$ case are the weights of all faces above or on the midline except the uppermost one, while in the $B$ case one takes geometric means of the form $\sqrt{y_iy_{i'}}$ where $i$ is any face above the midline but not the uppermost one, cf. Section~\ref{sec:pbc}.

\begin{proposition}\label{prop:fg}
The boundary measurement map $\tilde \F^{ms}(\Gamma_\word) \to G^{u,v}$ identifies the so-defined face weight coordinates and Fock-Goncharov coordinates associated with the double word $\word$. 
\end{proposition}
\begin{proof}
\begin{figure}[t]
\centering
\begin{tikzpicture}[scale = 1]
\centering
\node () at (0.5,0.75){ \footnotesize $t$};
\node () at (0.5,2.25){ \footnotesize $t$};

\draw [dashed] (0,3) -- (1,3);
\draw [dashed] (1,0) -- (1,3) node[pos = 1/6, right] {\footnotesize$i$} node[pos = 1/3, right] {\footnotesize$i+1$} node[pos = 5/6, right] {\footnotesize$n+1 - i$} node[pos = 2/3, right] {\footnotesize$n-i$};;; ;
\draw [ dashed] (0,0) -- (1,0) ;
\draw [ dashed]  (0,0) -- (0,3) node[pos = 1/6, left] {\footnotesize$i$} node[pos = 1/3, left] {\footnotesize$i+1$} node[pos = 5/6, left] {\footnotesize$n+1 - i$} node[pos = 2/3, left] {\footnotesize$n-i$};;; ;;
\draw (0,2.5) -- (1,2.5);
\draw (0,2)  -- (1,2);
\draw (0,1)  -- (1,1);
\draw (0,0.5)  -- (1,0.5);
\end{tikzpicture}
\caption{Plabic graphs representing fundamental cocharacters in type $B_k$. Here $n = 2k+1$ is the dimension of the defining representation.}\label{fig:graphfund}
\end{figure}
\begin{figure}
\centering
\begin{tikzpicture}[xscale = 0.5, yscale = -0.5]

\draw [dashed](-1,3) -- (3,3);
\draw [dashed](-1,-3) -- (3,-3);
\draw [dashed] (3,-3) -- (3,3) node[pos = 5/6, right] {\footnotesize$k$} node[pos = 1/6, right] {\footnotesize$k+2$} node[pos = 1/2, right] {\footnotesize$k+1$}; 
\draw [ dashed]  (-1,-3) -- (-1,3)  node[pos = 5/6, left] {\footnotesize$k$} node[pos = 1/6, left] {\footnotesize$k+2$} node[pos = 1/2, left] {\footnotesize$k+1$} ; 
\node [draw,circle,color=black, fill=white,inner sep=0pt,minimum size=3pt, opacity = 1] (A) at (0.3,0) {};
\node [draw,circle,color=black, fill=black,inner sep=0pt,minimum size=3pt, opacity = 1] (B) at (1,0.7) {};
\node [draw,circle,color=black, fill=white,inner sep=0pt,minimum size=3pt, opacity = 1] (C) at (1.7,0) {};
\node [draw,circle,color=black, fill=black,inner sep=0pt,minimum size=3pt, opacity = 1] (D) at (1,-0.7) {};
\node [draw,circle,color=black, fill=white,inner sep=0pt,minimum size=3pt, opacity = 1] (E) at (1,2) {};
\node [draw,circle,color=black, fill=black,inner sep=0pt,minimum size=3pt, opacity = 1] (F) at (1,-2) {};
\node () at (0,1) {\footnotesize$\footnotesize\sqrt{2}$};
\node () at (0,-1) {\footnotesize$1/\sqrt{2}$};
\node () at (2,1) {\footnotesize$1/\sqrt{2}$};
\node () at (2,-1) {\footnotesize$\sqrt{2}$};
\node () at (1,2.5) {\footnotesize$$};
\node () at (1,-2.5) {\footnotesize$$};
\draw [] (A) -- (B) -- (C) -- (D) -- (A);
\draw (A) -- +(-1.3,0);
\draw (B) -- (E);
\draw (E) -- +(2,0);
\draw (E) -- +(-2,0);
\draw (C) -- +(1.3,0);
\draw (D) -- (F);
\draw (F) -- +(2,0);
\draw (F) -- +(-2,0);

\end{tikzpicture}
\qquad\qquad
\begin{tikzpicture}[xscale = 0.5, yscale = 0.5]

\draw [dashed](-1,3) -- (3,3);
\draw [dashed](-1,-3) -- (3,-3);
\draw [dashed] (3,-3) -- (3,3) node[pos = 5/6, right] {\footnotesize$k+2$} node[pos = 1/6, right] {\footnotesize$k$} node[pos = 1/2, right] {\footnotesize$k+1$}; 
\draw [ dashed]  (-1,-3) -- (-1,3)  node[pos = 5/6, left] {\footnotesize$k+2$} node[pos = 1/6, left] {\footnotesize$k$} node[pos = 1/2, left] {\footnotesize$k+1$} ; 
\node [draw,circle,color=black, fill=white,inner sep=0pt,minimum size=3pt, opacity = 1] (A) at (0.3,0) {};
\node [draw,circle,color=black, fill=black,inner sep=0pt,minimum size=3pt, opacity = 1] (B) at (1,0.7) {};
\node [draw,circle,color=black, fill=white,inner sep=0pt,minimum size=3pt, opacity = 1] (C) at (1.7,0) {};
\node [draw,circle,color=black, fill=black,inner sep=0pt,minimum size=3pt, opacity = 1] (D) at (1,-0.7) {};
\node [draw,circle,color=black, fill=white,inner sep=0pt,minimum size=3pt, opacity = 1] (E) at (1,2) {};
\node [draw,circle,color=black, fill=black,inner sep=0pt,minimum size=3pt, opacity = 1] (F) at (1,-2) {};
\node () at (0,1) {\footnotesize$\sqrt{2}$};
\node () at (0,-1) {\footnotesize$1/\sqrt{2}$};
\node () at (2,1) {\footnotesize$1/\sqrt{2}$};
\node () at (2,-1) {\footnotesize$\sqrt{2}$};
\node () at (1,2.5) {\footnotesize$$};
\node () at (1,-2.5) {\footnotesize$$};
\draw [] (A) -- (B) -- (C) -- (D) -- (A);
\draw (A) -- +(-1.3,0);
\draw (B) -- (E);
\draw (E) -- +(2,0);
\draw (E) -- +(-2,0);
\draw (C) -- +(1.3,0);
\draw (D) -- (F);
\draw (F) -- +(2,0);
\draw (F) -- +(-2,0);

\end{tikzpicture}
\caption{Graphs associated to $X_k := \exp(E_k)$ and $X_{-k} := \exp(F_k)$, where $E_k, F_k$ are Chevalley generators corresponding to the short root in type $B_k$.}\label{fig:fgprop}
\end{figure}

In type $B$ case, the fundamental coweights are
$
H_i = \sum_{j = 1}^i (E_{j,j} - E_{2k + 2 - j, 2k + 2 - j}),  i \in [1,k].
$
Figure~\ref{fig:graphfund} shows plabic graphs corresponding to the associated cocharacters $Y_i(t) = \exp(H^i\log t)$. Here all sources are directly connected to the corresponding sinks, and all face weights which are not explicitly shown are equal to $1$, except for the weights of the upper- and lowermost faces which we disregard. The graphs corresponding to $X_{\pm i}$ are as in Figure \ref{fig:graphroots} with all face weights set to $1$, except for the case $i = k$ when the graphs are as in Figure~\ref{fig:fgprop} (as usual, all sources which are not explicitly shown are directly connected to the corresponding sinks, and all face weights which are not explicitly shown are equal to $1$). To find a weighting $\mathcal Y \in \tilde \F^{ms}(\Gamma_\word) $ associated with a given collection of Fock-Goncharov coordinates, we need to glue the graphs corresponding to $X_{\pm i}$ and $Y_i(t_j)$ in the same order as we multiply $X_{\pm i}$ and $Y_i(t_j)$. The unweighted graph obtained by such gluing is precisely $\Gamma_\word$, because the graphs corresponding to $X_{\pm i}$ are just weighted versions of $\Gamma_{\pm i}$, while the graphs corresponding to $Y_i$ are weighted versions of the identity element in the monoid of plabic graphs. Note that since between any two $Y_i$ there is either exactly one $X_i$ or $X_{-i}$, faces whose weights are given by different $t_j$ are never glued to each other. Furthermore, the geometric means $\sqrt{y_i y_{i'}}$ for the graphs corresponding to $X_{\pm i}$ are equal to $1$, so, upon gluing all the graphs together, we obtain a weighting with geometric means equal to $t_j$, as desired.

The argument in the $C$ case is analogous. In that case, the fundamental coweights are
$ H_i = \sum_{j = 1}^i (E_{j,j} - E_{2k + 1 - j, 2k + 1 - j}), i \in [1,k-1]$,  
$H_k = \frac{1}{2} \sum_{j = 1}^k (E_{j,j} - E_{2k + 1 - j, 2k + 1 - j}).$ 
\end{proof}

\begin{example}\label{fgex}

\begin{figure}[t]
\centering
\begin{tikzpicture}[scale = 1]
\centering
\node () at (0,0) {
\begin{tikzpicture}[xscale = 1, yscale = 2]
\node () at (0.5,0.75){ \footnotesize $t_1$};
\node () at (0.5,2.25){ \footnotesize $t_1$};
\node () at (0.5,1.25){ \footnotesize $$};
\node () at (0.5,1.75){ \footnotesize $$};

\draw [dashed] (0,2.75) -- (1,2.75);
\draw [dashed] (1,0.25) -- (1,2.75) ; 
\draw [ dashed] (0,0.25) -- (1,0.25) ;
\draw [ dashed]  (0,0.25) -- (0,2.75) ; 
\draw (0,2.5) -- (1,2.5);
\draw (0,2)  -- (1,2);
\draw (0,1)  -- (1,1);
\draw (0,1.5)  -- (1,1.5);
\draw (0,0.5)  -- (1,0.5);
\end{tikzpicture}
};

\node () at (1.5,0) {
\begin{tikzpicture}[xscale = 1, yscale = 2]
\centering
\node () at (0.5,0.25){ \footnotesize $$};
\node () at (0.5,2.25){ \footnotesize $$};
\node () at (0.5,1.25){ \footnotesize $$};

\draw [dashed] (0,2.75) -- (1,2.75);
\draw [dashed] (1,0.25) -- (1,2.75) ; 
\draw [ dashed] (0,0.25) -- (1,0.25) ;
\draw [ dashed]  (0,0.25) -- (0,2.75) ; 
\node [draw,circle,color=black, fill=white,inner sep=0pt,minimum size=3pt, opacity = 1] (A) at (0.5,2.5) {};
\node [draw,circle,color=black, fill=black,inner sep=0pt,minimum size=3pt, opacity = 1] (B) at (0.5,2) {};
\draw (A) -- (B);
\draw (0,2.5) -- (A) -- (1,2.5);
\draw (0,2) -- (B) -- (1,2);
\draw (0,1.5)  -- (1,1.5);
\node [draw,circle,color=black, fill=white,inner sep=0pt,minimum size=3pt, opacity = 1] (A) at (0.5,1) {};
\node [draw,circle,color=black, fill=black,inner sep=0pt,minimum size=3pt, opacity = 1] (B) at (0.5,0.5) {};
\draw (A) -- (B);
\draw (0,1) -- (A) -- (1,1);
\draw (0,0.5) -- (B) -- (1,0.5);
\end{tikzpicture}
};

\node () at (3,0) {
\begin{tikzpicture}[xscale = 1, yscale = 2]
\node () at (0.5,0.75){ \footnotesize $t_2$};
\node () at (0.5,2.25){ \footnotesize $t_2$};
\node () at (0.5,1.25){ \footnotesize $$};
\node () at (0.5,1.75){ \footnotesize $$};

\draw [dashed] (0,2.75) -- (1,2.75);
\draw [dashed] (1,0.25) -- (1,2.75) ; 
\draw [ dashed] (0,0.25) -- (1,0.25) ;
\draw [ dashed]  (0,0.25) -- (0,2.75) ; 
\draw (0,2.5) -- (1,2.5);
\draw (0,2)  -- (1,2);
\draw (0,1)  -- (1,1);
\draw (0,1.5)  -- (1,1.5);
\draw (0,0.5)  -- (1,0.5);
\end{tikzpicture}
};

\node () at (4.5,0) {
\begin{tikzpicture}[xscale = 1, yscale = 2]
\node () at (0.5,1.25){ \footnotesize $t_3$};
\node () at (0.5,1.75){ \footnotesize $t_3$};
\node () at (0.5,1.25){ \footnotesize $$};
\node () at (0.5,1.75){ \footnotesize $$};
\draw [dashed] (0,2.75) -- (1,2.75);
\draw [dashed] (1,0.25) -- (1,2.75) ; 
\draw [ dashed] (0,0.25) -- (1,0.25) ;
\draw [ dashed]  (0,0.25) -- (0,2.75) ; 
\draw (0,2.5) -- (1,2.5);
\draw (0,2)  -- (1,2);
\draw (0,1)  -- (1,1);
\draw (0,1.5)  -- (1,1.5);
\draw (0,0.5)  -- (1,0.5);
\end{tikzpicture}
};

\node () at (6.5,0) {
\begin{tikzpicture}[xscale = 0.5, yscale = -0.5]

\draw [dashed](-1,5) -- (3,5);
\draw [dashed](-1,-5) -- (3,-5);
\draw [dashed] (3,-5) -- (3,5); 
\draw [ dashed]  (-1,-5) -- (-1,5) ; 
\node [draw,circle,color=black, fill=white,inner sep=0pt,minimum size=3pt, opacity = 1] (A) at (0.3,0) {};
\node [draw,circle,color=black, fill=black,inner sep=0pt,minimum size=3pt, opacity = 1] (B) at (1,0.7) {};
\node [draw,circle,color=black, fill=white,inner sep=0pt,minimum size=3pt, opacity = 1] (C) at (1.7,0) {};
\node [draw,circle,color=black, fill=black,inner sep=0pt,minimum size=3pt, opacity = 1] (D) at (1,-0.7) {};
\node [draw,circle,color=black, fill=white,inner sep=0pt,minimum size=3pt, opacity = 1] (E) at (1,2) {};
\node [draw,circle,color=black, fill=black,inner sep=0pt,minimum size=3pt, opacity = 1] (F) at (1,-2) {};
\node () at (0,1) {\footnotesize$\sqrt{2}$};
\node () at (0,-1) {\footnotesize$1/\sqrt{2}$};
\node () at (2,1) {\footnotesize$1/\sqrt{2}$};
\node () at (2,-1) {\footnotesize$\sqrt{2}$};
\node () at (1,2.5) {\footnotesize$$};
\node () at (1,-2.5) {\footnotesize$$};
\draw [] (A) -- (B) -- (C) -- (D) -- (A);
\draw (A) -- +(-1.3,0);
\draw (B) -- (E);
\draw (E) -- +(2,0);
\draw (E) -- +(-2,0);
\draw (C) -- +(1.3,0);
\draw (D) -- (F);
\draw (F) -- +(2,0);
\draw (F) -- +(-2,0);
\draw (-1,4) -- (3,4);
\draw (-1,-4) -- (3,-4);

\end{tikzpicture}

};

\node () at (8.5,0) {
\begin{tikzpicture}[xscale = 1, yscale = 2]
\node () at (0.5,1.25){ \footnotesize $t_4$};
\node () at (0.5,1.75){ \footnotesize $t_4$};
\node () at (0.5,1.25){ \footnotesize $$};
\node () at (0.5,1.75){ \footnotesize $$};

\draw [dashed] (0,2.75) -- (1,2.75);
\draw [dashed] (1,0.25) -- (1,2.75) ; 
\draw [ dashed] (0,0.25) -- (1,0.25) ;
\draw [ dashed]  (0,0.25) -- (0,2.75) ; 
\draw (0,2.5) -- (1,2.5);
\draw (0,2)  -- (1,2);
\draw (0,1)  -- (1,1);
\draw (0,1.5)  -- (1,1.5);
\draw (0,0.5)  -- (1,0.5);
\end{tikzpicture}
};

\end{tikzpicture}
\caption{Identifying face weights with Fock-Goncharov coordinates. All omitted face weights are equal to $1$.}\label{fig:fgexample}
\end{figure}
Consider the double word $\word = (-1,2)$ in type $B_2$. The Fock-Goncharov parametrization $(\C^*)^4 \to \SO_5$ of the corresponding double Bruhat cell in $\SO_5$ is given by
$$
(t_1, t_2, t_3, t_4) \mapsto Y_1(t_1)X_{-1}Y_1(t_2)Y_2(t_3)X_2Y_2(t_4).
$$
The exact location of the $Y_i$ symbols is irrelevant as long as $X_{-1}$ is between the two instances of $Y_1$, and $X_2$ is between the two instances of $Y_2$. The corresponding plabic graph is obtained by concatenating graphs in Figure \ref{fig:fgexample}, in the order as shown. We see that the resulting graph has exactly four symmetric pairs of faces, not counting the uppermost, lowermost, and the face on the midline. Furthermore, the symmetrized face weights are precisely Fock-Goncharov parameters $t_1, t_2, t_3, t_4$.
\end{example}

\medskip

\section{Total positivity in types $B$ and $C$}\label{sec:tp}

\paragraph{The totally nonnegative part of $G(\Omega_n)$.} Let $G$ be a complex  reductive Lie group with Lie algebra $\g$ of rank $k$, $T \subset G$ be a maximal torus, and $E_i, F_i \in \g$, $i \in [1,k]$ be the corresponding Chevalley generators. Set $T^{>0} := \{ H \in T \mid \chi(H) \in \R_+ \,\forall\, \chi \in \mathrm{Hom}(T, \C^*)\}.$
The \emph{totally nonnegative part $G^{\geq 0}$} of $G$ \cite{Lus, fomin1999double} is defined as the submonoid generated by $T^{>0}$ and $X_i(t)$, $i \in [-k,-1] \cup [1,k]$, $t \in \R_+$, where, as above, $X_i(t) := \exp(t E_i), X_{-i}(t) := \exp(t F_i)$. The submonoid $G^{\geq 0}$ depends on the choice of a root decomposition and Chevalley generators, but different choices lead to submonoids which are conjugate in~$G$.

Now, recall that the group  $G(\Omega_n)$ consists of linear operators $A \in GL_n$ preserving an antidiagonal bilinear form $\Omega_n = \sum_{i=1}^n (-1)^{i+1}E_{i, n+1-i}$. It is the symplectic group $\Sp_n$ if $n$ is even and orthogonal group $O_n$ if $n$ is odd. Take the Chevalley basis in the corresponding Lie algebra given by \eqref{cb1} if $n$ is even or \eqref{cb2} if $n$ is odd. The aim of this section is to show that the corresponding submonoid $G(\Omega_n)^{\geq 0}$ is precisely the intersection of $G(\Omega_n)$ with totally nonnegative matrices. One inclusion is straightforward:  the matrices generating $G(\Omega_n)^{\geq 0}$ are totally nonnegative, so every matrix in $G(\Omega_n)^{\geq 0}$ is totally nonnegative. Our goal is to prove the opposite inclusion: \begin{proposition}\label{prop:tp} If $A \in G(\Omega_n)$ is totally nonnegative as an element of $GL_n$, then  $A \in G(\Omega_n)^{\geq 0}$.\end{proposition}

A similar result appears in \cite{barkley2024two}, where it is shown that, for certain flag varieties in types $B$ and $C$, Lusztig's positivity coincides with positivity of Pl\"ucker coordinates. 

\begin{remark}
For any complex reductive Lie group $G$ one can also define the \emph{totally positive part} $G^{> 0}$, see~\cite{Lus}. It can be characterized as the intersection of $G^{\geq 0}$ with the open double Bruhat cell $G^{w_0,w_0}$, where $w_0$ is the longest element in the Weyl group of $G$ \cite[Proposition 2.13]{Lus}. In particular, $GL_{n}^{ >0} = GL_{n}^{\geq 0} \cap GL_n^{w_0, w_0}$.  At the same time, for $G = G(\Omega_n)$, one has $G^{w_0, w_0} = GL_n^{w_0, w_0} \cap G$, so for such $G$ we have
\begin{gather*}
G^{> 0} = G^{\geq 0} \cap G^{w_0,w_0} = G^{\geq 0} \cap (GL_n^{w_0, w_0} \cap G)\stackrel{(*)}{=} (GL_{n}^{\geq 0} \cap G ) \cap (GL_n^{w_0, w_0} \cap G) \\ = (GL_{n}^{\geq 0} \cap GL_n^{w_0, w_0}) \cap G = GL_{n}^{ >0}  \cap G,
\end{gather*}
where $(*)$ holds by Proposition \ref{prop:tp}. In other words, $G(\Omega_n)^{> 0}$ is precisely the set of totally positive matrices in $G(\Omega_n)$.
\end{remark}

\paragraph{The Weyl group of $G(\Omega_n)$. } We begin with discussing the structure of the Weyl group of $G(\Omega_n)$. It is embedded in the Weyl group $S_n$ of $GL_n$
as the centralizer $C_{S_{n}}(w_0)$ of the order-reversing permutation $w_0 \in S_{n}$. That centralizer is generated by involutions $s_i$, $i \in [1,k]$, $k = \lfloor n/2 \rfloor$, given by
\begin{gather}\label{eq:sgen}
\begin{gathered}
s_i := (i, i+1)(n - i, n - i + 1), \quad i \in [1,k-1],\\
s_k := \begin{cases}
 (k, k+1) \quad \mbox{if } n =2k,\\
   (k, k+2) \quad \mbox{if } n =2k + 1.
\end{cases}
\end{gathered}
\end{gather}
For any $w \in C_{S_{n}}(w_0)$, let $\mathrm{inv}(w)$ be the number of inversions, i.e., pairs $ i,j \in [1,n]$ such that $i < j$ but $w(i) > w(j)$, and let $\mathrm{neg}(w)$ be the number of $ i \in [1,n]$ such that $  i < \frac{1}{2}(n+1)$ and $w(i) > \frac{1}{2}(n+1) $ (one can identify the Weyl group $ C_{S_{n}}(w_0)$ of type $B=C$ with signed permutations; upon that identification, $\mathrm{neg}(w)$ becomes the number of $i > 0$ such that $w(i) < 0$, hence the notation). 

The following result is well-known (see e.g. \cite[Proposition 2.3]{brenti2017odd}), although we were not able to find the form we need in the existing literature. 
\begin{proposition}\label{prop:length}
 For any $w \in C_{S_{n}}(w_0)$ any reduced (i.e. minimal length) factorization of $w$ in terms of the generators $s_i$ consists of
$
\frac{1}{2}(\mathrm{inv}(w) +(-1)^n \mathrm{neg}(w))
$
factors, precisely $\mathrm{neg}(w)$ of which are $s_k$.
\end{proposition}
\begin{proof}
\begin{figure}
\centering
\begin{tikzpicture}[yscale = 0.4, xscale = 0.6]
\centering
\node  (1) at (0,1.5) {$1$};
\node  (2) at (0,0.5) {$2$};
\node  (3) at (0,-0.5) {$3$};
\node  (4) at (0,-1.5) {$4$};
\node  (1') at (4,-0.5) {$1$};
\node  (2') at (4,-1.5) {$2$};
\node  (3') at (4,1.5) {$3$};
\node  (4') at (4,0.5) {$4$};
\draw (2) -- (2');
\draw (1) -- (1');
\draw (3) -- (3');
\draw (4) -- (4');
\node () at (10,0) {$(13)(24) = (23)\cdot (12)(34) \cdot (23)$};
\draw [dotted] (1.5,-2) -- (1.5,2);
\draw [dotted] (2.5,-2) -- (2.5,2);
\end{tikzpicture}
\caption{Factorizing an element in the centralizer of $w_0$.}\label{fig:fact}
\end{figure}


First, assume that $n$ is even, $ n =2k$. Let $w = \prod s_{i_j}$ be a factorization of  $w \in C_{S_{2k}}(w_0)$. Let also $p$ be the number of factors $s_k$, and $q$ be the number of other factors. Since both functions $\mathrm{inv}$ and $\mathrm{neg}$ satisfy $f(w_1w_2) \leq f(w_1) + f(w_2)$, we have
$$
\mathrm{inv}(w) \leq \sum \mathrm{inv}(s_{i_j}) = p + 2q,
$$
and
\begin{equation}\label{ineq2}
\mathrm{neg}(w) \leq \sum \mathrm{neg}(s_{i_j}) = p,
\end{equation}
so
\begin{equation}\label{ineq3}
\mathrm{inv}(w) + \mathrm{neg}(w) \leq 2(p+q).
\end{equation}
This means that any factorization  of $w$ in terms of the generators $s_i$ consists of at least
$
\frac{1}{2}(\mathrm{inv}(w) + \mathrm{neg}(w))
$
factors, and at least $\mathrm{neg}(w)$ of them are $s_k$. Furthermore, if \eqref{ineq3} is an equality, then so is \eqref{ineq2}. Therefore, to complete the proof, it suffices to find a factorization of $w$ of length exactly $
\frac{1}{2}(\mathrm{inv}(w) + \mathrm{neg}(w))$. Such a factorization can be constructed using the following standard technique. For each $i \in [1, 2k]$, connect the points $(0,i)$ and $(1, w(i))$ by a straight line interval. This gives an axially symmetric diagram as shown in Figure~\ref{fig:fact}. If needed, perturb the diagram (preserving the axial symmetry) by shifting the endpoints of intervals in the vertical direction, so that all intersections are pairwise, and there are at most two intersection points with the same horizontal coordinate. As a result, one gets a diagram which can be cut into axially symmetric subdiagrams corresponding to generators $s_i$, giving rise to a factorization $w = \prod s_{i_j}$. As before, let $p$ be the number of factors $s_k$, and $q$ be the number of other factors. Note that the diagram representing $s_k$ has one intersection, while diagrams representing other generators have two intersections. So, the total number of intersections in the diagram representing $w$ is $p + 2q$. On the other hand, an intersection happens each time we have an inversion, so
$$
 p + 2q = \mathrm{inv}(w).
$$
Further, $\mathrm{neg}(w)$ is equal to the number of lines in the diagram which start above the axis of symmetry and end below. These are in bijection with subdiagrams representing the generator $s_k$, so
$$
p = \mathrm{neg}(w).
$$
Thus, the length of the so-constructed factorization of $w$ is
$$
 p + q = \tfrac{1}{2}((p+2q)+ q) = \tfrac{1}{2}(\mathrm{inv}(w) + \mathrm{neg}(w)),
$$
ending the proof in the even case.

Now, if $n$ is odd, $n = 2k+1$, then any $w \in C_{S_{2k+1}}(w_0)$ fixes the value $k+1$, so relabeling $k+2 \mapsto k + 1, \dots, 2k+1 \mapsto 2k$ we get an isomorphism $C_{S_{2k+1}}(w_0) \to C_{S_{2k}}(w_0)$. This isomorphism sends generators to generators, preserves the number $\mathrm{neg}(w)$, and decreases the number of inversions by $2\mathrm{neg}(w)$. Thus, the odd case reduces to the even case.
\end{proof}
\begin{corollary}\label{cor:red}
Let  $w \in C_{S_{n}}(w_0)$. Then, for any reduced factorization $w = \prod s_{i_j}$, we have $\mathrm{inv}(w) = \sum \mathrm{inv}(s_{i_j})$.
\end{corollary}
\begin{proof}
By the previous result, the number of  factors $s_k$ is $\mathrm{neg}(w)$, while the number of other factors is $\frac{1}{2}(\mathrm{inv}(w) +(-1)^n \mathrm{neg}(w)) - \mathrm{neg}(w) $. Furthermore, we have $\mathrm{inv}(s_{i}) = 2$ if $i \neq k$ and $\mathrm{inv}(s_{k}) = 2 - (-1)^n$, so
\begin{gather*}
\sum \mathrm{inv}(s_{i_j}) = ( 2 - (-1)^n) \mathrm{neg}(w) + 2(\tfrac{1}{2}(\mathrm{inv}(w) +(-1)^n \mathrm{neg}(w)) - \mathrm{neg}(w) ) = \mathrm{inv}(w). \\[-1.2\normalbaselineskip]\mathstrut \qedhere
\end{gather*}
\end{proof}

This result means that any reduced factorization of $w$ in the $B=C$ type Weyl group $C_{S_n}(w_0)$ can also be viewed as a reduced factorization in the $A$ type Weyl group $S_n$, cf. \cite[Proposition 3.9]{barkley2024two}. To obtain the latter from the former, replace each $s_i$ with its reduced expression in terms  of transpositions $(i,i+1)$. Such an expression is unique (up to
permuting commuting transpositions) unless $n = 2k+1$ is odd and $i = k$, in which case we have two factorizations $(k, k+2) = (k,k+1)(k+1, k+2)(k, k+1) =(k+1, k+2)(k, k+1)(k+1, k+2). $

\paragraph{Proof of Proposition \ref{prop:tp}.}
Suppose $A \in G(\Omega_n)$ is a totally nonnegative matrix.  We aim to prove that  $A \in G(\Omega_n)^{\geq 0}$. Since $A$ is totally nonnegative, by \cite[Proposition 2.29]{fomin1999double}, it admits a Gaussian decomposition $A = LDU$, where $L$ is lower unitriangular, $D$ is diagonal, $U$ is upper unitriangular, and all three matrices $L$, $D$, $U$ are totally nonnegative.
Recall that the involution $\tau \colon GL_n \to GL_n$ is defined by $\tau(X) := \Omega_n X^{-t} \Omega_n^{-1}$, and the fixed point set of $\tau$ is precisely $G(\Omega_n)$. So, since $A \in G(\Omega_n)$, we have $A = \tau(A) = \tau(L)\tau(D)\tau(U),$ 
and since $\tau$ preserves the subgroups of upper- and lower-triangular matrices, by uniqueness of Gaussian decomposition we must have $L, D, U \in G(\Omega_n)$. Therefore, it suffices to show that a totally nonnegative upper-triangular matrix  $A \in G(\Omega_n)$ belongs to  $G(\Omega_n)^{\geq 0}$ (the lower-triangular case is reduced to the upper-triangular one by considering the transposed matrix).

Suppose $A \in G(\Omega_n)$ is upper-triangular and totally nonnegative. Since $A$ is upper-triangular, it belongs to a double Bruhat cell of the form $GL_n^{{id}, v}$. Further, observe that since $\tau$ preserves the subgroups of upper- and lower-triangular matrices, for any $u,v \in S_n$, one has $\tau(GL_n^{u,v}) = GL_n^{\bar u, \bar v},$ where $\bar w := \lw w \lw$. Therefore, since $A \in GL_n^{{id}, v}$ is fixed by $\tau$, we have $ \lw v \lw = v$, i.e., $v \in C_{S_n}(w_0)$.  

Recall that, for any reductive Lie algebra with Chevalley basis $E_i, F_i$, we define $X_i(t) := \exp(t E_i)$. In what follows, we denote $X_i(t)$ in type $A$ by $X^A_i(t)$. We have $X^A_i(t) = \Id + tE_{i, i +1}$. Likewise, we denote $X_i(t)$ in types $B$ and $C$ by  $X^{BC}_i(t)$. Using \eqref{cb1} and \eqref{cb2}, we find
$$
{X^{BC}_i(t)} =  \begin{cases} \exp(t(E_{i, i+1} + E_{n-i, n -i + 1})) \mbox{ if $i < k$},\\
\exp(tE_{k, k+1}) \mbox{ if $n = 2k$ and $i = k$}, \\
\exp(\sqrt{2} t (E_{k, k+1} + E_{k+1, k+2}))\mbox{ if $n = 2k+1$ and $i = k$}.
\end{cases}
$$

 Take a reduced factorization $v = \prod_{j=1}^{m} s_{i_j},$ where $s_{i_j}$ are as in \eqref{eq:sgen}. By Corollary~\ref{cor:red}, replacing each $s_{i_j}$ by its reduced factorization in $S_n$, we also get a reduced factorization of $v$ in $S_n$. So, by \cite[Theorem 1.3]{fomin1999double}, the matrix $A$ can be written in a unique way in the form
$
A = H \prod_{j=1}^{m} A_j
$
where $H \in GL_n$ is diagonal with entries in $\R_+$, and
\begin{subnumcases}{A_j = } X^A_{i_j}(t)X^A_{n - i_j}(t') \mbox{, where $t, t' \in \R_+$, if $i_j < k$}, \label{case1}\\
X^A_k(t) \mbox{, where $t \in \R_+$, if $n=2k$ and $i_j = k$}, \label{case2}\\
X^A_k(t)X^A_{k+1}(t')X^A_k(t'') \mbox{,  where $t, t', t'' \in \R_+$, if $n = 2k+1$ and $i_j = k$.}\label{case3}
\end{subnumcases}
Further, note that $\tau(X^A_i(t)) = X^A_{n-i}(t)$. So, if $A_j$ is of the form \eqref{case1}, we have 
$$
\tau(A_j) = X^A_{n-i_j}(t)X^A_{i_j}(t') = X^A_{i_j}(t')X^A_{n-i_j}(t) .
$$
Likewise, if $A_j$ is of the form \eqref{case2}, we have 
$$
\tau(A_j) = X^A_k(t).
$$
Finally, if $A_j$ is of the form \eqref{case3}, then 
\begin{gather*}
\tau(A_j) = X^A_{k+1}(t)X^A_{k}(t')X^A_{k+1}(t'')   =  X^A_k(\tfrac{t't''}{t+t''})X^A_{k+1} ( t+t'')X^A_k(\tfrac{tt'}{t + t''} ).
\end{gather*}
Finally, note that $\tau(H)$ is diagonal with entries in $\R_+$. So, the factorization
$
A = \tau(A) =  \tau(H) \prod_{j=1}^{m} \tau(A_j)
$
is of the same form as $A = H \prod_{j=1}^{m} A_j$. Therefore, by uniqueness we have $\tau(H) = H$ and $\tau(A_j) = A_j$ for every~$j$. The former, along with entries of $H$ being in $\R_+$, implies $H \in T^{>0} \subset G(\Omega_n)^{\geq 0}$. So, it suffices to check that $A_j \in G(\Omega_n)^{\geq 0}$ for all $j$. Indeed, if $A_j$ is of the form  \eqref{case1}, then $\tau(A_j) = A_j$ implies $t' = t$, so
 $$
 A_j =  X^A_{i_j}(t)X^A_{n - i_j}(t)  = X^{BC}_{i_j}(t),
 $$
which belongs to $G(\Omega_n)^{\geq 0}$ by definition of the latter.
 Similarly, since in the $n = 2k$ case we have $X_k^A(t) = X_k^{BC}(t)$, if $A_j$ is of the form  \eqref{case2}, then it again belongs to $G(\Omega_n)^{\geq 0}$ by definition.
Finally, if $A_j$ is of the form  \eqref{case3}, then $\tau(A_j) = A_j$ implies $  t'= 2t = 2t''$, so
$$ A_j = X^A_k(t)X^A_{k+1}(t / 2)X^A_k(t) =  X_k^{BC}(\sqrt{2}t) \in G(\Omega_n)^{\geq 0}.$$
Thus, we have $H \in  G(\Omega_n)^{\geq 0}$ and $ A_j \in G(\Omega_n)^{\geq 0}$ for all $j$. So, $ A = H \prod_{j=1}^{m} A_j \in G(\Omega_n)^{\geq 0}$, as desired. \qed

\medskip

\bibliographystyle{plain}
\bibliography{bc.bib}

\end{document}